\documentclass[reqno,10pt,centertags]{amsart} 
\usepackage{amsmath,amsthm,amscd,amssymb,latexsym,esint,upref,stmaryrd,
enumerate,verbatim,yfonts,bm,mathtools,comment}
\usepackage{color}
\usepackage{hyperref} 
\newcommand*{\mailto}[1]{\href{mailto:#1}{\nolinkurl{#1}}}

\newcommand{\bbC}{{\mathbb{C}}}

\newcommand{\bbN}{{\mathbb{N}}}

\newcommand{\bbR}{{\mathbb{R}}}

\newcommand{\cA}{{\mathcal A}}
\newcommand{\cB}{{\mathcal B}}

\newcommand{\cH}{{\mathcal H}}
\newcommand{\cI}{{\mathcal I}}

\newcommand{\cN}{{\mathcal N}}

\newcommand{\cP}{{\mathcal P}}

\newcommand{\cS}{{\mathcal S}}

\newcommand{\cZ}{{\mathcal Z}}

\newcommand{\beq}{\begin{equation}}
\newcommand{\enq}{\end{equation}}

\DeclareMathOperator{\rank}{rank}

\DeclareMathOperator{\dom}{dom}

\DeclareMathOperator{\tr}{tr}

\DeclareMathOperator{\diag}{diag}

\renewcommand{\Re}{\text{\rm Re}}
\renewcommand{\Im}{\text{\rm Im}}
\renewcommand{\ln}{\text{\rm ln}}

\newcommand{\loc}{\operatorname{loc}}

\newcommand{\AC}{\text{\textnormal{AC}}}
\newcommand{\SL}{\text{\textnormal{SL}}}

\newcommand{\no}{\notag}
\newcommand{\lb}{\label}
\newcommand{\f}{\frac}

\newcommand{\ol}{\overline}

\newcommand{\wti}{\widetilde}

\newcommand{\bi}{\bibitem}

\let\geq\geqslant
\let\leq\leqslant

\makeatletter
\def\theequation{\@arabic\c@equation}

\allowdisplaybreaks 
\numberwithin{equation}{section}

\newtheorem{theorem}{Theorem}[section]
\newtheorem{proposition}[theorem]{Proposition}
\newtheorem{lemma}[theorem]{Lemma}
\newtheorem{corollary}[theorem]{Corollary}
\newtheorem{definition}[theorem]{Definition}
\newtheorem{hypothesis}[theorem]{Hypothesis}

\theoremstyle{remark}
\newtheorem{remark}[theorem]{Remark}

\begin{document}

\title[Explicit Krein Resolvent Identities with Applications]{Explicit Krein Resolvent Identities for Singular Sturm--Liouville Operators with Applications to Bessel Operators}


\author[S.\ B.\ Allan]{S. Blake Allan}
\address{Department of Mathematics, 
Baylor University, One Bear Place \#81850,
Waco, TX 76798-7328, USA}
\email{\mailto{Blake\_Allan1@baylor.edu}}

\author[J.\ H.\ Kim]{Justin Hanbin Kim}
\address{Department of Mathematics, Vanderbilt University, PMB 353510, 2301 Vanderbilt Place, Nashville, TN 37235}
\email{\mailto{Han.bin.kim@vanderbilt.edu}}

\author[G.\ Michajlyszyn]{Gregory Michajlyszyn}
\address{Department of Mathematics, University of Rochester, 500 Joseph C. Wilson Blvd., \#271017, Rochester, NY 14627}
\email{\mailto{gmichajl@u.rochester.edu}}

\author[R. Nichols]{Roger Nichols}
\address{Department of Mathematics, The University of Tennessee at Chattanooga, 
415 EMCS Building, Dept. 6956, 615 McCallie Ave, Chattanooga, TN 37403, USA}
\email{\mailto{Roger-Nichols@utc.edu}}
\urladdr{\url{http://www.utc.edu/faculty/roger-nichols/index.php}}

\author[D.\ Rung]{Don Rung}
\address{Department of Mathematics, Sewanee: The University of the South, 735 University Ave., Sewanee, TN 37375}
\email{\mailto{rungdc0@sewanee.edu}}


\date{\today}


\begin{abstract} 
We derive explicit Krein resolvent identities for generally singular Sturm--Liouville operators in terms of boundary condition bases and the Lagrange bracket.  As an application of the resolvent identities obtained, we compute the trace of the resolvent difference of a pair of self-adjoint realizations of the Bessel expression $-d^2/dx^2+(\nu^2-(1/4))x^{-2}$ on $(0,\infty)$ for values of the parameter $\nu\in[0,1)$ and use the resulting trace formula to explicitly determine the spectral shift function for the pair.
\end{abstract}

\maketitle



\section{Introduction} \lb{s1}

In the classic theory of self-adjoint extensions of a densely defined symmetric operator $S$ with equal and finite deficiency indices, Krein's resolvent identity expresses the difference of the resolvent operators of any two self-adjoint extensions of $S$ in terms of its defect vectors (cf., e.g., \cite[Section 84]{AG81}, \cite[Appendix A]{CGNZ14}, and \cite[Lemma 2.30]{Te14}).  When $S$ is the closed minimal operator generated by a second-order Sturm--Liouville differential expression $\tau$ on an interval $(a,b)\subseteq \bbR$, its deficiency indices are at most equal to two (their precise common value depending upon the number of limit circle endpoints for $\tau$) so Krein's identity expresses the difference of the resolvent operators of two self-adjoint extensions of $S$ as an operator of rank at most equal to two.  Recently, the explicit form of Krein's identity was derived in \cite{CGNZ14} for all self-adjoint extensions in the case where $\tau$ is {\it regular} on $(a,b)$ in terms of the boundary values of the quasiderivatives of a distinguished basis of defect vectors, see \cite[eq.~(3.5)]{CGNZ14}.  This is made possible by the fact that functions in the domain of the maximal Sturm--Liouville operator $S^*$, and their quasiderivatives, possess boundary values at a regular endpoint.  In contrast, when $\tau$ is singular at an endpoint, neither functions in the domain of $S^*$ nor their quasiderivatives necessarily possess boundary values at the singular endpoint.  It is  for this reason that, in lieu of boundary values of the functions themselves, one typically uses the Wronskian (cf., e.g., \cite[Section 5]{EGNT13}), the Lagrange bracket and boundary condition bases/functions (cf., e.g., \cite[Section 6]{Ev05} and \cite[Definition 10.4.3]{Ze05}), or generalized boundary values (cf.~\cite{GLN19}) to parametrize the self-adjoint extensions of $S$.

In this paper, we derive the explicit form of Krein's resolvent identity for singular Sturm--Liouville operators on $(a,b)$ using boundary condition bases and the Lagrange bracket.  As a concrete application of the identities obtained, we consider the Bessel differential expression on $(0,\infty)$ indexed by the parameter $\nu\in [0,1)$.  The Bessel differential expression is singular at both endpoints of $(0,\infty)$ for $\nu\neq 1/2$, and its self-adjoint realizations form a one-parameter family.  Applying the general form of Krein's resolvent identity obtained in Section \ref{s3}, we explicitly compute the difference of the resolvent of the Friedrichs extension and that of any other self-adjoint realization of the Bessel expression.  Using the resulting identity, we then compute the trace of the difference of resolvents, which leads to an explicit expression for the spectral shift function of the pair.

We briefly summarize the contents of each of the remaining sections of this paper.  In Section \ref{s2}, we recall essential facts on self-adjoint extensions of three-term Sturm--Liouville operators on an interval $(a,b)\subseteq \bbR$.  Section \ref{s3} treats in detail the case of one limit circle endpoint.  Assuming that $a$ is the lone limit circle endpoint, we explicitly determine in Theorem \ref{t3.4} the form of Krein's resolvent identity for the difference of the resolvent of any self-adjoint extension of the minimal operator and the resolvent of a fixed reference self-adjoint extension in terms of a fixed boundary condition basis at $a$, the Lagrange bracket, and the Weyl--Titchmarsh solution at $b$.  The difference of resolvents is a rank one operator due to the presence of exactly one limit circle endpoint, so the Krein identity obtained immediately yields an explicit formula for the trace of the corresponding resolvent difference.  Analogously, Section \ref{s4} addresses the case where both endpoints $\{a,b\}$ are limit circle endpoints.  Treating separately the self-adjoint extensions parametrized by separated boundary conditions and those parametrized by coupled boundary conditions, we explicitly determine in Theorems \ref{t4.4}--\ref{t4.7} the form of Krein's resolvent identity for the difference of the resolvent of any self-adjoint extension of the minimal operator and the resolvent of a fixed reference self-adjoint extension in terms of fixed boundary condition bases at $a$ and $b$, the Lagrange bracket, and a distinguished pair of linearly independent solutions of the corresponding Sturm--Liouville differential equation.  The difference of resolvents is generally a rank two operator, but in certain special cases (cf.~Theorems \ref{t4.5} and \ref{t4.7}) the difference is rank one, owing to the fact that the two self-adjoint extensions also extend a symmetric operator which is itself a proper extension of the minimal operator.  At the end of Section \ref{s4}, we explain how the Krein resolvent identities obtained in \cite{CGNZ14} for regular Sturm--Liouville operators may be obtained as special cases of Theorems \ref{t4.4}--\ref{t4.7}.  In Section \ref{s5}, we consider, as an example, the Bessel differential expression (cf., e.g., \cite{AB15}, \cite{AB16}, \cite{BDG11}, \cite{DR18}, \cite{EK07}, \cite{GLN19}, \cite{KLP06}, and the references cited therein),
\begin{equation}
\tau_{\nu}=-\frac{d^2}{dx^2}+\frac{(\nu^2-(1/4))}{x^2},\quad x\in(0,\infty),\, \nu\in [0,1).
\end{equation}
The right endpoint $x=\infty$ is always a singular endpoint, and the left endpoint $x=0$ is a singular endpoint if $\nu\neq 1/2$, as it is regular if $\nu=1/2$.  For $\nu\in[0,1)$, $\tau_{\nu}$ is in the limit circle case at $x=0$ and in the limit point case at $x=\infty$, so the expression $\tau_{\nu}$ falls within the scope of the theory developed in Section \ref{s3}.  Applying the abstract identities developed in Section \ref{s3}, we determine the explicit form of Krein's identity in terms of the parameter $\nu$ and an explicit Weyl--Titchmarsh solution at $\infty$ and use this form to calculate the trace of the difference of the resolvent of the Friedrichs extension and that of any other self-adjoint extension in Propositions \ref{p5.1} and \ref{p5.6}.  The resulting trace formula is then used to determine the spectral shift function corresponding to the Friedrichs extension and any other self-adjoint extension in Propositions \ref{p5.3} and \ref{p5.8}.  As a byproduct, the explicit form of the spectral shift function for the pair yields a characterization of the nonnegative self-adjoint realizations of the Bessel expression and allows one to determine the single simple negative eigenvalue of any self-adjoint realization which is not nonnegative.  For completeness, the basic facts on the spectral shift function relevant to the analysis in Section \ref{s5} are collected in Appendix \ref{sA}.

Finally, we summarize some of the general notation used in this paper.  Let $\cH$ be a separable complex Hilbert space, $\langle\cdot,\cdot\rangle_{\cH}$ the inner product in $\cH$ (linear in the second argument), and $I_{\cH}$ the identity operator in $\cH$.  Next, let $T$ be a linear operator mapping (a subspace of) a Hilbert space into another, with $\dom(T)$ and $\ker(T)$ denoting the domain and kernel (i.e., null space) of $T$.  If $T$ is densely defined, then $T^*$ denotes the Hilbert space adjoint of $T$.  The resolvent set, spectrum, essential spectrum, absolutely continuous spectrum, and point spectrum of a closed linear operator in $\cH$ will be denoted by $\rho(\cdot)$, $\sigma(\cdot)$, $\sigma_{\rm ess}(\cdot)$, $\sigma_{\rm ac}(\cdot)$, and $\sigma_{\rm p}(\cdot)$, respectively.  The $\ell^p$-based trace ideals over $\cH$ will be denoted by $\cB_p (\cH)$, $p\in[1,\infty)$, and $\tr_{\cH}$ denotes the trace functional on $\cB_1(\cH)$.

Throughout, $\chi_S$ denotes the characteristic function of a set $S\subseteq \bbR$.  If $z\in \bbC$, then $\overline{z}$ denotes the complex conjugate of $z$.  To avoid cumbersome notation, for $(a,b)\subseteq \bbR$, $I_{(a,b)}$ and $\langle \,\cdot\,,\,\cdot\,\rangle_{(a,b)}$ denote the identity operator and inner product in the weighted Hilbert space $L^2((a,b);r(x)\, dx)$, respectively, and $\tr_{(a,b)}$ denotes the trace functional on $\cB_1\big(L^2((a,b);r(x)\, dx)\big)$.  In addition, $\SL_2(\bbR)$ denotes the set of all $R\in \bbR^{2\times2}$ with $\det(R)=1$, $I_{2\times2}$ denotes the $2\times 2$ identity matrix in $\bbR^{2\times 2}$, $\AC_{\loc}(a,b)$ denotes the set of locally absolutely continuous complex-valued functions on $(a,b)$, and if $M\subseteq\bbR$ is Lebesgue measurable, then $|M|$ denotes the Lebesgue measure of $M$.  We employ the following convention throughout:  ``$0<\varepsilon\ll1$'' means ``for all $\varepsilon\in (0,\varepsilon_0)$ for some $\varepsilon_0\in (0,\infty)$.''

\section{Self-Adjoint Extensions of Singular Sturm--Liouville Operators} \lb{s2}

In this preparatory section, we recall some of the essential facts on self-adjoint extensions of Sturm--Liouville operators, with particular emphasis on the singular case.  The primary motivation for recalling these facts here is to set up much of the notation and conventions to be employed in later sections.  As such, in most cases we only provide statements of the pertinent facts and defer to references for their proofs.  We begin by introducing the following hypothesis, which is assumed throughout this section.

\begin{hypothesis}\lb{h2.1}
Let $-\infty\leq a < b \leq \infty$ be fixed and suppose that $p$, $q$, and $r$ are real-valued and Lebesgue measurable on $(a,b)$ with $p>0$, $r>0$ a.e.~on $(a,b)$ and
\begin{equation}\lb{2.1}
p^{-1},\, q,\, r\in L^1_{\loc}((a,b);dx).
\end{equation}
\end{hypothesis}

Assuming Hypothesis \ref{h2.1}, we define 
\begin{equation}\lb{2.2}
\mathfrak{D}(a,b) = \{f\in \AC_{\loc}(a,b)\,|\, pf'\in \AC_{\loc}(a,b)\},
\end{equation}
and introduce the differential expression $\tau$ by
\begin{equation}\lb{2.3}
\tau f = \frac{1}{r}\big[-(pf')'+qf\big],\quad f\in \mathfrak{D}(a,b),
\end{equation}
where the prime denotes differentiation with respect to the independent variable.  In addition, we define the Lagrange bracket of a pair of functions $f,g\in \mathfrak{D}(a,b)$ by
\begin{equation}\lb{2.4}
[f,g](x) = f(x)\overline{(pg')(x)} - (pf')(x)\overline{g(x)},\quad x\in (a,b).
\end{equation}

The next result is a Pl\"ucker-type identity.  It relates the Lagrange brackets of pairs of functions in $\mathfrak{D}(a,b)$.

\begin{lemma}[{\cite[Lemma 2.5]{EGNT13}}]\lb{l2.2}
Assume Hypothesis \ref{h2.1}.  If $f_1, f_2, f_3, f_4\in\mathfrak{D}(a,b)$, then
\begin{equation}\lb{2.5}
\begin{split}
[f_1,f_2](x)[\ol{f_3},f_4](x) + [f_1,f_3](x)[\ol{f_4},f_2](x) + [f_1,f_4](x)[\ol{f_2},f_3](x)=0,&\\
x\in(a,b).&
\end{split}
\end{equation}
\end{lemma}

Next, we recall the identities of Lagrange and Green, which relate the Lagrange bracket to the differential expression $\tau$ (cf., e.g., \cite[eq.~(2.6) \&  Lemma 2.3]{EGNT13}).

\begin{lemma}[Lagrange's identity \& Green's formula]\lb{l2.3}
If $f,g\in \mathfrak{D}(a,b)$, then
\begin{equation}\lb{2.6}
\overline{g}\tau f - f\overline{\tau g} = \frac{1}{r}\frac{d}{dx}[f,g],
\end{equation}
and, consequently, for any $\alpha,\beta\in (a,b)$,
\begin{equation}\lb{2.7}
\int_{\alpha}^{\beta} \Big(\overline{g(x)}(\tau f)(x) - f(x) \overline{(\tau g)(x)}\Big) r(x)\, dx = [f,g](\beta) - [f,g](\alpha).
\end{equation}
\end{lemma}

Following \cite[Section 3]{EGNT13}, we now introduce the maximal and minimal operators associated to $\tau$.  The maximal operator associated to $\tau$ is denoted $T_{\max}$ and is defined by
\begin{align}
T_{\max}f=\tau f,\quad f\in \dom(T_{\max}) = \{g\in \mathfrak{D}(a,b)\,|\, g,\tau g\in L^2((a,b);r(x)\, dx)\}.\lb{2.8}
\end{align}
The operator $T_{\max}$ is densely defined, and its adjoint is the (closed) minimal operator, $T_{\min}$:
\begin{equation}\lb{2.9}
T_{\min} := (T_{\max})^*.
\end{equation}
In turn, the minimal operator is densely defined and its adjoint is the maximal operator:
\begin{equation}\lb{2.10}
(T_{\min})^* = T_{\max}.
\end{equation}

\begin{definition}\lb{d2.4}
A measurable function $f:(a,b)\to \bbC$ lies in $L^2((a,b);r(x)\, dx)$ near $a$ $($resp., $b$$)$ if $\chi_{(a,c)}f\in L^2((a,b);r(x)\, dx)$ $($resp., $\chi_{(c,b)}f\in L^2((a,b);r(x)\, dx)$$)$ for each $c\in(a,b)$. 
 Furthermore, $g\in\mathfrak{D}(a,b)$ lies in $\dom(T_{\max})$ near $a$ $($resp., $b$$)$ if $g$ and $\tau g$ both lie in $L^2((a,b);r(x)\, dx)$ near $a$ $($resp., $b$$)$.
\end{definition}

One verifies that $g\in \mathfrak{D}(a,b)$ lies in $\dom(T_{\max})$ near $a$ (resp., $b$) if and only if $\ol{g}$ lies in $\dom(T_{\max})$ near $a$ (resp., $b$).  Moreover, as a consequence of Green's formula, the Lagrange bracket of a pair of functions that lie in $\dom(T_{\max})$ near an endpoint has a finite limiting value at that endpoint.

\begin{lemma}[{\cite[Lemma 3.2]{EGNT13}}]\lb{l2.5}
If $f$ and $g$ lie in $\dom(T_{\max})$ near $a$ $($resp., near $b$$)$, then the limit
\begin{equation}\lb{2.11}
[f,g](a):=\lim_{x\to a^+}[f,g](x)\quad \text{$($resp., $[f,g](b):=\lim_{x\to b^-}[f,g](x)$$)$}
\end{equation}
exists and is finite.
\end{lemma}

The minimal operator may be characterized directly using the Lagrange bracket as follows:
\begin{align}
&T_{\min}f=\tau f,\lb{2.12}\\
&f\in \dom(T_{\min})=\{g\in \dom(T_{\max})\,|\,[g,h](a)=[g,h](b)=0, \, h\in \dom(T_{\max})\}.\no
\end{align}
It then follows that the minimal operator $T_{\min}$ is a densely defined, closed, symmetric operator in the Hilbert space $L^2((a,b);r(x)\, dx)$.

Recall that if $A$ and $B$ are two linear operators in a Hilbert space $\cH$, then $B$ is said to be an {\it extension} of $A$ (equivalently, $A$ is a {\it restriction} of $B$), denoted $A\subseteq B$, if and only if $\dom(A)\subseteq \dom(B)$ and $Au=Bu$ for all $u\in \dom(A)$.  For the remainder of this section, we will be interested in self-adjoint extensions of the minimal operator $T_{\min}$.  That $T_{\min}$ actually possesses self-adjoint extensions is a consequence of von Neumann's theory of self-adjoint extensions and Weyl's limit point/limit circle classification of endpoints.

Assuming Hypothesis \ref{h2.1}, one can consider for any $z\in \bbC$ the differential equation $\tau y = z y$ on the interval $(a,b)$, that is
\begin{equation}\lb{2.13}
-(py')'+qy = z r y\quad \text{on}\quad (a,b).
\end{equation}
A function $y\in \mathfrak{D}(a,b)$ is said to be a {\it solution} to \eqref{2.13} if $y$ satisfies \eqref{2.13} pointwise a.e.~on $(a,b)$.

\begin{definition}\lb{d2.6}
The differential expression $\tau$ is in the limit circle case at $a$ $($resp., $b$$)$ if for each $z\in \bbC$ all solutions to \eqref{2.13} lie in $L^2((a,b);r(x)\, dx)$ near $a$ $($resp., $b$$)$.  The differential expression $\tau$ is in the limit point case at $a$ $($resp., $b$$)$ if for each $z\in \bbC$, there is some solution to \eqref{2.13} which does not lie in $L^2((a,b);r(x)\, dx)$ near $a$ $($resp., $b$$)$.
\end{definition}

Weyl's alternative states that the classification of an endpoint as limit point or limit circle exhausts all possibilities; that is, $\tau$ is in one of these cases (limit point or limit circle) at each endpoint of $(a,b)$ (cf., e.g., \cite[Lemma 4.1]{EGNT13}).
\begin{theorem}[Weyl's Alternative]\lb{t2.7}
If there exists a $z_0\in \bbC$ such that every solution of $\tau y = z_0 y$ lies in $L^2((a,b);r(x)\, dx)$ near $a$ $($resp., $b$$)$, then $\tau$ is in the limit circle case at $a$ $($resp., $b$$)$.
\end{theorem}
If $z\in \bbC$ and $\tau$ is in the limit point case at an endpoint $c\in \{a,b\}$, then there is at least one solution to \eqref{2.13} which does not lie in $L^2((a,b);r(x)\, dx)$ near $c$.  It is entirely natural to ask whether there is {\it any} nontrivial solution to \eqref{2.13} which lies in  $L^2((a,b);r(x)\, dx)$ near $c$.  A nontrivial solution which  lies in  $L^2((a,b);r(x)\, dx)$ near $c$ is guaranteed to exist if $z$ is a point of {\it regular type} of $T_{\min}$.

\begin{definition}\lb{d2.8}
A point $z\in \bbC$ is a point of regular type of $T_{\min}$ if $T_{\min}-zI_{(a,b)}$ is an injection and $(T_{\min}-zI_{(a,b)})^{-1}$ is bounded.  The set of all points of regular type of $T_{\min}$ is denoted by $\mathrm{r}(T_{\min})$.
\end{definition}

\begin{lemma}[{\cite[Theorem 4.2 \& Corollary 4.3]{EGNT13}}]\lb{l2.9}
Let $c\in \{a,b\}$.  If $z\in \mathrm{r}(T_{\min})$, then there is a nontrivial solution of $\tau u = zu$ which lies in $L^2((a,b);r(x)\, dx)$ near $c$.  Moreover, this solution is unique up to constant multiples if $\tau$ is in the limit point case at $c$.
\end{lemma}

The limit point/limit circle classification of endpoints may be characterized in terms of the Lagrange bracket and functions in $\dom(T_{\max})$.

\begin{lemma}[{\cite[Lemma 4.4]{EGNT13}}]\lb{l2.10}
Assume Hypothesis \ref{h2.1}.  If $c\in \{a,b\}$, then $\tau$ is in the limit point case at $c$ if and only if
\begin{equation}\lb{2.14}
[f,g](c)=0,\quad f,g\in \dom(T_{\max}),
\end{equation}
and $\tau$ is in the limit circle case at $c$ if and only if there exists $f\in \dom(T_{\max})$ such that
\begin{equation}\lb{2.15}
\text{$[f,f](c)=0$ and $[f,g](c)\neq0$ for some $g\in \dom(T_{\max})$}.
\end{equation}
\end{lemma}

The significance of Weyl's limit point/limit circle classification is that it provides a means for completely characterizing the deficiency indices of $T_{\min}$.  Recall that if $S$ is a densely defined symmetric operator in a Hilbert space $\cH$, then the {\it deficiency indices} of $S$ are defined by
\begin{equation}\lb{2.16}
d_{\pm}(S) := \dim(\ker(S^*\mp i I_{\cH})).
\end{equation}
By von Neumann's theory of self-adjoint extensions (cf., e.g., \cite[Section X.1]{RS75}), $S$ possesses self-adjoint extensions if and only if $d_+(S)=d_-(S)$.  In the case of $T_{\min}$, the deficiency indices are always equal and they assume one of only three possible values, depending upon the number of limit circle endpoints, as the following theorem shows.

\begin{theorem}[{\cite[Theorem 4.6]{EGNT13}}]\lb{t2.11}
If Hypothesis \ref{h2.1} is satisfied, then $d_+(T_{\min})=d_-(T_{\min})$ and
\begin{equation}\lb{2.17}
d_{\pm}(T_{\min})=
\begin{cases}
0, & \text{if $\tau$ is limit circle at no endpoint of $(a,b)$},\\
1, & \text{if $\tau$ is limit circle at exactly one endpoint of $(a,b)$},\\
2, & \text{if $\tau$ is limit circle at both endpoints of $(a,b)$}.
\end{cases}
\end{equation}
In particular, $T_{\min}$ possesses self-adjoint extensions.
\end{theorem}

By Theorem \ref{t2.11}, the minimal operator $T_{\min}$ has self-adjoint extensions.  If $T$ is a self-adjoint extension of $T_{\min}$, then the relation $T_{\min}\subseteq T$ and \eqref{2.9} imply
\begin{equation}\lb{2.18}
T_{\min}\subseteq T \subseteq T_{\max}.
\end{equation}
Hence, $T$ is a self-adjoint extension of $T_{\min}$ if and only if $T$ is a self-adjoint restriction of $T_{\max}$.

\begin{remark}\lb{r2.12}
If $T$ is a self-adjoint extension of $T_{\min}$ then $ \rho(T)\subset \mathrm{r}(T_{\min})$.  In particular, by Lemma \ref{l2.9}, if $c\in \{a,b\}$ and $z\in \rho(T)$, then there is a nontrivial solution of $\tau u = zu$ which lies in $L^2((a,b);r(x)\, dx)$ near $c$.  This solution is unique up to constant multiples if $c$ is a limit point endpoint.\hfill $\diamond$
\end{remark}

Next, we recall the notion of what it means for two self-adjoint extensions of a symmetric operator $S$ to be {\it relatively prime}.

\begin{definition}\lb{d2.13}
If $T$ and $T'$ are self-adjoint extensions of a symmetric operator $S$, then the {\it maximal common part} of $T$ and $T'$ is the operator $C_{T,T'}$ defined by
\begin{equation}\lb{2.19}
C_{T,T'}u = Tu,\quad u\in \dom(C_{T,T'})=\{f\in \dom(T)\cap\dom(T')\,|\, Tf=T'f\}.
\end{equation}
Moreover, $T$ and $T'$ are said to be relatively prime with respect to $S$ if $C_{T,T'}=S$.
\end{definition}

Since a self-adjoint extension $T$ of $T_{\min}$ is also a self-adjoint restriction of $T_{\max}$, to characterize the self-adjoint extension $T$, it suffices to characterize the domain of $T$ (the action of $T$ being that of $T_{\max}$).  The domain of a self-adjoint extension can be characterized in terms of the Lagrange bracket and {\it boundary condition bases}.

\begin{definition}[{\cite[Definition 10.4.3]{Ze05}}]\lb{d2.14}
A pair of real-valued functions $\{\phi,\psi\}$ on $(a,b)$ is called a boundary condition basis at $a$ $($resp., $b$$)$ if $\phi,\psi\in \dom(T_{\max})$ and $[\psi,\phi](a)=1$ $($resp., $[\psi,\phi](b)=1$$)$.  
\end{definition}

Lemma \ref{l2.10} (in particular, \eqref{2.14}) shows that a boundary condition basis cannot exist at a limit point endpoint.  However, a boundary condition basis always exists at a limit circle endpoint.
\begin{lemma}\lb{l2.15}
Assume Hypothesis \ref{h2.1}.  If $c\in \{a,b\}$ and $\tau$ is in the limit circle case at $c$, then there exists a boundary condition basis $\{\phi_c,\psi_c\}$ at $c$.
\end{lemma}
\begin{proof}
Let $c\in \{a,b\}$ and suppose $\tau$ is in the limit circle case at $c$.  By Lemma \ref{l2.10}, there exists $f\in \dom(T_{\max})$ such that \eqref{2.15} holds.  Writing $f=\Re(f)+i\Im(f)$ and applying linearity of the Lagrange bracket, one infers that either $[\Re(f),g](a)\neq 0$ or $[\Im(f),g](a)\neq 0$.  Taking $\wti f = \Re(f)$ or $\wti f = \Im(f)$ accordingly, one obtains a real-valued function $\wti f\in \dom(T_{\max})$ with $[\wti f,g](a)\neq 0$.  Similarly, decomposing $g$ into its real and imaginary parts, one obtains a real-valued function $\wti g\in \dom(T_{\max})$ such that $[\wti f,\wti g](a)\neq 0$.  Taking $\phi_c = \wti g$ and $\psi_c = \{[\wti f,\wti g](a)\}^{-1}\wti f$, one infers that $\{\phi_c,\psi_c\}$ is a boundary condition basis at $c$.
\end{proof}

The next lemma provides a characterization of $T_{\min}$ in terms of boundary condition bases.

\begin{lemma}\lb{l2.16}
Assume Hypothesis \ref{h2.1}.  The following statements $(i)$ and $(ii)$ hold.\\
$(i)$  If $\tau$ is in the limit circle case at $a$, $\{\phi_a,\psi_a\}$ is a boundary condition basis at $a$, and $\tau$ in the limit point case at $b$, then
\begin{equation}\lb{2.20}
\dom(T_{\min}) = \{g\in \dom(T_{\max})\,|\, [g,\phi_a](a)=[g,\psi_a](a)=0\}.
\end{equation}
An analogous statement holds if $\tau$ is in the limit point case at $a$ and in the limit circle case at $b$.\\
$(ii)$  If $\tau$ is in the limit circle case at $a$ and $b$ and $\{\phi_c,\psi_c\}$ is a boundary condition basis at the endpoint $c\in\{a,b\}$, then
\begin{equation}\lb{2.21}
\dom(T_{\min}) = \{g\in \dom(T_{\max})\,|\, [g,\phi_a](a)=[g,\psi_a](a)=[g,\phi_b](b)=[g,\psi_b](b)=0\}.
\end{equation}
\end{lemma}
\begin{proof}
We provide a proof of $(i)$; the proof of $(ii)$ is similar.  Suppose $\tau$ is in the limit circle case at $a$ with $\{\phi_a,\psi_a\}$ a boundary condition basis at $a$, and suppose $\tau$ is in the limit point case at $b$.  Denote the set on the right-hand side in \eqref{2.20} by $\cA$, and let $g\in \cA$, so that
\begin{equation}\lb{2.22}
[g,\phi_a](a)=[g,\psi_a](a)=0.
\end{equation}
If $h\in \dom(T_{\max})$, then an application of Lemma \ref{2.2} with the choices $f_1=g$, $f_2=h$, $f_3=\phi_a$, and $f_4=\psi_a$ yields
\begin{align}
0&=[g,h](a)[\phi_a,\psi_a](a)+[g,\phi_a](a)[\psi_a,h](a)+[g,\psi_a](a)[\ol{h},\phi_a](a)\no\\
&= -[g,h](a).
	\lb{2.23}
\end{align}
Therefore, $[g,h](a)=0$.  In addition, since $\tau$ is in the limit point case at $b$, Lemma \ref{l2.10} implies $[g,h](b)=0$.  Since $h\in \dom(T_{\max})$ was arbitrary, it follows from \eqref{2.12} that $g\in \dom(T_{\min})$.  Hence, $\cA\subseteq \dom(T_{\min})$.

Conversely, if $g\in \dom(T_{\min})$, then $[g,h](a)=0$ for all $h\in \dom(T_{\max})$ by \eqref{2.12}.  Separately choosing $h=\phi_a$ and $h=\psi_a$, one concludes that $g\in \cA$.  Hence, $\dom(T_{\min})\subseteq\cA$.  Having shown the two set inclusions, \eqref{2.20} follows.
\end{proof}

Next, we recall several theorems on the parametrization of the self-adjoint extensions of $T_{\min}$.  The precise form of the self-adjoint extensions depends on the limit point/limit circle classification of $\tau$ at each of the endpoints $\{a,b\}$.  One of the primary reasons for stating the parametrizations here is to introduce notation to be used in later sections.  To slightly shorten the statement of theorems and to make assumptions clear, we introduce the following basic hypothesis.

\begin{hypothesis}\lb{h2.17}
In addition to Hypothesis \ref{h2.1}, let $T_{\max}$ and $T_{\min}$ denote the maximal and minimal operators defined by \eqref{2.8} and \eqref{2.9} $($equivalently, \eqref{2.12}$)$, respectively.
\end{hypothesis}

To begin with, if $\tau$ is in the limit point case at both $a$ and $b$, then $T_{\min}$ is a self-adjoint operator.

\begin{theorem}[{\cite[Theorem 5.2]{EGNT13}}]\lb{t2.18}
Assume Hypothesis \ref{h2.17}.  If $\tau$ is in the limit point case at both $a$ and $b$, then $T_{\min}=T_{\max}$.  That is, $T_{\min}$ is self-adjoint and, therefore, possesses no proper self-adjoint extensions.
\end{theorem}

In the case of exactly one limit circle endpoint, all self-adjoint extensions of $T_{\min}$ are characterized by a separated boundary condition using a boundary condition basis at the limit circle endpoint.

\begin{theorem}[{\cite[Theorem 6.2]{EGNT13}}]\lb{t2.19}
Assume Hypothesis \ref{h2.17} and let $c\in \{a,b\}$.  Suppose $\tau$ is in the limit circle case at $c$, $\{\phi_c,\psi_c\}$ is a boundary condition basis at $c$, and that $\tau$ is in the limit point case at the other endpoint.  If $\theta\in [0,\pi)$, then the operator $T_{\theta}$ defined by
\begin{equation}\lb{2.24}
\begin{split}
&T_{\theta}f = T_{\max}f,\\
&f\in \dom(T_{\theta}) = \{g\in \dom(T_{\max})\,|\, \cos(\theta)[g,\phi_c](c)+\sin(\theta)[g,\psi_c](c)=0\},
\end{split}
\end{equation}
is a self-adjoint extension of $T_{\min}$.   Conversely, if $T$ is a self-adjoint extension of $T_{\min}$, then $T=T_{\theta}$ for some $\theta\in [0,\pi)$.
\end{theorem}

If $\tau$ is in the limit circle case at both $a$ and $b$, then one must impose boundary conditions at both endpoints to obtain a self-adjoint extension.  In this case, self-adjoint boundary conditions are categorized into two classes: {\it separated} boundary conditions (cf.~\eqref{2.25} below) and {\it coupled} boundary conditions (cf.~\eqref{2.26} below).

\begin{theorem}[{\cite[Theorem 6.4]{EGNT13}}, {\cite[Section 10.4.5]{Ze05}}]\lb{t2.20}
Assume Hypothesis \ref{h2.17}.  Suppose that $\tau$ is in the limit circle case at both $a$ and $b$, and let $\{\phi_a,\psi_a\}$ and $\{\phi_b,\psi_b\}$ denote boundary condition bases at $a$ and $b$, respectively.  Then the following statements $(i)$--\,$(iii)$ hold.\\[1mm]
$(i)$  If $\alpha,\beta\in [0,\pi)$, then the operator $T_{\alpha,\beta}$ defined by
\begin{align}
&T_{\alpha,\beta}f = T_{\max}f,\lb{2.25}\\
&f\in \dom(T_{\alpha,\beta}) = \bigg\{g\in \dom(T_{\max})\,\bigg|\,
\begin{aligned}
\cos(\alpha)[g,\phi_a](a)+\sin(\alpha)[g,\psi_a](a)&=0,\no\\
\cos(\beta)[g,\phi_b](b)+\sin(\beta)[g,\psi_b](b)&=0
\end{aligned}
\bigg\},\no
\end{align}
is a self-adjoint extension of $T_{\min}$.\\[1mm]
$(ii)$  If $\eta\in [0,\pi)$ and $R\in \SL_{2}(\bbR)$, then the operator $T_{R,\eta}$ defined by
\begin{align}
&T_{R,\eta}f = T_{\max}f,\lb{2.26}\\
&f\in \dom(T_{R,\eta}) = \bigg\{g\in \dom(T_{\max})\,\bigg|\,
\begin{pmatrix}
[g,\phi_b](b)\\
[g,\psi_b](b)
\end{pmatrix}
= e^{i\eta}R
\begin{pmatrix}
[g,\phi_a](a)\\
[g,\psi_a](a)
\end{pmatrix}
\bigg\},\no
\end{align}
is a self-adjoint extension of $T_{\min}$.\\[1mm]
$(iii)$  If $T$ is a self-adjoint extension of $T_{\min}$, then $T=T_{\alpha,\beta}$ for some $\alpha,\beta\in[0,\pi)$ or $T=T_{R,\eta}$ for some $\eta\in [0,\pi)$ and some $R\in \SL_2(\bbR)$.
\end{theorem}

\begin{remark}\lb{r2.21}
The parametrization in \eqref{2.24} is a restatement (in the language of boundary condition bases and the Lagrange bracket) of \cite[Theorem 6.2]{EGNT13}.  Specifically, \eqref{2.24} is obtained from \cite[eq.~(6.10)]{EGNT13} by choosing ``$w_1$'' and ``$w_2$'' in the notation of \cite[eqs.~(6.1)--(6.4)]{EGNT13} to be $\psi_c$ and $\phi_c$, respectively.  The parametrization in \eqref{2.25} is obtained from \cite[eq.~(6.23)]{EGNT13} by choosing ``$w_1$'' and ``$w_2$'' in \cite[eqs.~(6.1)--(6.4)]{EGNT13} such that
\begin{equation}\lb{2.27}
\text{$w_1\in \dom(T_{\max})$ coincides with $\psi_a$ near $a$ and $\psi_b$ near $b$}
\end{equation}
and
\begin{equation}\lb{2.28}
\text{$w_2\in \dom(T_{\max})$ coincides with $\phi_a$ near $a$ and $\phi_b$ near $b$}.
\end{equation}
These choices are possible by the Naimark patching lemma \cite[Chapter V, Section 17.3, Lemma 2]{Na68}.  Finally, the parametrization in \eqref{2.26} follows from \cite[eq.~(6.24)]{EGNT13} with the same choices \eqref{2.27} and \eqref{2.28} after a minor additional observation.  For fixed $R\in \SL_2(\bbR)$ and fixed $\varphi\in [0,\pi)$, the boundary conditions in \cite[eq.~(6.24)]{EGNT13} with the choices \eqref{2.27} and \eqref{2.28} actually read
\begin{equation}\lb{2.29}
\begin{pmatrix}
[g,\phi_b](b)\\
-[g,\psi_b](b)
\end{pmatrix}= e^{i\varphi}R
\begin{pmatrix}
[g,\phi_a](a)\\
-[g,\psi_a](a)
\end{pmatrix}.
\end{equation}
However, multiplying from the left on both sides of \eqref{2.29} by the $2\times 2$ diagonal matrix $J:=\diag(1,-1)$ and using $J^2=I_{2\times 2}$, the condition in \eqref{2.29} is equivalent to
\begin{equation}
\begin{pmatrix}
[g,\phi_b](b)\\
[g,\psi_b](b)
\end{pmatrix}= e^{i\varphi}JRJ
\begin{pmatrix}
[g,\phi_a](a)\\
[g,\psi_a](a)
\end{pmatrix}.
	\lb{2.30}
\end{equation}
Upon taking $\eta=\varphi$ and $\wti R = JRJ\in \bbR^{2\times 2}$, one infers that $\wti R \in \SL_2(\bbR)$ and \eqref{2.29} is equivalent to
\begin{equation}
\begin{pmatrix}
[g,\phi_b](b)\\
[g,\psi_b](b)
\end{pmatrix}= e^{i\eta}\wti R
\begin{pmatrix}
[g,\phi_a](a)\\
[g,\psi_a](a)
\end{pmatrix}.
	\lb{2.31}
\end{equation}
As a result, \eqref{2.26} encompasses all self-adjoint extensions as characterized by \cite[eq.~(6.24)]{EGNT13} and vice versa.\hfill $\diamond$
\end{remark}

\section{The Case of Exactly One Limit Circle Endpoint} \lb{s3}

In this section, we assume that $\tau$ is in the limit circle case at exactly one endpoint.  Fixing $T_0$ (cf.~\eqref{2.24}) as a reference self-adjoint extension of $T_{\min}$, we derive explicit Krein resolvent identities that relate the resolvent of any other self-adjoint extension $T_{\theta}$, $\theta\in (0,\pi)$, of $T_{\min}$ to the resolvent of $T_0$.  The resolvent identity is then used to compute the trace of the difference of the resolvents of $T_{\theta}$ and $T_0$.  We treat in detail the case where $a$ is the lone limit circle endpoint.  Analogous formulas hold if $b$ is the only limit circle endpoint.  We fix some assumptions to begin:

\begin{hypothesis}\lb{h3.1}
Assume, in addition to Hypothesis \ref{h2.17}, that:\\[1mm]
$(i)$ $\tau$ is in the limit point case at $b$ and in the limit circle case at $a$ with $\{\phi_a,\psi_a\}$ a boundary condition basis at $a$.\\[1mm]
$(ii)$  For each $\theta\in [0,\pi)$, $T_{\theta}$ is the self-adjoint extension of $T_{\min}$ defined by \eqref{2.24} with $c=a$.\\[1mm]
$(iii)$  For each $z\in \rho(T_0)$, $u_z$ is the unique solution $($cf.~\cite[Lemma 10.4.8]{Ze05}$)$ to \eqref{2.13} which satisfies
\begin{equation}
[u_z,\phi_a](a) = 0\quad \text{and}\quad [u_z,\psi_a](a)=1.
	\lb{3.1}
\end{equation}
$(iv)$  For each $z\in \rho(T_0)$, $w_z$ is the unique solution $($cf.~Remark \ref{r2.12}$)$ to \eqref{2.13} which satisfies
\begin{equation}
w_z\in L^2((a,b);r(x)\, dx)\quad \text{and}\quad [w_z,\phi_a](a)=1.
	\lb{3.2}
\end{equation}
\end{hypothesis}

Assuming Hypothesis \ref{h3.1}, the functions $u_z$ and $w_z$ are called the {\it regular} and {\it Weyl--Titchmarsh solutions}, respectively, and, since $p$, $q$, and $r$ are real-valued,
\begin{equation}
\overline{u_z} = u_{\overline{z}},\quad \overline{w_z} = w_{\overline{z}},\quad z\in \rho(T_0).
	\lb{3.3}
\end{equation}
In particular, $u_z$ and $w_z$ are real-valued when $z\in \bbR\cap\rho(T_0)$.  By Theorem \ref{t2.11}, the deficiency indices of $T_{\max}$ are $d_{\pm}(T_{\max})=1$.  In fact,
\begin{equation}
\text{$w_z$ spans $\ker(T_{\max}-zI_{(a,b)})$ for each $z\in \rho(T_0)$.}
	\lb{3.4}
\end{equation}

The following lemma characterizes, for fixed $f\in L^2((a,b);r(x)\, dx)$, boundary data of $(T_0-zI_{(a,b)})^{-1}f$, $z\in \rho(T_0)$, in terms of the inner product of $f$ with $w_{\ol{z}}$.

\begin{lemma}\lb{l3.2}
Assume Hypothesis \ref{h3.1}.  If $z\in \rho(T_0)$, then
\begin{equation}
\big[(T_{0}-zI_{(a,b)})^{-1}f,\psi_{a}\big](a) = -\left<w_{\ol{z}},f\right>_{(a,b)},\quad f\in L^2((a,b);r(x)\, dx).
	\lb{3.5}
\end{equation}
\end{lemma}
\begin{proof}
Let $z\in \rho(T_0)$ and $f\in L^2((a,b);r(x)\, dx)$.  By \cite[Theorem 7.1]{EGNT13} combined with \eqref{3.1} and \eqref{3.2}, $(T_0-zI_{(a,b)})^{-1}$ is an integral operator with kernel (i.e., Green's function) given by
\begin{align}
G_{0,z}(x,y)=\f{1}{[w_{z},u_{\ol{z}}](a)}
\begin{cases}
u_{z}(y)w_{z}(x),\quad & a<y\leq x<b,\\
u_{z}(x)w_{z}(y),\quad & a<x\leq y<b
\end{cases},\quad z\in\rho(T_{0}),
	\lb{3.6}
\end{align}
so that
\begin{equation}
\big[(T_{0}-zI_{(a,b)})^{-1}f\big](x)  = \int_a^b G_{0,z}(x,y)f(y)r(y)\, dy,\quad x\in (a,b).
	\lb{3.7}
\end{equation}
Differentiating throughout \eqref{3.7}, one obtains (where prime denotes differentiation with respect to $x$)
\begin{align}
&\big[(T_{0}-zI_{(a,b)})^{-1}f\big]'(x)  
	\lb{3.8}\\
&\quad =\f{1}{[w_{z},u_{\ol{z}}](a)}\bigg(w_{z}'(x)\int_{a}^{x}u_{z}(y)f(y)r(y)\,dy+w_{z}(x)u_{z}(x)f(x)r(x)\no\\
&\quad\quad+u_{z}'(x)\int_{x}^{b}w_{z}(y)f(y)r(y)\,dy-u_{z}(x)w_{z}(x)f(x)r(x)\bigg)\no\\
&\quad=\f{1}{[w_{z},u_{\ol{z}}](a)}\bigg(w_{z}'(x)\int_{a}^{x}u_{z}(y)f(y)r(y)\,dy+u_{z}'(x)\int_{x}^{b}w_{z}(y)f(y)r(y)\,dy\bigg)\no
\end{align}
for a.e.~$x\in (a,b)$.  Applying \eqref{3.7} and \eqref{3.8}, one obtains
\begin{align}
&[w_{z},u_{\ol{z}}](a)\cdot\big[(T_{0}-zI_{(a,b)})^{-1}f,\psi_{a}\big](a)\no\\
&\quad =\underset{x\downarrow a}\lim\Bigg[\bigg(\int_{a}^{x}u_{z}(y)w_{z}(x)f(y)r(y)\,dy+\int_{x}^{b}u_{z}(x)w_{z}(y)f(y)r(y)\,dy\bigg)(p\psi_{a}')(x)\no\\
&\quad\quad-p(x)\bigg(w_{z}'(x)\int_{a}^{x}u_{z}(y)f(y)r(y)\,dy
+u_{z}'(x)\int_{x}^{b}w_{z}(y)f(y)r(y)\,dy\bigg)\psi_{a}(x)\Bigg]\no\\
&\quad =\underset{x\downarrow a}\lim\Bigg[\Big(w_{z}(x)(p\psi_{a}')(x)-(pw_{z}')(x)\psi_{a}(x)\Big)\int_{a}^{x}u_{z}(y)f(y)r(y)\,dy\no\\
&\quad\quad+\Big(u_{z}(x)(p\psi_{a}')(x)-(pu_{z}')(x)\psi_{a}(x)\Big)\int_{x}^{b}w_{z}(y)f(y)r(y)\,dy\Bigg]\no\\
&\quad=\underset{x\downarrow a}\lim\Bigg[[w_{z},\psi_{a}](x)\int_{a}^{x}u_{z}(y)f(y)r(y)\,dy+[u_{z},\psi_{a}](x)\int_{x}^{b}w_{z}(y)f(y)r(y)\,dy\Bigg]\no\\
&\quad =[u_{z},\psi_{a}](a)\int_{a}^{b}w_{z}(y)f(y)r(y)\,dy\no\\
&\quad =\int_{a}^{b}w_{z}(y)f(y)r(y)\,dy =\left<\ol{w_{z}},f\right>_{(a,b)}=\left<w_{\ol{z}},f\right>_{(a,b)}.
	\lb{3.9}
\end{align}
The limit leading to \eqref{3.9} exists by Lemma \ref{l2.5}.  An application of the Pl\"ucker-type identity \eqref{2.5} with the choices $f_{1}=w_{z}$, $f_{2}=\psi_{a}$, $f_{3}=u_{\ol{z}}$, and $f_{4}=\phi_{a}$ yields $[w_{z},u_{\ol{z}}](a)=-1$, and the claim in \eqref{3.5} follows.
\end{proof}

We recall the following abstract result for the computation of the trace of a rank one operator and provide its short proof for completeness.

\begin{lemma} \lb{l3.3}
Let $\cH$ be a separable Hilbert space, with $f,g\in\cH$, and define the rank one operator $A=\left<f,\,\cdot\,\right>_{\cH}g$ on $\dom(A)=\cH$.  Then $A\in \cB_1(\cH)$ and
\begin{equation}\lb{3.10a}
\tr_{\cH}(A)=\left<f,g\right>_{\cH}.
\end{equation}
\end{lemma}
\begin{proof}
Since $A$ is finite rank, $A\in \cB_1(\cH)$.  Fix an orthonormal basis $\{e_{\iota}\}_{\iota\in\cI}$ of $\cH$ (with $\cI\subseteq \bbN$ an appropriate indexing set), and compute
\begin{align}
\tr_{\cH}(A)=\sum\limits_{\iota\in\cI}\left<e_{\iota},\left<f,e_{\iota}\right>_{\cH}g\right>_{\cH}=\sum\limits_{\iota\in\cI}\left<f,e_{\iota}\right>_{\cH}\left<e_{\iota},g\right>_{\cH}=\left<f,g\right>_{\cH}.
	\lb{3.10}
\end{align}
\end{proof}

With these preparations out of the way, we turn to differences of resolvents of the self-adjoint extensions $T_{\theta}$ of $T_{\min}$, fixing $T_{0}$ as a reference extension.  The main result of this section is an explicit Krein-type resolvent identity and a corresponding trace formula for resolvent differences:

\begin{theorem}	\lb{t3.4}
Assume Hypothesis \ref{h3.1} and suppose $\theta\in (0,\pi)$.  Then $T_0$ and $T_{\theta}$ are relatively prime with respect to $T_{\min}$.  Moreover, for each $z\in \rho(T_0)\cap\rho(T_{\theta})$, the scalar
\begin{align}
k_{\theta}(z)=\cot(\theta)+[w_{z},\psi_{a}](a)
	\lb{3.11}
\end{align}
is nonzero and the following operator equality holds:
\begin{align}
(T_{\theta}-zI_{(a,b)})^{-1}-(T_{0}-zI_{(a,b)})^{-1} = k_{\theta}(z)^{-1}\left<w_{\ol{z}},\,\cdot\,\right>_{(a,b)}w_{z}.
	\lb{3.12}
\end{align}
In particular, for each $z\in \rho(T_0)\cap\rho(T_{\theta})$,
\begin{equation}
\big[(T_{\theta}-zI_{(a,b)})^{-1}-(T_{0}-zI_{(a,b)})^{-1}\big] \in \cB_1\big(L^2((a,b);r(x)\, dx)\big)
	\lb{3.13}
\end{equation}
and
\begin{align}
\tr_{(a,b)}\left((T_{\theta}-zI_{(a,b)})^{-1}-(T_{0}-zI_{(a,b)})^{-1}\right)=\f{\left<w_{\ol{z}},w_{z}\right>_{(a,b)}}{\cot(\theta)+[w_{z},\psi_{a}](a)}.
	\lb{3.14}
\end{align}
\end{theorem}
\begin{proof}
Let $\theta\in (0,\pi)$.  To prove that $T_0$ and $T_{\theta}$ are relatively prime with respect to $T_{\min}$, it suffices to prove
\begin{equation}
\dom(T_0)\cap\dom(T_{\theta}) \subseteq \dom(T_{\min}).
	\lb{3.15}
\end{equation}
To this end, let $g\in\dom(T_0)\cap\dom(T_{\theta})$.  By \eqref{2.24},
\begin{align}
\cos(\theta)[g,\phi_{a}](a)+\sin(\theta)[g,\psi_{a}](a)=0\ \hbox{ and }\ [g,\phi_{a}](a)=0.
	\lb{3.16}
\end{align}
However, \eqref{3.16} implies $[g,\psi_{a}](a)=0$ since $\sin(\theta)\neq0$ for $\theta\in(0,\pi)$.  By Lemma \ref{l2.16} $(i)$, $g\in\dom(T_{\min})$.  This completes the proof that $T_0$ and $T_{\theta}$ are relatively prime with respect to $T_{\min}$.

Let $z\in \rho(T_0)\cap\rho(T_{\theta})$ be fixed.  Suppose, by way of contradiction, that $k_{\theta}(z)=0$.  By Hypothesis \ref{h3.1} $(iv)$, $w_{z}\in\dom(T_{\max})$.  However, $k_{\theta}(z)=0$ implies
\begin{align}
0=\sin(\theta)k_{\theta}(z)=\cos(\theta)+\sin(\theta)[w_{z},\psi_{a}](a),
	\lb{3.17}
\end{align}
in which case, by \eqref{3.2}, $w_z\in\dom(T_{\theta})$, as well.  Now, $w_z\in\ker(T_{\max}-zI_{(a,b)})\backslash\{0\}$ implies $T_{\theta}w_z=zw_z$ so that $z\in\sigma(T_{\theta})$.  This is a contradiction to the assumption $z\in \rho(T_{\theta})$.  Therefore, $k_{\theta}(z)\neq 0$.

To prove \eqref{3.12}, define the operator
\begin{equation}
\begin{split}
F_{\theta}(z)=(T_{0}-zI_{(a,b)})^{-1}+k_{\theta}(z)^{-1}\left<w_{\ol{z}},\,\cdot\,\right>_{(a,b)}w_{z},&\\ \dom(F_{\theta}(z))=L^2((a,b);r(x)\, dx).&
	\lb{3.18}
\end{split}
\end{equation}
It suffices to show that
\begin{align}
(T_{\theta}-zI_{(a,b)})F_{\theta}(z)=I_{(a,b)};
	\lb{3.19}
\end{align}
that is, it suffices to show that for every $f\in L^2((a,b);r(x)\, dx)$,
\begin{equation}
F_{\theta}(z)f\in \dom(T_{\theta})
	\lb{3.20}
\end{equation}
and
\begin{equation}
(T_{\theta}-zI_{(a,b)})(F_{\theta}(z)f)=f.
	\lb{3.21}
\end{equation}
To this end, let $f\in L^2((a,b);r(x)\, dx)$.  It is clear from the definition of $F_{\theta}(z)$ that
\begin{equation}
F_{\theta}(z)f\in \dom(T_{\max}),
	\lb{3.22}
\end{equation}
so the proof of \eqref{3.20} reduces to showing that $F_{\theta}(z)f$ satisfies the boundary condition in \eqref{2.24} with $c=a$; that is,
\begin{equation}
\cos(\theta)[F_{\theta}(z)f,\phi_a](a)+\sin(\theta)[F_{\theta}(z)f,\psi_a](a)=0.
	\lb{3.23}
\end{equation}
One computes
\begin{align}
[F_{\theta}(z)f,\phi_{a}](a)&=\big[(T_{0}-zI_{(a,b)})^{-1}f,\phi_{a}\big](a)+k_{\theta}(z)^{-1}\left<w_{\ol{z}},f\right>_{(a,b)}[w_{z},\phi_{a}](a)\no\\
&=k_{\theta}(z)^{-1}\left<w_{\ol{z}},f\right>_{(a,b)}
	\lb{3.24}
\end{align}
by definition of $w_{z}$ and the fact that $(T_{0}-zI_{(a,b)})^{-1}f\in\dom(T_{0})$.  In addition,
\begin{align}
&[F_{\theta}(z)f,\psi_{a}](a)\no\\
&\quad =\big[(T_{0}-zI_{(a,b)})^{-1}f,\psi_{a}\big](a)+k_{\theta}(z)^{-1}\left<w_{\ol{z}},f\right>_{(a,b)}[w_{z},\psi_{a}](a).
	\lb{3.25}
\end{align}
An application of Lemma \ref{l3.2} in the first term on the right-hand side in \eqref{3.25} yields
\begin{align}
[F_{\theta}(z)f,\psi_{a}](a)=-\left<w_{\ol{z}},f\right>_{(a,b)}+k_{\theta}(z)^{-1}\left<w_{\ol{z}},f\right>_{(a,b)}[w_{z},\psi_{a}](a).
	\lb{3.26}
\end{align}
Finally, to verify \eqref{3.23}, one uses \eqref{3.24} and \eqref{3.26} as follows:
\begin{align}
&\cos(\theta)[F_{\theta}(z)f,\phi_{a}](a)+\sin(\theta)[F_{\theta}(z)f,\psi_{a}](a)\lb{3.27}\\
&\quad=\big\{\cos(\theta)k_{\theta}(z)^{-1}-\sin(\theta)+\sin(\theta)k_{\theta}(z)^{-1}[w_{z},\psi_{a}](a)\big\}\left<w_{\ol{z}},f\right>_{(a,b)}\no\\
&\quad=\big\{\cos(\theta)-\sin(\theta)k_{\theta}(z)+\sin(\theta)[w_{z},\psi_{a}](a)\big\}k_{\theta}(z)^{-1}\left<w_{\ol{z}},f\right>_{(a,b)}\no\\
&\quad=\big\{\cos(\theta)-\sin(\theta)\big(\cot(\theta)+[w_{z},\psi_{a}](a)\big)+\sin(\theta)[w_{z},\psi_{a}](a)\big\}\no\\
&\qquad\times k_{\theta}(z)^{-1}\left<w_{\ol{z}},f\right>_{(a,b)}\no\\
&\quad=0.\no
\end{align}
The proof of \eqref{3.21} combines \eqref{3.4} and the fact that $T_{\max}$ is an extension of both $T_{\theta}$ and $T_0$:
\begin{align}
&(T_{\theta}-zI_{(a,b)})F_{\theta}(z)f\lb{3.28}\\
&\quad=(T_{\theta}-zI_{(a,b)})\big[(T_{0}-zI_{(a,b)})^{-1}f+k_{\theta}(z)^{-1}\left<w_{\ol{z}},f\right>_{(a,b)}w_{z}\big]\no\\
&\quad=(T_{\max}-zI_{(a,b)})\big[(T_{0}-zI_{(a,b)})^{-1}f+k_{\theta}(z)^{-1}\left<w_{\ol{z}},f\right>_{(a,b)}w_{z}\big]\no\\
&\quad=(T_{\max}-zI_{(a,b)})(T_{0}-zI_{(a,b)})^{-1}f+k_{\theta}(z)^{-1}\left<w_{\ol{z}},f\right>_{(a,b)}(T_{\max}-zI_{(a,b)})w_{z}\no\\
&\quad=(T_{0}-zI_{(a,b)})(T_{0}-zI_{(a,b)})^{-1}f=I_{(a,b)}f=f.\no
\end{align}
The right-hand side of \eqref{3.12} is a rank one (hence, trace class) operator, so \eqref{3.13} follows.  Finally, applying Lemma \ref{l3.3}, \eqref{3.12}, and linearity of the trace functional, one computes
\begin{align}
\tr_{(a,b)}\left((T_{\theta}-zI_{(a,b)})^{-1}-(T_{0}-zI_{(a,b)})^{-1}\right)&=k_{\theta}(z)^{-1}\tr_{(a,b)}\big(\left<w_{\ol{z}},\,\cdot\,\right>_{(a,b)}w_{z}\big)\no\\
&=\f{\left<w_{\ol{z}},w_{z}\right>_{(a,b)}}{\cot(\theta)+[w_{z},\psi_{a}](a)}.
	\lb{3.29}
\end{align}
\end{proof}

\begin{remark}\lb{r3.5}
The identity in \eqref{3.12} yields a similar identity that relates the resolvents of any two self-adjoint extensions $T_{\theta_j}$, $\theta_j\in [0,\pi)$, $j\in \{1,2\}$.  For $\theta_{1},\theta_{2}\in[0,\pi)$ and $z\in\rho(T_{\theta_{1}})\cap\rho(T_{\theta_{2}})\cap \rho(T_0)$, the difference
\begin{equation}
(T_{\theta_{1}}-zI_{(a,b)})^{-1}-(T_{\theta_{2}}-zI_{(a,b)})^{-1}
	\lb{3.30}
\end{equation}
can be completely characterized by 
\begin{equation}
(T_{\theta_{j}}-zI_{(a,b)})^{-1}-(T_{0}-zI_{(a,b)})^{-1},\quad j\in\{1,2\},
	\lb{3.31}
\end{equation}
by adding and subtracting $(T_{0}-zI_{(a,b)})^{-1}$ and applying \eqref{3.12} to obtain
\begin{align}
&(T_{\theta_{1}}-zI_{(a,b)})^{-1}-(T_{\theta_{2}}-zI_{(a,b)})^{-1}
	\lb{3.32}\\
&\quad=\left[(T_{\theta_{1}}-zI_{(a,b)})^{-1}-(T_{0}-zI_{(a,b)})^{-1}\right]-\left[(T_{\theta_{2}}-zI_{(a,b)})^{-1}-(T_{0}-zI_{(a,b)})^{-1}\right]\no\\
&\quad =\big[k_{\theta_1}(z)^{-1}-k_{\theta_2}(z)^{-1}\big]\left<w_{\ol{z}},\,\cdot\,\right>_{(a,b)}w_{z}.\no
\end{align}
\hfill $\diamond$
\end{remark}

\section{The Case of Two Limit Circle Endpoints} \lb{s4}

In this section, we assume that $\tau$ is in the limit circle case at both endpoints of $(a,b)$.  Fixing $T_{0,0}$ (cf.~\eqref{2.25}) as a reference self-adjoint extension of $T_{\min}$, we derive explicit Krein resolvent identities that relate the resolvent of any other self-adjoint extension of $T_{\min}$ to the resolvent of $T_{0,0}$.  We distinguish the two cases of self-adjoint extensions parametrized by separated boundary conditions \eqref{2.25} and those parametrized by coupled boundary conditions \eqref{2.26}.  To set the stage, we introduce the following hypothesis.

\begin{hypothesis}\lb{h4.1}
In addition to Hypothesis \ref{h2.17}, suppose that $\tau$ is in the limit circle case at $a$ and $b$ and:\\[1mm]
$(i)$ Let $\{\phi_a,\psi_a\}$ and $\{\phi_b,\psi_b\}$ be boundary condition bases at $a$ and $b$, respectively.\\[1mm]
$(ii)$ For each $\alpha,\beta \in [0,\pi)$, let $T_{\alpha,\beta}$ denote the self-adjoint extension of $T_{\min}$ defined in \eqref{2.25}.  In particular, $T_{0,0}$ denotes the self-adjoint extension of $T_{\min}$ with domain
\begin{equation}\lb{4.1}
\dom(T_{0,0})=\{g\in\dom(T_{\max})\, |\, [g,\phi_a](a) = [g,\phi_b](b)=0\}.
\end{equation}
$(iii)$  For each $z\in \rho(T_{0,0})$, let $\{u_{z,j}\}_{j=1,2}$ denote solutions to \eqref{2.13} which satisfy the boundary conditions
\begin{equation}\lb{4.2}
\begin{aligned}
[u_{z,1},\phi_a](a)=0, &\quad [u_{z,1},\phi_b](b)=1,\\
[u_{z,2},\phi_a](a)=1, &\quad [u_{z,2},\phi_b](b)=0.
\end{aligned}
\end{equation}
$(iv)$ For each $\eta\in [0,\pi)$ and each $R\in \SL_{2}(\bbR)$, let $T_{R,\eta}$ denote the self-adjoint extension of $T_{\min}$ defined in \eqref{2.26}.\\
\end{hypothesis}

Solutions $u_{z,j}$, $j\in \{1,2\}$, of \eqref{2.13} satisfying \eqref{4.2} exist for $z\in \rho(T_{0,0})$.  To obtain $u_{z,1}$, for example, consider the unique solution $u$ to \eqref{2.13} satisfying the initial conditions
\begin{equation}\lb{4.3}
[u,\phi_a](a)=0\quad \text{and}\quad [u,\psi_a](a)=1.
\end{equation}
Note that the initial value problem for \eqref{2.13} corresponding to \eqref{4.3} has a unique solution by \cite[Lemma 10.4.8]{Ze05}.  One infers that $[u,\phi_b](b)\neq0$; otherwise, $u\in \dom(T_{0,0})$ and $z$ is an eigenvalue of $T_{0,0}$ (however, we have assumed $z\in \rho(T_{0,0})$).  Therefore, one may take $u_{z,1}=\{[u,\phi_b](b)\}^{-1}u$.  The solution $u_{z,2}$ is obtained in an analogous manner.

Assuming Hypothesis \ref{4.1}, the fact that $p$, $q$, and $r$ are real-valued implies
\begin{equation}\lb{4.4}
\ol{u_{z,j}} = u_{\ol{z},j},\quad j\in\{1,2\}.
\end{equation}
Therefore, $u_{z,j}$, $j\in \{1,2\}$, is real-valued when $z\in \bbR\cap\rho(T_{0,0})$.

Since $\tau$ is in the limit circle case at $a$ and $b$ and solutions to \eqref{2.13} are locally absolutely continuous, one infers
\begin{equation}\lb{4.5}
u_{z,j}\in \ker(T_{\max}-zI_{(a,b)})\subset \dom(T_{\max}),\quad j\in \{1,2\},\, z\in \rho(T_{0,0}).
\end{equation}
In particular,
\begin{equation}\lb{4.6}
\text{$\{u_{z,j}\}_{j=1,2}$ is a basis for $\ker(T_{\max}-zI_{(a,b)})$ for each $z\in \rho(T_{0,0})$.}
\end{equation}

The following lemma characterizes, for fixed $f\in L^2((a,b);r(x)\, dx)$, boundary data of $(T_{0,0}-zI_{(a,b)})^{-1}f$, $z\in\rho(T_{0,0})$, in terms of inner products of $f$ with $u_{\ol{z},j}$, $j\in \{1,2\}$.

\begin{lemma}\lb{l4.2}
Assume Hypothesis \ref{h4.1}.  If $z\in \rho(T_{0,0})$, then
\begin{equation}\lb{4.7}
\begin{split}
\big[(T_{0,0}-zI_{(a,b)})^{-1}f,\psi_a\big](a) &= -\langle u_{\ol{z},2},f\rangle_{(a,b)},\\
\big[(T_{0,0}-zI_{(a,b)})^{-1}f,\psi_b\big](b) &=  \langle u_{\ol{z},1},f\rangle_{(a,b)},\quad f\in L^2((a,b);r(x)\, dx).
\end{split}
\end{equation}
\end{lemma}
\begin{proof}
Let $z\in \rho(T_{0,0})$ be fixed.  By hypothesis,
\begin{equation}\lb{4.8}
\text{$u_{z,1}$ satisfies the boundary condition at $a$ appearing in \eqref{4.1},}
\end{equation}
and
\begin{equation}\lb{4.9}
\text{$u_{z,2}$ satisfies the boundary condition at $b$ appearing in \eqref{4.1}.}
\end{equation}
By \cite[Theorem 7.3]{EGNT13}, combined with \eqref{4.8} and \eqref{4.9}, the operator $(T_{0,0}-zI_{(a,b)})^{-1}$ is an integral operator with integral kernel
\begin{equation}\lb{4.10}
G_{0,0,z}(x,y) = \frac{1}{[u_{z,2},u_{\ol{z},1}](b)}
\begin{cases}
u_{z,1}(x)u_{z,2}(y),& a< x\leq y< b,\\
u_{z,1}(y)u_{z,2}(x),& a<y\leq x<b,
\end{cases}
\end{equation}
so that for every $f\in L^2((a,b);r(x)\, dx)$,
\begin{equation}\lb{4.11}
\big[(T_{0,0}-zI_{(a,b)})^{-1}f\big](x) = \int_a^b G_{0,0,z}(x,y)f(y)r(y)\, dy,\quad x\in (a,b).
\end{equation}
Note that by \eqref{2.6},
\begin{equation}\lb{4.12}
\text{$[u_{z,2},u_{\ol{z},1}](x)$ is a constant function of $x\in(a,b)$.}
\end{equation}
For $f\in L^2((a,b);r(x)\, dx)$, one computes
\begin{align}
&\big[(T_{0,0}-zI_{(a,b)})^{-1}f,\psi_a\big](a)\lb{4.13}\\
&\quad = \lim_{x\downarrow a}\big\{\big[(T_{0,0}-zI_{(a,b)})^{-1}f\big](x)p(x)\psi_a'(x) - \psi_a(x)p(x)\big[(T_{0,0}-zI_{(a,b)})^{-1}f\big]'(x)\big\}.\no
\end{align}
To determine $\big[(T_{0,0}-zI_{(a,b)})^{-1}f\big]'$, one applies \eqref{4.10}--\eqref{4.12} as follows
\begin{align}
&[u_{z,2},u_{\ol{z},1}](b)\cdot\big[(T_{0,0}-zI_{(a,b)})^{-1}f\big]'(x)\lb{4.14}\\
&\quad = u_{z,2}'(x)\int_a^x u_{z,1}(y)f(y)r(y)\, dy + u_{z,2}(x)u_{z,1}(x)f(x)r(x)\no\\
&\quad\quad+ u_{z,1}'(x)\int_x^b u_{z,2}(y)f(y)r(y)\, dy - u_{z,1}(x)u_{z,2}(x)f(x)r(x)\no\\
&\quad =  u_{z,2}'(x)\int_a^x u_{z,1}(y)f(y)r(y)\, dy + u_{z,1}'(x)\int_x^b u_{z,2}(y)f(y)r(y)\, dy \no
\end{align}
for a.e.~$x\in (a,b)$.  Therefore, \eqref{4.11}, \eqref{4.12}, and \eqref{4.14} imply
\begin{align}
&\big[(T_{0,0}-zI_{(a,b)})^{-1}f,\psi_a\big](x)\lb{4.15}\\
&\quad = \frac{1}{[u_{z,2},u_{\ol{z},1}](a)}p(x)\psi_a'(x)\bigg\{u_{z,2}(x)\int_a^x u_{z,1}(y)f(y)r(y)\, dy \no\\
&\hspace*{4.7cm}+ u_{z,1}(x)\int_x^b u_{z,2}(y)f(y)r(y)\, dy\bigg\} \no\\
&\qquad - \frac{1}{[u_{z,2},u_{\ol{z},1}](a)} p(x)\psi_a(x)\bigg\{u_{z,2}'(x)\int_a^x u_{z,1}(y)f(y)r(y)\, dy \no\\
&\hspace*{5cm} + u_{z,1}'(x)\int_x^b u_{z,2}(y)f(y)r(y)\, dy\bigg\}\no\\
&\quad= \frac{[u_{z,2},\psi_a](x)}{[u_{z,2},u_{\ol{z},1}](a)}\int_a^x u_{z,1}(y)f(y)r(y)\, dy + \frac{[u_{z,1},\psi_a](x)}{[u_{z,2},u_{\ol{z},1}](a)}\int_x^b u_{z,2}(y)f(y)r(y)\, dy,\no\\
&\hspace*{10.78cm}x\in(a,b).\no
\end{align}
Taking the limit $x\downarrow a$ throughout \eqref{4.15} and applying \eqref{4.4} yields
\begin{equation}\lb{4.16}
\big[(T_{0,0}-zI_{(a,b)})^{-1}f,\psi_a\big](a) =  \frac{[u_{z,1},\psi_a](a)}{[u_{z,2},u_{\ol{z},1}](a)} \langle u_{\ol{z},2},f \rangle_{(a,b)}.
\end{equation}
Next, an application of Lemma \ref{l2.2} with the choices $f_1=u_{z,2}$, $f_2=u_{\ol{z},1}$, $f_3=\psi_a$, and $f_4=\phi_a$ yields
\begin{align}\lb{4.18}
[u_{z,2},u_{\ol{z},1}](a) = -[u_{z,1},\psi_a](a).
\end{align}
Finally, \eqref{4.16} and \eqref{4.18} combine to yield the first identity in \eqref{4.7}.  The second identity in \eqref{4.7} is established in an entirely analogous manner, and we omit further details at this point.
\end{proof}

\begin{remark}\lb{r4.3}
An application of Lemma \ref{l2.2} with the choices $f_1=u_{z,2}$, $f_2=u_{\ol{z},1}$, $f_3=\psi_b$, and $f_4=\phi_b$ yields
\begin{align}\lb{4.20}
[u_{z,2},u_{\ol{z},1}](b) = [u_{z,2},\psi_b](b).
\end{align}
Thus, in light of \eqref{4.12} and \eqref{4.18}, one infers
\begin{equation}\lb{4.21}
-[u_{z,1},\psi_a](a)=[u_{z,2},\psi_b](b).
\end{equation}
${}$ \hfill $\diamond$
\end{remark}

With these preparations in place, we are now ready to state the first set of main results in this section, a Krein resolvent identity for $T_{0,0}$ and $T_{\alpha,\beta}$.  To simplify the statement of theorems, we treat the case when $T_{0,0}$ and $T_{\alpha,\beta}$ are relatively prime separate from the degenerate case when $T_{0,0}$ and $T_{\alpha,\beta}$ have a maximal common part which is a proper extension of $T_{\min}$.

\begin{theorem}\lb{t4.4}
Assume Hypothesis \ref{h4.1}.  If $\alpha,\beta\in(0,\pi)$, then $T_{0,0}$ and $T_{\alpha,\beta}$ are relatively prime with respect to $T_{\min}$.  Moreover, for each $z\in \rho(T_{0,0})\cap\rho(T_{\alpha,\beta})$ the matrix
\begin{align}\lb{4.22}
K_{\alpha,\beta}(z) =
\begin{pmatrix}
\cot(\beta)+[u_{z,1},\psi_b](b) & -[u_{z,1},\psi_a](a)\\[2mm]
[u_{z,2},\psi_b](b) & -\cot(\alpha)-[u_{z,2},\psi_a](a)
\end{pmatrix}
\end{align}
is invertible and
\begin{align}\lb{4.23}
(T_{\alpha,\beta}-zI_{(a,b)})^{-1} = (T_{0,0}-zI_{(a,b)})^{-1} - \sum_{j,k=1}^2\big[K_{\alpha,\beta}(z)^{-1}\big]_{j,k}\langle u_{\ol{z},j},\,\cdot\,\rangle_{(a,b)} u_{z,k}.
\end{align}
\end{theorem}
\begin{proof}
Let $\alpha,\beta\in (0,\pi)$.  To prove that $T_{0,0}$ and $T_{\alpha,\beta}$ are relatively prime with respect to $T_{\min}$, it suffices to prove
\begin{equation}\lb{4.24}
\dom(T_{0,0})\cap\dom(T_{\alpha,\beta}) \subseteq \dom(T_{\min}).
\end{equation}
To this end, let $y\in \dom(T_{0,0})\cap\dom(T_{\alpha,\beta})$.  The condition $y\in \dom(T_{0,0})$ implies
\begin{equation}\lb{4.25}
[y,\phi_a](a)=0\quad \text{and}\quad [y,\phi_b](b)=0,
\end{equation}
and the condition $y\in \dom(T_{\alpha,\beta})$ implies
\begin{equation}\lb{4.26}
\begin{split}
\cos(\alpha)[y,\phi_a](a) + \sin(\alpha)[y,\psi_a](a)&=0,\\
\cos(\beta)[y,\phi_b](b) + \sin(\beta)[y,\psi_b](b)&=0.
\end{split}
\end{equation}
The equations in \eqref{4.25} and \eqref{4.26} together imply
\begin{equation}\lb{4.27}
[y,\psi_a](a)=0\quad \text{and}\quad [y,\psi_b](b)=0,
\end{equation}
since $\alpha,\beta\in (0,\pi)$ implies $\sin(\alpha)\neq 0$ and $\sin(\beta)\neq 0$.  The inclusion $y\in \dom(T_{\min})$ follows from \eqref{2.21} in light of \eqref{4.25} and \eqref{4.27}.  Thus, \eqref{4.24} is established.  It remains to prove the invertibility of the matrix \eqref{4.22} and to establish the resolvent identity \eqref{4.23}.

Let $z\in \rho(T_{0,0})\cap\rho(T_{\alpha,\beta})$ be fixed.  Suppose, by way of contradiction, that $K_{\alpha,\beta}(z)$ is a singular matrix.  Then $K_{\alpha,\beta}(z)$ has a non-trivial null space, so there exists $c,d\in \bbC$ with $|c|^2+|d|^2\neq 0$ and
\begin{equation}
K_{\alpha,\beta}(z)
\begin{pmatrix}
c\\
d
\end{pmatrix}
=
\begin{pmatrix}
0\\
0
\end{pmatrix}.
	\lb{4.28}
\end{equation}
Therefore,
\begin{equation}\lb{4.29}
\begin{split}
\big\{\cos(\beta)+\sin(\beta)[u_{z,1},\psi_b](b)\big\}c - \sin(\beta)[u_{z,1},\psi_a](a)d&=0,\\
\sin(\alpha)[u_{z,2},\psi_b](b)c - \big\{\cos(\alpha) + \sin(\alpha)[u_{z,2},\psi_a](a)\big\}d&=0.
\end{split}
\end{equation}
By \eqref{4.21}, the set of equations in \eqref{4.29} can be recast as
\begin{align}
\begin{split}
\big\{\cos(\beta)+\sin(\beta)[u_{z,1},\psi_b](b)\big\}c + \sin(\beta)[u_{z,2},\psi_b](b)d&=0,\\
\sin(\alpha)[u_{z,1},\psi_a](a)c + \big\{\cos(\alpha) + \sin(\alpha)[u_{z,2},\psi_a](a)\big\}d&=0.\lb{4.30}
\end{split}
\end{align}
By the first equation in \eqref{4.30}, the function $cu_{z,1}+du_{z,2}$ satisfies the boundary condition for functions in $\dom(T_{\alpha,\beta})$ at the endpoint $b$.  Indeed, using \eqref{4.2} one computes
\begin{align}
&\cos(\beta)[cu_{z,1}+du_{z,2},\phi_b](b) + \sin(\beta)[cu_{z,1}+du_{z,2},\psi_b](b)\no\\
&\quad =\cos(\beta)[u_{z,1},\phi_b](b)c + \cos(\beta)[u_{z,2},\phi_b](b)d\no\\
&\qquad + \sin(\beta)[u_{z,1},\psi_b](b)c + \sin(\beta)[u_{z,2},\psi_b](b)d\no\\
&\quad = \big\{\cos(\beta) + \sin(\beta)[u_{z,1},\psi_b](b)\big\}c + \sin(\beta)[u_{z,2},\psi_b](b)d\no\\
&\quad = 0,
	\lb{4.31}
\end{align}
where the final equality follows from the first equation in \eqref{4.30}.  On the other hand, employing \eqref{4.2} once more yields
\begin{align}
&\cos(\alpha)[cu_{z,1}+du_{z,2},\phi_a](a) + \sin(\alpha)[cu_{z,1}+du_{z,2},\psi_a](a)\no\\
&\quad=\cos(\alpha)d + \sin(\alpha)[u_{z,1},\psi_a](a)c + \sin(\alpha)[u_{z,2},\psi_a](a)d\no\\
&\quad= \big\{\cos(\alpha) + \sin(\alpha)[u_{z,2},\psi_a](a) \big\}d + \sin(\alpha)[u_{z,1},\psi_a](a)c\no\\
&\quad=0,\lb{4.32}
\end{align}
where the final equality follows from the second equation in \eqref{4.30}.  Therefore, the function $cu_{z,1}+du_{z,2}$ satisfies the boundary condition for functions in $\dom(T_{\alpha,\beta})$ at the endpoint $a$.  Since $cu_{z,1}+du_{z,2}$ belongs to $\dom(T_{\max})$, one concludes that
\begin{equation}
(cu_{z,1}+du_{z,2})\in \dom(T_{\alpha,\beta}).
	\lb{4.33}
\end{equation}
Since $u_{z,1}$ and $u_{z,2}$ are linearly independent and $|c|^2+|d|^2\neq 0$, the linear combination $cu_{z,1}+du_{z,2}$ is not the zero function.  Finally, \eqref{4.5} actually implies that $cu_{z,1}+du_{z,2}$ is an eigenfunction of $T_{\alpha,\beta}$ with eigenvalue $z$, a contradiction to the assumption that $z\in \rho(T_{\alpha,\beta})$.  This concludes the proof that $K_{\alpha,\beta}(z)$ is invertible.

To prove \eqref{4.23}, define the operator
\begin{align}
F_{\alpha,\beta}(z)=(T_{0,0}-zI_{(a,b)})^{-1} - \sum_{j,k=1}^2\big[K_{\alpha,\beta}(z)^{-1}\big]_{j,k}\langle u_{\ol{z},j},\,\cdot\,\rangle_{(a,b)} u_{z,k},&\\
\dom(F_{\alpha,\beta}(z))=L^2((a,b);r(x)\, dx).&\no
	\lb{4.34}
\end{align}
It suffices to show that
\begin{equation}
(T_{\alpha,\beta}-zI_{(a,b)})F_{\alpha,\beta}(z) = I_{(a,b)};
	\lb{4.35}
\end{equation}
that is, it suffices to show that for every $f\in L^2((a,b);r(x)\, dx)$,
\begin{equation}\lb{4.36}
F_{\alpha,\beta}(z)f\in \dom(T_{\alpha,\beta}),
\end{equation}
and
\begin{equation}\lb{4.37}
(T_{\alpha,\beta}-zI_{(a,b)})(F_{\alpha,\beta}(z)f)=f.
\end{equation}
To this end, let $f\in L^2((a,b);r(x)\, dx)$.  It is clear from the definition of $F_{\alpha,\beta}(z)$ that
\begin{equation}
F_{\alpha,\beta}(z)f\in \dom(T_{\max}),
	\lb{4.38}
\end{equation}
so the proof of \eqref{4.36} reduces to showing that $F_{\alpha,\beta}(z)f$ satisfies the boundary conditions in \eqref{2.25}; that is, it suffices to prove:
\begin{align}
\cos(\alpha)[F_{\alpha,\beta}(z)f,\phi_a](a) + \sin(\alpha)[F_{\alpha,\beta}(z)f,\psi_a](a)&=0,\lb{4.39}\\
\cos(\beta)[F_{\alpha,\beta}(z)f,\phi_b](b) + \sin(\beta)[F_{\alpha,\beta}(z)f,\psi_b](b)&=0.\lb{4.40}
\end{align}
To show \eqref{4.39}, one uses $\big[(T_{0,0}-zI_{(a,b)})^{-1}f,\phi_a\big](a)=0$ and \eqref{4.2} to compute
\begin{align}
&[F_{\alpha,\beta}(z)f,\phi_a](a)\no\\
&\quad = \big[(T_{0,0}-zI_{(a,b)})^{-1}f,\phi_a\big](a) - \sum_{j,k=1}^2\big[K_{\alpha,\beta}(z)^{-1}\big]_{j,k}\langle u_{\ol{z},j},f\rangle_{(a,b)}[u_{z,k},\phi_a](a)\no\\
&\quad= -\sum_{j=1}^2 \big[K_{\alpha,\beta}(z)^{-1}\big]_{j,2} \langle u_{\ol{z},j},f\rangle_{(a,b)}\no\\
&\quad= -[\det(K_{\alpha,\beta}(z))]^{-1}[u_{z,1},\psi_a](a)\langle u_{\ol{z},1},f\rangle_{(a,b)} \no\\
&\qquad- [\det(K_{\alpha,\beta}(z))]^{-1}\big\{ \cot(\beta) + [u_{z,1},\psi_b](b) \big\}\langle u_{\ol{z},2},f\rangle_{(a,b)}.\lb{4.42}
\end{align}
Moreover, using Lemma \ref{l4.2} and the explicit form of $K_{\alpha,\beta}(z)^{-1}$ obtained from \eqref{4.22}, one computes
\begin{align}
&[F_{\alpha,\beta}(z)f,\psi_a](a)\no\\
&\quad =\big[(T_{0,0}-zI_{(a,b)})^{-1}f,\psi_a\big](a) - \sum_{j,k=1}^2\big[K_{\alpha,\beta}(z)^{-1} \big]_{j,k}\langle u_{\ol{z},j},f\rangle_{(a,b)} [u_{z,k},\psi_a](a)\no\\
&\quad = -\langle u_{\ol{z},2},f\rangle_{(a,b)}\no\\
&\qquad+ [\det(K_{\alpha,\beta}(z))]^{-1}\big\{\cot(\alpha) + [u_{z,2},\psi_a](a) \big\}\langle u_{\ol{z},1},f\rangle_{(a,b)} [u_{z,1},\psi_a](a)\no\\
&\qquad - [\det(K_{\alpha,\beta}(z))]^{-1}[u_{z,1},\psi_a](a)\langle u_{\ol{z},1},f\rangle_{(a,b)} [u_{z,2},\psi_a](a)\no\\
&\qquad + [\det(K_{\alpha,\beta}(z))]^{-1}[u_{z,2},\psi_b](b)\langle u_{\ol{z},2},f \rangle_{(a,b)} [u_{z,1},\psi_a](a)\no\\
&\qquad - [\det(K_{\alpha,\beta}(z))]^{-1}\big\{\cot(\beta) + [u_{z,1},\psi_b](b) \big\}\langle u_{\ol{z},2},f \rangle_{(a,b)} [u_{z,2},\psi_a](a).
	\lb{4.43}
\end{align}
Therefore, upon combining \eqref{4.42} and \eqref{4.43}, one obtains
\begin{align}
&\det(K_{\alpha,\beta}(z))\big\{\cos(\alpha)[F_{\alpha,\beta}(z)f,\phi_a](a) + \sin(\alpha)[F_{\alpha,\beta}(z)f,\psi_a](a)\big\}\no\\
&\quad= -\langle u_{\ol{z},1},f \rangle_{(a,b)} \big\{ \cos(\alpha)[u_{z,1},\psi_a](a) - \cos(\alpha)[u_{z,1},\psi_a](a)\no\\
&\qquad\quad - \sin(\alpha)[u_{z,2},\psi_a](a)[u_{z,1},\psi_a](a) + \sin(\alpha)[u_{z,1},\psi_a](a)[u_{z,2},\psi_a](a)\big\}\no\\
&\qquad - \langle u_{\ol{z},2},f \rangle_{(a,b)} \big\{ \cos(\alpha)\cot(\beta) + \cos(\alpha)[u_{z,1},\psi_b](b) + \sin(\alpha)\det(K_{\alpha,\beta}(z))\no\\
&\qquad\quad- \sin(\alpha)[u_{z,2},\psi_b](b)[u_{z,1},\psi_a](a) + \sin(\alpha)\cot(\beta)[u_{z,2},\psi_a](a)\no\\
&\qquad\quad+ \sin(\alpha)[u_{z,1},\psi_b](b)[u_{z,2},\psi_a](a)\big\}.
	\lb{4.44}
\end{align}
By inspection, the expression in braces multiplying $\langle u_{\ol{z},1},f \rangle_{(a,b)}$ on the right-hand side of \eqref{4.44} vanishes.  Fully expanding $\det(K_{\alpha,\beta}(z))$ using \eqref{4.22}, one infers that the expression in braces multiplying $\langle u_{\ol{z},2},f \rangle_{(a,b)}$ on the right-hand side of \eqref{4.44} equals
\begin{align}
&\cos(\alpha)\cot(\beta) + \cos(\alpha)[u_{z,1},\psi_b](b) - \cos(\alpha)\cot(\beta) - \sin(\alpha)\cot(\beta)[u_{z,2},\psi_a](a)\no\\
&\quad - \cos(\alpha)[u_{z,1},\psi_b](b) - \sin(\alpha)[u_{z,1},\psi_b](b)[u_{z,2},\psi_a](a)\no\\
&\quad+\sin(\alpha)[u_{z,2},\psi_b](b)[u_{z,1},\psi_a](a)-\sin(\alpha)[u_{z,2},\psi_b](b)[u_{z,1},\psi_a](a)\no\\
&\quad + \sin(\alpha)\cot(\beta)[u_{z,2},\psi_a](a)+\sin(\alpha)[u_{z,1},\psi_b](b)[u_{z,2},\psi_a](a),
	\lb{4.45}
\end{align}
which, by inspection, also vanishes.  Consequently,
\begin{equation}
\det(K_{\alpha,\beta}(z))\big\{ \cos(\alpha)[F_{\alpha,\beta}(z)f,\phi_a](a) + \sin(\alpha)[F_{\alpha,\beta}(z)f,\psi_a](a)\big\}=0,
	\lb{4.46}
\end{equation}
and since $K_{\alpha,\beta}(z)$ is invertible, \eqref{4.39} follows.  To prove \eqref{4.40}, one proceeds in a manner analogous to the proof of \eqref{4.39} and computes
\begin{align}
&[F_{\alpha,\beta}(z)f,\phi_b](b)\no\\
&\quad= \big[(T_{0,0}-zI_{(a,b)})^{-1}f,\phi_b\big](b) - \sum_{j,k=1}^2 \big[K_{\alpha,\beta}(z)^{-1}\big]_{j,k}\langle u_{\ol{z},j},f \rangle_{(a,b)} [u_{z,k},\phi_b](b)\no\\
&\quad= -\sum_{j=1}^2 \big[K_{\alpha,\beta}(z)^{-1}\big]_{j,1}\langle u_{\ol{z},j},f \rangle_{(a,b)}\no\\
&\quad=  [\det(K_{\alpha,\beta}(z))]^{-1}\big\{\cot(\alpha)+[u_{z,2},\psi_a](a)\big\}\langle u_{\ol{z},1},f \rangle_{(a,b)}\no\\
&\qquad+ [\det(K_{\alpha,\beta}(z))]^{-1}[u_{z,2},\psi_b](b)\langle u_{\ol{z},2},f\rangle_{(a,b)}\lb{4.47}
\end{align}
and
\begin{align}
&[F_{\alpha,\beta}(z)f,\psi_b](b)\no\\
&\quad=\big[(T_{0,0}-zI_{(a,b)})^{-1}f,\psi_b\big](b) - \sum_{j,k=1}^2 \big[K_{\alpha,\beta}(z)^{-1}\big]_{j,k}\langle u_{\ol{z},j},f\rangle_{(a,b)} [u_{z,k},\psi_b](b)\no\\
&\quad=\langle u_{\ol{z},1},f\rangle_{(a,b)} \no\\
&\qquad + [\det(K_{\alpha,\beta}(z))]^{-1}\big\{\cot(\alpha) + [u_{z,2},\psi_a](a) \big\}\langle u_{\ol{z},1},f\rangle_{(a,b)} [u_{z,1},\psi_b](b)\no\\
&\qquad - [\det(K_{\alpha,\beta}(z))]^{-1} [u_{z,1},\psi_a](a)\langle u_{\ol{z},1},f\rangle_{(a,b)} [u_{z,2},\psi_b](b)\no\\
&\qquad + [\det(K_{\alpha,\beta}(z))]^{-1}[u_{z,2},\psi_b](b)\langle u_{\ol{z},2},f\rangle_{(a,b)} [u_{z,1},\psi_b](b)\no\\
&\qquad - [\det(K_{\alpha,\beta}(z))]^{-1}\big\{\cot(\beta) + [u_{z,1},\psi_b](b) \big\}\langle u_{\ol{z},2},f\rangle_{(a,b)} [u_{z,2},\psi_b](b).\lb{4.48}
\end{align}
Upon combining \eqref{4.47} and \eqref{4.48}, one infers
\begin{align}
&\det(K_{\alpha,\beta}(z))\big\{ \cos(\beta)[F_{\alpha,\beta}(z)f,\phi_b](b)+\sin(\beta)[F_{\alpha,\beta}(z)f,\psi_b](b)\big\}\no\\
&\quad=-\langle u_{\ol{z},1},f\rangle_{(a,b)}\big\{ -\cos(\beta)\cot(\alpha) - \cos(\beta)[u_{z,2},\psi_a](a)-\sin(\beta)\det(K_{\alpha,\beta}(z))\no\\
&\qquad\quad - \sin(\beta)\cot(\alpha)[u_{z,1},\psi_b](b)-\sin(\beta)[u_{z,2},\psi_a](a)[u_{z,1},\psi_b](b)\no\\
&\qquad \quad+ \sin(\beta)[u_{z,1},\psi_a](a)[u_{z,2},\psi_b](b)\big\}\no\\
&\qquad - \langle u_{\ol{z},2},f\rangle_{(a,b)} \big\{-\cos(\beta)[u_{z,2},\psi_b](b) - \sin(\beta)[u_{z,2},\psi_b](b)[u_{z,1},\psi_b](b)\no\\
&\qquad \quad + \cos(\beta)[u_{z,2},\psi_b](b) + \sin(\beta)[u_{z,1},\psi_b](b)[u_{z,2},\psi_b](b)\big\}.\lb{4.49}
\end{align}
By inspection, the expression in braces multiplying $\langle u_{\ol{z},2},f\rangle_{(a,b)}$ on the right-hand side of \eqref{4.49} vanishes.  By fully expanding $\det(K_{\alpha,\beta}(z))$ using \eqref{4.22}, one infers that the expression in braces multiplying $\langle u_{\ol{z},1},f \rangle_{(a,b)}$ on the right-hand side of \eqref{4.49} equals
\begin{align}
&-\cos(\beta)\cot(\alpha)-\cos(\beta)[u_{z,2},\psi_a](a) + \cos(\beta)\cot(\alpha) + \cos(\beta)[u_{z,2},\psi_a](a)\no\\
&\quad + \sin(\beta)\cot(\alpha)[u_{z,1},\psi_b](b) + \sin(\beta)[u_{z,1},\psi_b](b)[u_{z,2},\psi_a](a)\no\\
&\quad - \sin(\beta)[u_{z,2},\psi_b](b)[u_{z,1},\psi_a](a) - \sin(\beta)\cot(\alpha)[u_{z,1},\psi_b](b)\no\\
&\quad - \sin(\beta)[u_{z,2},\psi_a](a)[u_{z,1},\psi_b](b) + \sin(\beta)[u_{z,1},\psi_a](a)[u_{z,2},\psi_b](b),
	\lb{4.50}
\end{align}
which, by inspection, also vanishes.  Consequently,
\begin{equation}
\det(K_{\alpha,\beta}(z))\big\{ \cos(\beta)[F_{\alpha,\beta}(z)f,\phi_b](b)+\sin(\beta)[F_{\alpha,\beta}(z)f,\psi_b](b)\big\}=0
	\lb{4.51}
\end{equation}
and since $K_{\alpha,\beta}(z)$ is invertible, \eqref{4.40} follows.  This completes the proof of \eqref{4.36}, and it remains to prove \eqref{4.37}.  The proof of \eqref{4.37} is a simple calculation which combines \eqref{4.5} and the fact that $T_{\max}$ is an extension of both $T_{\alpha,\beta}$ and $T_{0,0}$:
\begin{align}
(T_{\alpha,\beta}-zI_{(a,b)})F_{\alpha,\beta}(z)f&= (T_{\max}-zI_{(a,b)})F_{\alpha,\beta}(z)f\lb{4.52}\\
&= (T_{\max}-zI_{(a,b)})(T_{0,0}-zI_{(a,b)})^{-1}f\no\\
&\quad - \sum_{j,k=1}^2\big[K_{\alpha,\beta}(z)^{-1}\big]_{j,k}\langle u_{\ol{z},j},f\rangle_{(a,b)} (T_{\max}-zI_{(a,b)}) u_{z,k}\no\\
&= (T_{0,0}-zI_{(a,b)})(T_{0,0}-zI_{(a,b)})^{-1}f= I_{(a,b)}f= f.\no
\end{align}
\end{proof}

If $\alpha,\beta\in [0,\pi)$ and $\alpha\beta=0$, then $T_{\alpha,\beta}$ and $T_{0,0}$ are no longer relatively prime with respect to $T_{\min}$.  In this case, we obtain:

\begin{theorem}\lb{t4.5}
Assume Hypothesis \ref{h4.1}.  The following statements $(i)$ and $(ii)$ hold.\\[1mm]
$(i)$  If $\beta\in (0,\pi)$, then the maximal common part of $T_{0,0}$ and $T_{0,\beta}$ is the restriction of $T_{\max}$ to the set
\begin{equation}
\cS_1=\{y\in \dom(T_{\max})\,|\, [y,\phi_a](a) = [y,\phi_b](b) = [y,\psi_b](b)=0\}.
	\lb{4.53}
\end{equation}
Moreover, for each $z\in \rho(T_{0,0})\cap\rho(T_{0,\beta})$ the scalar
\begin{equation}\lb{4.54}
K_{0,\beta}(z)=-\cot(\beta)-[u_{z,1},\psi_b](b)
\end{equation}
is nonzero and
\begin{equation}\lb{4.55}
(T_{0,\beta}-zI_{(a,b)})^{-1} = (T_{0,0}-zI_{(a,b)})^{-1} + K_{0,\beta}(z)^{-1}\langle u_{\ol{z},1},\,\cdot\,\rangle_{(a,b)} u_{z,1}.
\end{equation}
$(ii)$  If $\alpha\in (0,\pi)$, then the maximal common part of $T_{0,0}$ and $T_{\alpha,0}$ is the restriction of $T_{\max}$ to the set
\begin{equation}
\cS_2=\{y\in \dom(T_{\max})\,|\, [y,\phi_a](a)=[y,\phi_b](b)=[y,\psi_a](a)=0\}.
	\lb{4.56}
\end{equation}
Moreover, for each $z\in \rho(T_{0,0})\cap\rho(T_{\alpha,0})$ the scalar
\begin{equation}
K_{\alpha,0}(z)=\cot(\alpha)+[u_{z,2},\psi_a](a)
	\lb{4.57}
\end{equation}
is nonzero and
\begin{equation}
(T_{\alpha,0}-zI_{(a,b)})^{-1} = (T_{0,0}-zI_{(a,b)})^{-1} + K_{\alpha,0}(z)^{-1}\langle u_{\ol{z},2},\,\cdot\,\rangle_{(a,b)} u_{z,2}.
	\lb{4.58}
\end{equation}
\end{theorem}
\begin{proof}
We provide the details of the proof of item $(i)$ only.  The proof of item $(ii)$ is entirely analogous.  Let $\beta\in (0,\pi)$ be fixed.  To prove that the maximal common part of $T_{0,0}$ and $T_{0,\beta}$ is the restriction of $T_{\max}$ to the set $\cS_1$, it suffices to show
\begin{equation}
\dom(T_{0,0})\cap\dom(T_{0,\beta}) = \cS_1.
	\lb{4.59}
\end{equation}
To this end, let $y\in \dom(T_{0,0})\cap\dom(T_{0,\beta})$.  Then the fact that $y\in \dom(T_{0,0})$ implies that $y$ satisfies the conditions in \eqref{4.25}, and the fact that $y\in \dom(T_{0,\beta})$ implies
\begin{equation}\lb{4.60}
\cos(\beta)[y,\phi_b](b) + \sin(\beta)[y,\psi_b](b)=0.
\end{equation}
Taken together, the relations in \eqref{4.25} and \eqref{4.60} imply $[y,\psi_b](b)=0$ since $\sin(\beta)\neq 0$ for $\beta \in (0,\pi)$.  Hence, $y\in \cS_1$.  Conversely, if $y\in \cS_1$, then \eqref{4.25} and \eqref{4.60} hold, so $y\in \dom(T_{0,0})\cap \dom(T_{0,\beta})$.  To complete the proof of item $(i)$, let $z\in \rho(T_{0,0})\cap \rho(T_{0,\beta})$ be fixed and let $K_{0,\beta}(z)$ be the scalar defined in \eqref{4.54}.  To prove the claim that $K_{0,\beta}(z)$ is nonzero, suppose on the contrary that $K_{0,\beta}(z)=0$.  We claim that $z$ is then an eigenvalue of $T_{0,\beta}$.  Indeed, $K_{0,\beta}(z)=0$ implies
\begin{equation}
\cos(\beta)[u_{z,1},\phi_b](b) + \sin(\beta)[u_{z,1},\psi_b](b)=0.
	\lb{4.61}
\end{equation}
Since $[u_{z,1},\phi_a](a)=0$ (cf.~\eqref{4.2}) and $u_{z,1}\in \dom(T_{\max})$, it follows that $u_{z,1}\in \dom(T_{0,\beta})$, so that $z$ is an eigenvalue of $T_{0,\beta}$ and $u_{z,1}$ is a corresponding eigenfunction.  This contradicts the assumption that $z\in \rho(T_{0,\beta})$ and completes the proof that $K_{0,\beta}(z)\neq 0$.  It remains to establish the resolvent identity in \eqref{4.55}.  Define
\begin{equation}\lb{4.62}
\begin{split}
F_{0,\beta}(z)=(T_{0,0}-zI_{(a,b)})^{-1} + K_{0,\beta}(z)^{-1}\langle u_{\ol{z},1},\,\cdot\,\rangle_{(a,b)} u_{z,1},&\\
\dom(F_{0,\beta}(z))=L^2((a,b);r(x)\, dx).&
\end{split}
\end{equation}
In order to prove \eqref{4.55}, it suffices to show
\begin{equation}
(T_{0,\beta}-zI_{(a,b)})F_{0,\beta}(z) = I_{(a,b)};
	\lb{4.63}
\end{equation}
that is, it suffices to show that for every $f\in L^2((a,b);r(x)\, dx)$,
\begin{equation}\lb{4.64}
F_{0,\beta}(z)f\in \dom(T_{0,\beta}),
\end{equation}
and
\begin{equation}\lb{4.65}
(T_{0,\beta}-zI_{(a,b)})(F_{0,\beta}(z)f)=f.
\end{equation}
To this end, let $f\in L^2((a,b);r(x)\, dx)$ be arbitrary.  It is clear from the definition of $F_{0,\beta}(z)$ that
\begin{equation}
F_{0,\beta}(z)f\in \dom(T_{\max}),
	\lb{4.66}
\end{equation}
so the proof of \eqref{4.64} reduces to showing that $F_{0,\beta}(z)f$ satisfies the boundary conditions for functions in $\dom(T_{0,\beta})$; that is, it suffices to prove
\begin{equation}\lb{4.67}
\begin{split}
[F_{0,\beta}(z)f,\phi_a](a)=0,\\
\cos(\beta)[F_{0,\beta}(z)f,\phi_b](b) + \sin(\beta)[F_{0,\beta}(z)f,\psi_b](b)=0.
\end{split}
\end{equation}
To check the first boundary condition in \eqref{4.67}, one uses \eqref{4.2} and \eqref{4.27} to compute
\begin{align}
[F_{0,\beta}(z)f,\phi_a](a) &= \big[(T_{0,0}-zI_{(a,b)})^{-1}f,\phi_a\big](a)\no\\
&\quad + K_{0,\beta}(z)^{-1}\langle u_{\ol{z},1},f\rangle_{(a,b)}[u_{z,1},\phi_a](a) = 0.
	\lb{4.68}
\end{align}
To check the second boundary condition in \eqref{4.67}, one computes
\begin{align}
&\cos(\beta)[F_{0,\beta}(z)f,\phi_b](b) + \sin(\beta)[F_{0,\beta}(z)f,\psi_b](b)\no\\
&\quad = \cos(\beta)\big\{ \big[(T_{0,0}-zI_{(a,b)})^{-1}f,\phi_b\big](b) + K_{0,\beta}(z)^{-1}\langle u_{\ol{z},1},f\rangle_{(a,b)} [u_{z,1},\phi_b](b)\big\}\no\\
&\qquad + \sin(\beta)\big\{ \big[(T_{0,0}-zI_{(a,b)})^{-1}f,\psi_b\big](b)+K_{0,\beta}(z)^{-1}\langle u_{\ol{z},1},f\rangle_{(a,b)} [u_{z,1},\psi_b](b)\big\}\no\\
&\quad = \cos(\beta)K_{0,\beta}(z)^{-1}\langle u_{\ol{z},1},f\rangle_{(a,b)}
+\sin(\beta)\langle u_{\ol{z},1},f\rangle_{(a,b)}\no\\
&\qquad+ \sin(\beta)K_{0,\beta}(z)^{-1}\langle u_{\ol{z},1},f\rangle_{(a,b)} [u_{z,1},\psi_b](b)\no\\
&\quad = \langle u_{\ol{z},1},f\rangle_{(a,b)}\big\{\cos(\beta)K_{0,\beta}(z)^{-1}+\sin(\beta)+\sin(\beta)K_{0,\beta}(z)^{-1}[u_{z,1},\psi_b](b)\big\}\no\\
&\quad = \langle u_{\ol{z},1},f\rangle_{(a,b)}\, K_{0,\beta}(z)^{-1}\big\{ \cos(\beta) + \sin(\beta)K_{0,\beta}(z)+\sin(\beta)[u_{z,1},\psi_b](b)\big\}\no\\
&\quad = \langle u_{\ol{z},1},f\rangle_{(a,b)}\, K_{0,\beta}(z)^{-1}\big\{\cos(\beta)-\cos(\beta)-\sin(\beta)[u_{z,1},\psi_b](b)\no\\
&\qquad +\sin(\beta)[u_{z,1},\psi_b](b)\big\}\no\\
&\quad = 0.\lb{4.69}
\end{align}
Note that the second identity in \eqref{4.7} is used to obtain the second equality in \eqref{4.69}.  This proves \eqref{4.64}, and subsequently, the claim in \eqref{4.65} is a result of the following calculation:
\begin{align}
(T_{0,\beta}-zI_{(a,b)})F_{0,\beta}(z)f&= (T_{\max}-zI_{(a,b)})(F_{0,\beta}(z)f)\lb{4.70}\\
&= (T_{\max}-zI_{(a,b)})(T_{0,0}-zI_{(a,b)})^{-1}f\no\\
&\quad+ K_{0,\beta}(z)^{-1}\langle u_{\ol{z},1},f\rangle_{(a,b)} (T_{\max}-zI_{(a,b)})u_{z,1}\no\\
&= (T_{0,0}-zI_{(a,b)})(T_{0,0}-zI_{(a,b)})^{-1}f= I_{(a,b)}f= f.\no
\end{align}
\end{proof}

Now, we derive results analogous to Theorems \ref{t4.4} and \ref{t4.5} for coupled boundary conditions.  Again, we separate the case in which $T_{R,\eta}$ and $T_{0,0}$ are relatively prime with respect to $T_{\min}$ from the rest.  The first is:

\begin{theorem}	\lb{t4.6}
Assume Hypothesis \ref{h4.1}.  If $R_{1,2}\neq 0$, then $T_{0,0}$ and $T_{R,\eta}$ are relatively prime with respect to $T_{\min}$.  Moreover, for each $z\in \rho(T_{0,0})\cap\rho(T_{R,\eta})$ the matrix
\begin{equation}\lb{4.71}
K_{R,\eta}(z) = \begin{pmatrix}
\dfrac{R_{2,2}}{R_{1,2}}-[u_{z,1},\psi_b](b) & -\dfrac{e^{-i\eta}}{R_{1,2}}+[u_{z,1},\psi_a](a)\\[4mm]
-\dfrac{e^{i\eta}}{R_{1,2}}-[u_{z,2},\psi_b](b) & \dfrac{R_{1,1}}{R_{1,2}}+[u_{z,2},\psi_a](a)
\end{pmatrix}
\end{equation}
is invertible and
\begin{equation}\lb{4.72}
(T_{R,\eta}-zI_{(a,b)})^{-1} = (T_{0,0}-zI_{(a,b)})^{-1} + \sum_{j,k=1}^2\big[K_{R,\eta}(z)^{-1}\big]_{j,k}\langle u_{\ol{z},j},\,\cdot\,\rangle_{(a,b)} u_{z,k}.
\end{equation}
\end{theorem}
\begin{proof}
Suppose that $R_{1,2}\neq 0$.  In order to show $T_{0,0}$ and $T_{R,\eta}$ are relatively prime with respect to $T_{\min}$, it suffices to show
\begin{equation}\lb{4.73}
\dom(T_{0,0})\cap \dom(T_{R,\eta}) \subseteq \dom(T_{\min}).
\end{equation}
Note that any $y\in \dom(T_{R,\eta})$ satisfies the boundary conditions
\begin{align}
[y,\phi_b](b) &= e^{i\eta}R_{1,1}[y,\phi_a](a)+e^{i\eta}R_{1,2}[y,\psi_a](a),\lb{4.74}\\
[y,\psi_b](b) &= e^{i\eta}R_{2,1}[y,\phi_a](a) + e^{i\eta}R_{2,2}[y,\psi_a](a).\lb{4.75}
\end{align}
Now, to prove \eqref{4.73}, let $y\in \dom(T_{0,0})\cap\dom(T_{R,\eta})$.  Then \eqref{4.25} and \eqref{4.74} imply
\begin{equation}\lb{4.76}
[y,\psi_a](a) = \frac{e^{-i\eta}}{R_{1,2}}\big\{[y,\phi_b](b)-e^{i\eta}R_{1,1}[y,\phi_a](a) \big\} = 0,
\end{equation}
and \eqref{4.25}, \eqref{4.75}, and \eqref{4.76} imply
\begin{equation}\lb{4.77}
[y,\psi_b](b) = e^{i\eta}R_{2,1}[y,\phi_a](a) + e^{i\eta}R_{2,2}[y,\psi_a](a) = 0.
\end{equation}
Therefore,
\begin{equation}\lb{4.78}
[y,\phi_a](a)=[y,\psi_a](a)=[y,\phi_b](b)=[y,\psi_b](b)=0,
\end{equation}
and it follows that $y\in \dom(T_{\min})$.  Thus, the containment in \eqref{4.73} holds.  This concludes the proof that $T_{0,0}$ and $T_{R,\eta}$ are relatively prime with respect to $T_{\min}$.

Next, for $z\in \rho(T_{0,0})\cap \rho(T_{R,\eta})$, we prove that the matrix $K_{R,\eta}(z)$ defined by \eqref{4.71} is invertible.  Suppose, by way of contradiction, that $K_{R,\eta}(z)$ is singular.  Then $e^{i\eta}R_{1,2}K_{R,\eta}(z)$ is a singular matrix, so its rows are linearly dependent: for some $\alpha\in \bbC$,
\begin{align}
e^{i\eta}R_{2,2} - e^{i\eta}R_{1,2}[u_{z,1},\psi_b](b) &= \alpha\big\{-e^{2i\eta}-e^{i\eta}R_{1,2}[u_{z,2},\psi_b](b)\big\},\lb{4.79}\\
-1 + e^{i\eta}R_{1,2}[u_{z,1},\psi_a](a) &= \alpha\big\{e^{i\eta}R_{1,1}+e^{i\eta}R_{1,2}[u_{z,2},\psi_a](a)\big\}.\lb{4.80}
\end{align}
The equality in \eqref{4.80} may be recast as
\begin{equation}\lb{4.81}
-1=\alpha e^{i\eta}R_{1,1} + \big\{\alpha[u_{z,2},\psi_a](a) - [u_{z,1},\psi_a](a)\big\}e^{i\eta} R_{1,2}.
\end{equation}
Define the function
\begin{equation}\lb{4.82}
g_z = \alpha u_{z,2} - u_{z,1},
\end{equation}
so that
\begin{align}
[g_z,\phi_a](a) &= \alpha[u_{z,2},\phi_a](a) - [u_{z,1},\phi_a](a) = \alpha,\lb{4.83}\\
[g_z,\phi_b](b) &= \alpha[u_{z,2},\phi_b](b) - [u_{z,1},\phi_b](b) = -1.\lb{4.84}
\end{align}
Note that by applying \eqref{4.83} and \eqref{4.84}, the identity in \eqref{4.81} may be recast in terms of $g_z$:
\begin{equation}\lb{4.85}
[g_z,\phi_b](b) = e^{i\eta}R_{1,1}[g_z,\phi_a](a) + e^{i\eta}R_{1,2}[g_z,\psi_a](a).
\end{equation}
In addition, by the definition of $g_z$,
\begin{equation}\lb{4.86}
[g_z,\psi_b](b) = \alpha[u_{z,2},\psi_b](b) - [u_{z,1},\psi_b](b).
\end{equation}
Therefore, by \eqref{4.79},
\begin{align}
R_{2,2} - R_{1,2}[u_{z,1},\psi_b](b) = -\alpha e^{i\eta} - \alpha R_{1,2}[u_{z,2},\psi_b](b),\lb{4.87}
\end{align}
which may be rewritten as
\begin{equation}\lb{4.88}
-\alpha[u_{z,2},\psi_b](b) + [u_{z,1},\psi_b](b) = \alpha \frac{e^{i\eta}}{R_{1,2}} + \frac{R_{2,2}}{R_{1,2}}.
\end{equation}
Thus,
\begin{align}
[g_z,\psi_b](b) &= -\alpha \frac{e^{i\eta}}{R_{1,2}} - \frac{R_{2,2}}{R_{1,2}}\no\\
&= -\alpha \frac{e^{i\eta}}{R_{1,2}} - \frac{R_{2,2}}{R_{1,2}} + e^{i\eta}R_{2,1}\alpha - e^{i\eta}R_{2,1}\alpha\no\\
&= e^{i\eta}R_{2,1}\alpha - \alpha\frac{e^{i\eta}}{R_{1,2}} - \frac{R_{2,2}}{R_{1,2}} - \alpha e^{i\eta}\bigg(\frac{-1 + R_{1,1}R_{2,2}}{R_{1,2}} \bigg)\no\\
&= e^{i\eta}R_{2,1}\alpha - \frac{R_{2,2}}{R_{1,2}} - \alpha e^{i\eta} \frac{R_{1,1}R_{2,2}}{R_{1,2}}\no\\
&= e^{i\eta}R_{2,1}\alpha + e^{i\eta}R_{2,2}\bigg(\frac{-1}{e^{i\eta}R_{1,2}} - \alpha \frac{R_{1,1}}{R_{1,2}} \bigg)\no\\
&= e^{i\eta}R_{2,1}[g_z,\phi_a](a) + e^{i\eta}R_{2,2}[g_z,\psi_a](a).\lb{4.89}
\end{align}
To obtain \eqref{4.89}, one uses \eqref{4.81}, which implies
\begin{equation}\lb{4.90}
-1 = \alpha e^{i\eta}R_{1,1} + [g_z,\psi_a](a)e^{i\eta}R_{1,2},
\end{equation}
so that
\begin{align}\lb{4.91}
[g_z,\psi_a](a) &= \frac{e^{-i\eta}}{R_{1,2}}\big(-1-\alpha e^{i\eta}R_{1,1} \big)= \frac{-1}{e^{i\eta}R_{1,2}} - \alpha \frac{R_{1,1}}{R_{1,2}}.
\end{align}
Hence, $g_z\in\dom(T_{\max})$ satisfies $T_{\max}g_z = zg_z$ and
\begin{align}
\begin{pmatrix}
[g_z,\phi_b](b)\\
[g_z,\psi_b](b)
\end{pmatrix}=
e^{i\eta}R
\begin{pmatrix}
[g_z,\phi_a](a)\\
[g_z,\psi_a](a)
\end{pmatrix}.\lb{4.92}
\end{align}
In light of \eqref{4.92}, one infers that $g_z\in \dom(T_{R,\eta})$, and it follows that $z$ is an eigenvalue of $T_{R,\eta}$, which is a contradiction to the assumption that $z\in \rho(T_{R,\eta})$.  Therefore, the matrix $K_{R,\eta}(z)$ must be invertible.

In order to complete the proof, it remains to establish the resolvent identity in \eqref{4.72}.  To prove \eqref{4.72}, let $z\in \rho(T_{0,0})\cap\rho(T_{R,\eta})$, and define the operator
\begin{equation}\lb{4.93}
\begin{split}
F_{R,\eta}(z)=(T_{0,0}-zI_{(a,b)})^{-1} + \sum_{j,k=1}^2\big[K_{R,\eta}(z)^{-1}\big]_{j,k}\langle u_{\ol{z},j},\,\cdot\,\rangle_{(a,b)} u_{z,k},&\\
\dom(F_{R,\eta}(z))=L^2((a,b);r(x)\,dx).&
\end{split}
\end{equation}
It suffices to show that
\begin{equation}\lb{4.94}
(T_{R,\eta}-zI_{(a,b)})F_{R,\eta}(z) = I_{(a,b)};
\end{equation}
that is, it suffices to show that for every $f\in L^2((a,b);r(x)\,dx)$,
\begin{equation}\lb{4.95}
F_{R,\eta}(z)f\in \dom(T_{R,\eta}),
\end{equation}
and
\begin{equation}\lb{4.96}
(T_{R,\eta}-zI_{(a,b)})F_{R,\eta}(z)f=f.
\end{equation}
To this end, let $f\in L^2((a,b);r(x)\,dx)$.  It is clear from the definition of $F_{R,\eta}(z)$ that
\begin{equation}\lb{4.97}
F_{R,\eta}(z)f\in \dom(T_{\max}),
\end{equation}
so the proof of \eqref{4.95} reduces to showing $F_{R,\eta}(z)f$ satisfies the boundary conditions in \eqref{2.21}; that is,
\begin{align}
R_{1,1}[F_{R,\eta}(z)f,\phi_a](a) + R_{1,2}[F_{R,\eta}(z)f,\psi_a](a) - e^{-i\eta}[F_{R,\eta}(z)f,\phi_b](b)&=0,\lb{4.98}\\
R_{2,1}[F_{R,\eta}(z)f,\phi_a](a) + R_{2,2}[F_{R,\eta}(z)f,\psi_a](a) - e^{-i\eta}[F_{R,\eta}(z)f,\psi_b](b)&=0.\lb{4.99}
\end{align}
To begin, one computes the product of the left-hand side of \eqref{4.98} with the factor $R_{1,2}\det(K_{R,\eta}(z))$ as follows:
\begin{align}
&\bigg\{ R_{1,1}\bigg(\big[(T_{0,0}-zI_{(a,b)})^{-1}f,\phi_a\big](a)+ \sum_{j,k=1}^2 \big[K_{R,\eta}(z)^{-1}\big]_{j,k}\langle u_{\ol{z},j},f\rangle_{(a,b)}[u_{z,k},\phi_a](a)\bigg)\no\\
&\qquad +R_{1,2}\bigg(\big[(T_{0,0}-zI_{(a,b)})^{-1}f,\psi_a\big](a) \no\\
&\hspace*{2.2cm}+ \sum_{j,k=1}^2 \big[K_{R,\eta}(z)^{-1}\big]_{j,k} \langle u_{\ol{z},j},f\rangle_{(a,b)} [u_{z,k},\psi_a](a)\bigg)\no\\
&\qquad -e^{-i\eta}\bigg(\big[(T_{0,0}-zI_{(a,b)})^{-1}f,\phi_b\big](b)\no\\
&\hspace*{2.2cm}+ \sum_{j,k=1}^2 \big[K_{R,\eta}(z)^{-1}\big]_{j,k} \langle u_{\ol{z},j},f\rangle_{(a,b)} [u_{z,k},\phi_b](b)\bigg)\bigg\}R_{1,2}\det(K_{R,\eta}(z))\no\\
&\quad = \Big[ R_{1,1}\big\{e^{-i\eta}-R_{1,2}[u_{z,1},\psi_a](a) \big\}+ R_{1,2}\big\{R_{1,1} + R_{1,2}[u_{z,2},\psi_a](a) \big\}[u_{z,1},\psi_a](a)\no\\
&\qquad\quad  + R_{1,2}\big\{e^{-i\eta} - R_{1,2}[u_{z,1},\psi_a](a) \big\}[u_{z,2},\psi_a](a)\no\\
&\qquad\quad  - e^{-i\eta}\big\{R_{1,1} + R_{1,2}[u_{z,2},\psi_a](a) \big\}\Big]\langle u_{\ol{z},1},f\rangle_{(a,b)}\no\\
&\qquad + \Big[ R_{1,1}\big\{R_{2,2} - R_{1,2} [u_{z,1},\psi_b](b) \big\} - R_{1,2}^2\det(K_{R,\eta}(z))\no\\
&\qquad\qquad + R_{1,2}\big\{e^{i\eta} + R_{1,2}[u_{z,2},\psi_b](b) \big\}[u_{z,1},\psi_a](a) - e^{-i\eta}\big\{e^{i\eta} + R_{1,2}[u_{z,2},\psi_b](b) \big\}\no\\
&\qquad\qquad +R_{1,2}\big\{R_{2,2} - R_{1,2}[u_{z,1},\psi_b](b) \big\}[u_{z,2},\psi_a](a)\Big]\langle u_{\ol{z},2},f\rangle_{(a,b)}.\lb{4.100}
\end{align}
The coefficient in square brackets which multiplies $\langle u_{\ol{z},1},f\rangle_{(a,b)}$ on the right-hand side in \eqref{4.100} equals
\begin{align}
&e^{-i\eta}R_{1,1} - R_{1,1}R_{1,2}[u_{z,1},\psi_a](a) + R_{1,2}R_{1,1}[u_{z,1},\psi_a](a)\no\\
&\quad + R_{1,2}^2[u_{z,2},\psi_a](a)[u_{z,1},\psi_a](a) + e^{-i\eta}R_{1,2}[u_{z,2},\psi_a](a)\no\\
&\quad - R_{1,2}^2[u_{z,1},\psi_a](a)[u_{z,2},\psi_a](a)- e^{-i\eta}R_{1,1} - e^{-i\eta}R_{1,2}[u_{z,2},\psi_a](a),\lb{4.101}
\end{align}
which vanishes by inspection.  Using the definition of $K_{R,\eta}(z)$ in \eqref{4.71} to compute $\det(K_{R,\eta}(z))$, the coefficient in square brackets which multiplies $\langle u_{\ol{z},2},f\rangle_{(a,b)}$ on the right-hand side in \eqref{4.100} equals
\begin{align}
&R_{1,1}R_{2,2} - R_{1,1}R_{1,2}[u_{z,1},\psi_b](b) - R_{2,2}R_{1,1} - R_{1,2}R_{2,2}[u_{z,2},\psi_a](a)\no\\
&\quad + R_{1,2}R_{1,1}[u_{z,1},\psi_b](b) + R_{1,2}^2[u_{z,1},\psi_b](b)[u_{z,2},\psi_a](a) + 1 - R_{1,2}e^{i\eta}[u_{z,1},\psi_a](a)\no\\
&\quad + R_{1,2}e^{-i\eta}[u_{z,2},\psi_b](b) - R_{1,2}^2[u_{z,2},\psi_b](b)[u_{z,1},\psi_a](a) + R_{1,2}e^{i\eta}[u_{z,1},\psi_a](a)\no\\
&\quad + R_{1,2}^2[u_{z,2},\psi_b](b)[u_{z,1},\psi_a](a) + R_{1,2}R_{2,2}[u_{z,2},\psi_a](a)\no\\
&\quad - R_{1,2}^2[u_{z,1},\psi_b](b)[u_{z,2},\psi_a](a) - 1 - e^{-i\eta}R_{1,2}[u_{z,2},\psi_b](b),\lb{4.102}
\end{align}
which also vanishes by inspection.  Therefore,
\begin{align}
&R_{1,2}\det(K_{R,\eta}(z))\big\{R_{1,1}[F_{R,\eta}(z)f,\phi_a](a) + R_{1,2}[F_{R,\eta}(z)f,\psi_a](a)\no\\
&\hspace*{6.1cm}- e^{-i\eta}[F_{R,\eta}(z)f,\psi_b](b)\big\}=0,\lb{4.103}
\end{align}
and since neither $R_{1,2}$ nor $\det(K_{R,\eta}(z))$ is zero, the boundary condition in \eqref{4.98} is satisfied.

Next, we compute the product of the left-hand side of \eqref{4.99} with the factor $R_{1,2}\det(K_{R,\eta}(z))$ as follows:
\begin{align}
&\bigg\{R_{2,1}\bigg(\big[(T_{0,0}-zI_{(a,b)})^{-1}f,\phi_a\big](a) + \sum_{j,k=1}^2\big[K_{R,\eta}(z)^{-1}\big]_{j,k}\langle u_{\ol{z},j},f\rangle_{(a,b)} [u_{z,k},\phi_a](a)\bigg)\no\\
&\qquad + R_{2,2}\bigg(\big[(T_{0,0}-zI_{(a,b)})^{-1}f,\psi_a\big](a)\no\\
&\hspace*{2.2cm}+ \sum_{j,k=1}^2 \big[K_{R,\eta}(z)^{-1}\big]_{j,k}\langle u_{\ol{z},j},f\rangle_{(a,b)} [u_{z,k},\psi_a](a)\bigg)\no\\
&\qquad - e^{-i\eta}\bigg(\big[(T_{0,0}-zI_{(a,b)})^{-1}f,\psi_b\big](b)\no\\
&\hspace*{2.2cm}+ \sum_{j,k=1}^2\big[K_{R,\eta}(z)^{-1}\big]_{j,k} \langle u_{\ol{z},j},f\rangle_{(a,b)}[u_{z,k},\psi_b](b)\bigg)\bigg\}R_{1,2}\det(K_{R,\eta}(z))\no\\
&\quad = \Big[e^{-i\eta}R_{2,1} - R_{2,1}R_{1,2}[u_{z,1},\psi_a](a) + R_{1,1}R_{2,2}[u_{z,1},\psi_a](a)\no\\
&\qquad\quad + R_{2,2}R_{1,2}[u_{z,2},\psi_a](a)[u_{z,1},\psi_a](a) + e^{-i\eta}R_{2,2}[u_{z,2},\psi_a](a)\no\\
&\qquad\quad - R_{2,2}R_{1,2}[u_{z,1},\psi_a](a)[u_{z,2},\psi_a](a) - e^{-i\eta}R_{1,2}\det(K_{R,\eta}(z))\no\\
&\qquad\quad - e^{-i\eta}R_{1,1}[u_{z,1},\psi_b](b) - e^{-i\eta}R_{1,2}[u_{z,2},\psi_a](a)[u_{z,1},\psi_b](b)\no\\
&\qquad\quad - e^{-2i\eta}[u_{z,2},\psi_b](b) + e^{-i\eta}R_{1,2}[u_{z,1},\psi_a](a)[u_{z,2},\psi_b](b)\Big]\langle u_{\ol{z},1},f\rangle_{(a,b)}\no\\
&\qquad + \Big[R_{2,1}R_{2,2} - R_{2,1}R_{1,2}[u_{z,1},\psi_b](b) - R_{2,2}R_{1,2}\det(K_{R,\eta}(z))\no\\
&\qquad \qquad + e^{i\eta}R_{2,2}[u_{z,1},\psi_a](a) + R_{2,2}R_{1,2}[u_{z,2},\psi_b](b)[u_{z,1},\psi_a](a)\no\\
&\qquad \qquad + R_{2,2}^2[u_{z,2},\psi_a](a) - R_{2,2}R_{1,2}[u_{z,1},\psi_b](b)[u_{z,2},\psi_a](a) - [u_{z,1},\psi_b](b)\no\\
&\qquad \qquad - e^{-i\eta}R_{1,2}[u_{z,2},\psi_b](b)[u_{z,1},\psi_b](b) - e^{-i\eta}R_{2,2}[u_{z,2},\psi_b](b)\no\\
&\qquad \qquad + e^{-i\eta}R_{1,2}[u_{z,1},\psi_b](b)[u_{z,2},\psi_b](b)\Big] \langle u_{\ol{z},2},f\rangle_{(a,b)}.\lb{4.104}
\end{align}
The coefficient in square brackets which multiplies $\langle u_{\ol{z},1},f\rangle_{(a,b)}$ on the right-hand side of \eqref{4.104} equals
\begin{align}
&e^{-i\eta}R_{2,1} - R_{2,1}R_{1,2}[u_{z,1},\psi_a](a)\no\\
&\qquad + R_{1,1}R_{2,2}[u_{z,1},\psi_a](a)
+ R_{2,2}R_{1,2}[u_{z,2},\psi_a](a)[u_{z,1},\psi_a](a)\no\\
&\qquad+ e^{-i\eta}R_{2,2}[u_{z,2},\psi_a](a) - R_{2,2}R_{1,2}[u_{z,1},\psi_a](a)[u_{z,2},\psi_a](a)\no\\
&\qquad - e^{-i\eta}\bigg[R_{2,2} - R_{1,2}[u_{z,1},\psi_b](b)\bigg]\bigg[ \frac{R_{1,1}}{R_{1,2}} + [u_{z,2},\psi_a](a)\bigg]\no\\
&\qquad + e^{-i\eta}\bigg[-e^{-i\eta} + R_{1,2}[u_{z,1},\psi_a](a)\bigg]\bigg[-\frac{e^{i\eta}}{R_{1,2}} - [u_{z,2},\psi_b](b)\bigg]\no\\
&\qquad - e^{-i\eta}R_{1,1}[u_{z,1},\psi_b](b) - e^{-i\eta}R_{1,2}[u_{z,2},\psi_a](a)[u_{z,1},\psi_b](b) - e^{-2i\eta}[u_{z,2},\psi_b](b)\no\\
&\qquad + e^{-i\eta}R_{1,2}[u_{z,1},\psi_a](a)[u_{z,2},\psi_b](b)\no\\
&\quad = e^{-i\eta} \bigg(\frac{R_{1,1}R_{2,2}-1}{R_{1,2}}\bigg) + (R_{1,1}R_{2,2}-R_{2,1}R_{1,2})[u_{z,1},\psi_a](a)\no\\
&\qquad + e^{-i\eta}R_{2,2}[u_{z,2},\psi_a](a) - e^{-i\eta}\frac{R_{2,2}R_{1,1}}{R_{1,2}} - e^{-i\eta}R_{2,2}[u_{z,2},\psi_a](a)\no\\
&\qquad + e^{-i\eta}R_{1,1}[u_{z,1},\psi_b](b) + e^{-i\eta}R_{1,2}[u_{z,1},\psi_b](b)[u_{z,2},\psi_a](a) + \frac{e^{-i\eta}}{R_{1,2}}\no\\
&\qquad + e^{-2i\eta}[u_{z,2},\psi_b](b) - [u_{z,1},\psi_a](a) - e^{-i\eta}R_{1,2}[u_{z,1},\psi_a](a)[u_{z,2},\psi_b](b)\no\\
&\qquad - e^{-i\eta}R_{1,1}[u_{z,1},\psi_b](b) - e^{-i\eta}R_{1,2}[u_{z,2},\psi_a](a)[u_{z,1},\psi_b](b) - e^{-2i\eta}[u_{z,2},\psi_b](b)\no\\
&\qquad + e^{-i\eta}R_{1,2}[u_{z,1},\psi_a](a)[u_{z,2},\psi_b](b)\no\\
&\quad = 0.\lb{4.105}
\end{align}
Similarly, the coefficient in square brackets which multiplies $\langle u_{\ol{z},2},f\rangle_{(a,b)}$ on the right-hand side of \eqref{4.104} equals
\begin{align}
&R_{2,2}\bigg(\frac{R_{1,1}R_{2,2}-1}{R_{1,2}} \bigg) + (-R_{1,1}R_{2,2} + 1)[u_{z,1},\psi_b](b) - \frac{R_{1,1}R_{2,2}^2}{R_{1,2}}\no\\
&\qquad - R_{2,2}^2[u_{z,2},\psi_a](a) + R_{2,2}R_{1,1}[u_{z,1},\psi_b](b) + R_{2,2}R_{1,2}[u_{z,1},\psi_b](b)[u_{z,2},\psi_a](a)\no\\
&\qquad + \frac{R_{2,2}}{R_{1,2}} - e^{i\eta}R_{2,2}[u_{z,1},\psi_a](a) + e^{-i\eta}R_{2,2}[u_{z,2},\psi_b](b)\no\\
&\qquad - R_{1,2}R_{2,2}[u_{z,2},\psi_b](b)[u_{z,1},\psi_a](a) + e^{i\eta}R_{2,2}[u_{z,1},\psi_a](a)\no\\
&\qquad + R_{2,2}R_{1,2}[u_{z,2},\psi_b](b)[u_{z,1},\psi_a](a) + R_{2,2}^2[u_{z,2},\psi_a](a)\no\\
&\qquad - R_{2,2}R_{1,2}[u_{z,1},\psi_b](b)[u_{z,2},\psi_a](a) - [u_{z,1},\psi_b](b) - e^{-i\eta}R_{2,2}[u_{z,2},\psi_b](b)\no\\
&\quad = 0.\lb{4.106}
\end{align}
As a result of \eqref{4.104}--\eqref{4.106}, one infers that $F_{R,\eta}(z)f$ satisfies
\begin{align}
&R_{1,2}\det(K_{R,\eta}(z))\big\{R_{2,1}[F_{R,\eta}(z)f,\phi_a](a) + R_{2,2}[F_{R,\eta}(z)f,\psi_a](a)\no\\
&\hspace*{6.1cm} - e^{-i\eta}[F_{R,\eta}(z)f,\phi_b](b)\big\}=0,\lb{4.107}
\end{align}
and since neither $R_{1,2}$ nor $\det(K_{R,\eta}(z))$ is zero, the boundary condition in \eqref{4.99} holds.  Now \eqref{4.95} follows from \eqref{4.97}, \eqref{4.98}, and \eqref{4.99}.  It remains to show \eqref{4.96}, but this is a straightforward calculation using the fact that $T_{0,0}$ and $T_{R,\eta}$ are both restrictions of $T_{\max}$:
\begin{align}
(T_{R,\eta}-zI_{(a,b)})F_{R,\eta}(z)f &= (T_{\max}-zI_{(a,b)})F_{R,\eta}(z)f\lb{4.108}\\
&= (T_{\max}-zI_{(a,b)})(T_{0,0}-zI_{(a,b)})^{-1}f \no\\
&\quad + \sum_{j,k=1}^2\big[K_{R,\eta}(z)^{-1}\big]_{j,k}\langle u_{\ol{z},j},f\rangle_{(a,b)} (T_{\max}-zI_{(a,b)}) u_{z,k}\no\\
&= (T_{0,0}-zI_{(a,b)})(T_{0,0}-zI_{(a,b)})^{-1}f= I_{(a,b)}f= f.\no
\end{align}
\end{proof}

If $R_{1,2}=0$, then $T_{R,\eta}$ and $T_{0,0}$ are no longer relatively prime with respect to $T_{\min}$.  In this case, we obtain:
\begin{theorem}	\lb{t4.7}
Assume Hypothesis \ref{h4.1}.  If $R_{1,2}=0$, then the maximal common part of $T_{R,\eta}$ and $T_{0,0}$ is the restriction of $T_{\max}$ to the set
\begin{align}
\cS_{R,\eta}=\{y\in\dom(T_{\max})\,|\,[y,\phi_{a}](a)=[y,\phi_{b}](b)=0,\,[y,\psi_{b}](b)=e^{i\eta}R_{2,2}[y,\psi_{a}](a)\}.
	\lb{4.109}
\end{align}
Moreover, for each $z\in\rho(T_{R,\eta})\cap\rho(T_{0,0})$, the scalar
\begin{align}
k_{R,\eta}(z)=R_{2,1}R_{2,2}+e^{i\eta}R_{2,2}[u_{z,R,\eta},\psi_{a}](a)-[u_{z,R,\eta},\psi_{b}](b)
	\lb{4.110}
\end{align}
is nonzero, and
\begin{align}
(T_{R,\eta}-zI_{(a,b)})^{-1}=(T_{0,0}-zI_{(a,b)})^{-1}+k_{R,\eta}(z)^{-1}\left<u_{\ol{z},R,\eta},\,\cdot\,\right>_{(a,b)}u_{z,R,\eta},
	\lb{4.111}
\end{align}
where
\begin{align}
u_{z,R,\eta}:= e^{-i\eta}R_{2,2}u_{z,2}+u_{z,1}.
	\lb{4.112}
\end{align}
\end{theorem}
\begin{proof}
Let $R\in \SL_{2}(\bbR)$ with $R_{1,2}=0$.  To prove that the maximal common part of $T_{0,0}$ and $T_{R,\eta}$ is the restriction of $T_{\max}$ to $\cS_{R,\eta}$, it suffices to show
\begin{align}
\dom(T_{0,0})\cap\dom(T_{R,\eta})=\cS_{R,\eta}.
	\lb{4.113}
\end{align}
To this end, suppose $y\in\dom(T_{0,0})\cap\dom(T_{R,\eta})$.  Then the fact that $y\in\dom(T_{0,0})$ implies $y$ satisfies the conditions in \eqref{4.25}, and the fact that $y\in\dom(T_{R,\eta})$ implies
\begin{align}
[y,\psi_{b}](b)=e^{i\eta}R_{2,2}[y,\psi_{a}](a).
	\lb{4.114}
\end{align}
Hence, $y\in\cS_{R,\eta}$.  Conversely, if $y\in\cS_{R,\eta}$, then \eqref{4.25} and \eqref{4.114} hold, so $y\in\dom(T_{0,0})\cap\dom(T_{R,\eta})$.  To complete the proof, let $z\in\rho(T_{0,0})\cap\rho(T_{R,\eta})$ be fixed and let $k_{R,\eta}(z)$ be the scalar defined in \eqref{4.110}.  To prove the claim that $k_{R,\eta}(z)$ is nonzero, suppose on the contrary that $k_{R,\eta}(z)=0$.  We claim that $z$ is then an eigenvalue of $T_{R,\eta}$ and $u_{z,R,\eta}$ defined by \eqref{4.112} is a corresponding eigenfunction. To justify this claim, it is enough to show that $u_{z,R,\eta}\in\dom(T_{R,\eta})$.  In turn, it suffices to show $u_{z,R,\eta}$ satisfies the boundary conditions in \eqref{2.26}.  To this end, one computes
\begin{align}
&e^{i\eta}R_{1,1}[u_{z,R,\eta},\phi_{a}](a)-[u_{z,R,\eta},\phi_{b}](b)\no\\
&\quad=e^{i\eta}R_{1,1}e^{-i\eta}R_{2,2}[u_{z,2},\phi_{a}](a)+e^{i\eta}R_{1,1}[u_{z,1},\phi_{a}](a)-e^{-i\eta}R_{2,2}[u_{z,2},\phi_{b}](b)\no\\
&\quad\quad-[u_{z,1},\phi_{b}](b)\no\\
&\quad=1-1=0,
	\lb{4.115}
\end{align}
using the conditions \eqref{4.2} and $1=\det(R)=R_{1,1}R_{2,2}$.  Moreover, the assumption $k_{R,\eta}(z)=0$ implies
\begin{align}
&e^{i\eta}R_{2,1}[u_{z,R,\eta},\phi_{a}](a)+e^{i\eta}R_{2,2}[u_{z,R,\eta},\psi_{a}](a)-[u_{z,R,\eta},\psi_{b}](b)\no\\
&\quad=R_{2,1}R_{2,2}[u_{z,2},\phi_{a}](a)+e^{i\eta}R_{2,1}[u_{z,1},\phi_{a}](a)+e^{i\eta}R_{2,2}[u_{z,R,\eta},\psi_{a}](a)\no\\
&\quad\quad-[u_{z,R,\eta},\psi_{b}](b)\no\\
&\quad=R_{2,1}R_{2,2}+e^{i\eta}R_{2,2}[u_{z,R,\eta},\psi_{a}](a)-[u_{z,R,\eta},\psi_{b}](b)=k_{R,\eta}(z)=0.
	\lb{4.116}
\end{align}
The identities \eqref{4.115} and \eqref{4.116} imply that $u_{z,R,\eta}\in\dom(T_{R,\eta})$, from which it follows that $z$ is an eigenvalue of $T_{R,\eta}$ with corresponding eigenfunction $u_{z,R,\eta}$.  This is a contradiction to the choice of $z\in \rho(T_{R,\eta})$.  This completes the proof that $k_{R,\eta}(z)\neq 0$.  It remains to establish the resolvent identity in \eqref{4.111}.  Define
\begin{align}
\begin{split}
F_{R,\eta}(z)=(T_{0,0}-zI_{(a,b)})^{-1}+k_{R,\eta}(z)^{-1}\left<u_{\ol{z},R,\eta},\,\cdot\,\right>_{(a,b)}u_{z,R,\eta},&\\
\dom(F_{R,\eta}(z))=L^{2}((a,b);r(x)\,dx).&
\end{split}
	\lb{4.117}
\end{align}
In order to prove \eqref{4.111}, it suffices to show
\begin{align}
(T_{R,\eta}-zI_{(a,b)})F_{R,\eta}(z)=I_{(a,b)};
	\lb{4.118}
\end{align}
that is, it suffices to show that for every $f\in L^{2}((a,b);r(x)\,dx)$,
\begin{align}
F_{R,\eta}(z)f\in\dom(T_{R,\eta}),
	\lb{4.119}
\end{align}
and
\begin{align}
(T_{R,\eta}-zI_{(a,b)})(F_{R,\eta}(z)f)=f.
	\lb{4.120}
\end{align}
To this end, let $f\in L^{2}((a,b);r(x)\,dx)$ be arbitrary.  It is clear from the definition of $F_{R,\eta}(z)$ that
\begin{align}
F_{R,\eta}(z)f\in\dom(T_{\max}),
	\lb{4.121}
\end{align}
so the proof of \eqref{4.119} reduces to showing that $F_{R,\eta}(z)f$ satisfies the boundary conditions for functions in $\dom(T_{R,\eta})$; that is,
\begin{align}
0&=e^{i\eta}R_{1,1}[F_{R,\eta}(z)f,\phi_{a}](a)-[F_{R,\eta}(z)f,\phi_{b}](b),
	\lb{4.122}\\
0&=e^{i\eta}R_{2,1}[F_{R,\eta}(z)f,\phi_{a}](a)+e^{i\eta}R_{2,2}[F_{R,\eta}(z)f,\psi_{a}](a)\no\\
&\quad -[F_{R,\eta}(z)f,\psi_{b}](b).
	\lb{4.123}
\end{align}
To check \eqref{4.122}, one uses \eqref{4.1} and \eqref{4.2} to compute
\begin{align}
&e^{i\eta}R_{1,1}[F_{R,\eta}(z)f,\phi_{a}](a)-[F_{R,\eta}(z)f,\phi_{b}](b)\no\\
&\quad=e^{i\eta}R_{1,1}\big[(T_{0,0}-zI_{(a,b)})^{-1}f,\phi_{a}\big](a)\no\\
&\quad\quad+e^{i\eta}R_{1,1}k_{R,\eta}(z)^{-1}\left<u_{\ol{z},R,\eta},f\right>_{(a,b)}\left\{e^{-i\eta}R_{2,2}[u_{z,2},\phi_{a}](a)+[u_{z,1},\phi_{a}](a)\right\}\no\\
&\quad\quad-\big[(T_{0,0}-zI_{(a,b)})^{-1}f,\phi_{b}\big](b)\no\\
&\quad\quad-k_{R,\eta}(z)^{-1}\left<u_{\ol{z},R,\eta},f\right>_{(a,b)}\left\{e^{-i\eta}R_{2,2}[u_{z,2},\phi_{b}](b)+[u_{z,1},\phi_{b}](b)\right\}\no\\
&\quad=k_{R,\eta}(z)^{-1}\left<u_{\ol{z},R,\eta},f\right>_{(a,b)}\left(R_{1,1}R_{2,2}-1\right)\no\\
&\quad=k_{R,\eta}(z)^{-1}\left<u_{\ol{z},R,\eta},f\right>_{(a,b)}\left(\det(R)-1\right)=0.
	\lb{4.124}
\end{align}
To check \eqref{4.123}, one uses \eqref{4.1}, \eqref{4.2}, and Lemma \ref{l4.2} to compute
\begin{align}
&e^{i\eta}R_{2,1}[F_{R,\eta}(z)f,\phi_{a}](a)+e^{i\eta}R_{2,2}[F_{R,\eta}(z)f,\psi_{a}](a)-[F_{R,\eta}(z)f,\psi_{b}](b)\no\\
&\quad=e^{i\eta}R_{2,1}\big[(T_{0,0}-zI_{(a,b)})^{-1}f,\phi_{a}\big](a)\no\\
&\quad\quad+e^{i\eta}R_{2,1}k_{R,\eta}(z)^{-1}\left<u_{\ol{z},R,\eta},f\right>_{(a,b)}\left\{e^{-i\eta}R_{2,2}[u_{z,2},\phi_{a}](a)+[u_{z,1},\phi_{a}](a)\right\}\no\\
&\quad\quad+e^{i\eta}R_{2,2}\big[(T_{0,0}-zI_{(a,b)})^{-1}f,\psi_{a}\big](a)\no\\
&\quad\quad+e^{i\eta}R_{2,2}k_{R,\eta}(z)^{-1}\left<u_{\ol{z},R,\eta},f\right>_{(a,b)}[u_{z,R,\eta},\psi_{a}](a)\no\\
&\quad\quad-\big[(T_{0,0}-zI_{(a,b)})^{-1}f,\psi_{b}\big](b)-k_{R,\eta}(z)^{-1}\left<u_{\ol{z},R,\eta},f\right>_{(a,b)}[u_{z,R,\eta},\psi_{b}](b)\no\\
&\quad=R_{2,1}R_{2,2}k_{R,\eta}(z)^{-1}\left<u_{\ol{z},R,\eta},f\right>_{(a,b)}-e^{-i\eta}R_{2,2}\left<u_{\ol{z},2},f\right>_{(a,b)}\no\\
&\quad\quad+e^{i\eta}k_{R,\eta}(z)^{-1}\left<u_{\ol{z},R,\eta},f\right>_{(a,b)}[u_{z,R,\eta},\psi_{a}](a)-\left<u_{\ol{z},1},f\right>_{(a,b)}\no\\
&\quad\quad-k_{R,\eta}(z)^{-1}\left<u_{\ol{z},R,\eta},f\right>_{(a,b)}[u_{z,R,\eta},\psi_{b}](b)\no\\
&\quad=\left<u_{\ol{z},R,\eta},f\right>_{(a,b)}k_{R,\eta}(z)^{-1}\left\{R_{2,1}R_{2,2}+e^{i\eta}R_{2,2}[u_{z,R,\eta},\psi_{a}](a)-[u_{z,R,\eta},\psi_{b}](b)\right\}\no\\
&\quad\quad-\left<u_{\ol{z},R,\eta},f\right>_{(a,b)}\no\\
&\quad=\left<u_{\ol{z},R,\eta},f\right>_{(a,b)}-\left<u_{\ol{z},R,\eta},f\right>_{(a,b)}=0.
	\lb{4.125}
\end{align}
This proves \eqref{4.119}, and subsequently, the claim in \eqref{4.120} is a result of the following calculation:
\begin{align}
(T_{R,\eta}-zI_{(a,b)})F_{R,\eta}(z)f&=(T_{\max}-zI_{(a,b)})F_{R,\eta}(z)f\lb{4.126}\\
&=(T_{\max}-zI_{(a,b)})(T_{0,0}-zI_{(a,b)})^{-1}f\no\\
&\quad+k_{R,\eta}(z)^{-1}\left<u_{\ol{z},R,\eta},f\right>_{(a,b)}(T_{\max}-zI_{(a,b)})u_{z,R,\eta}\no\\
&=(T_{0,0}-zI_{(a,b)})(T_{0,0}-zI_{(a,b)})^{-1}f=f=I_{(a,b)}f.\no
\end{align}
\end{proof}

\begin{remark}
Using linearity of the trace functional, the rank one trace formula in \eqref{3.10a}, and \eqref{4.23}, \eqref{4.55}, \eqref{4.58}, \eqref{4.72}, \eqref{4.111}, one may obtain explicit trace formulas which are analogous to \eqref{3.14} for the resolvent differences
\begin{equation}
(T_{\alpha,\beta}-zI_{(a,b)})^{-1} - (T_{0,0}-zI_{(a,b)})^{-1},\quad z\in \rho(T_{\alpha,\beta})\cap\rho(T_{0,0}),\, \alpha,\beta\in[0,\pi),
\end{equation}
and
\begin{equation}
\begin{split}
(T_{R,\eta}-zI_{(a,b)})^{-1} - (T_{0,0}-zI_{(a,b)})^{-1},\quad z\in \rho(T_{R,\eta})\cap\rho(T_{0,0}),&\\
R\in \SL_2(\bbR),\,\eta\in[0,\pi).&
\end{split}
\end{equation}
\hfill $\diamond$
\end{remark}

\begin{remark}
Assume Hypothesis \ref{h2.17}.  Explicit Krein resolvent identities for three-term Sturm--Liouville operators were derived in detail in \cite{CGNZ14} under the additional assumption that $\tau$ is {\it regular} on $(a,b)$.  Recall that $\tau$ is said to be regular on $(a,b)$ if $a$ and $b$ are finite and
\begin{equation}
p^{-1},\, q,\, r\in L^1((a,b);dx).
	\lb{4.127}
\end{equation}
Treating both separated and coupled boundary conditions, the authors of \cite{CGNZ14} derive Krein resolvent identities that relate the resolvent of any self-adjoint extension of $T_{\min}$, with either separated or coupled boundary conditions, to the resolvent of the Dirichlet extension (parametrized by vanishing boundary values) of $T_{\min}$ in the regular case.  Here we briefly comment on how the resolvent identities from \cite{CGNZ14} can be recovered  as special cases of the Krein resolvent identities obtained in Section \ref{s4}.  For simplicity, we consider only those self-adjoint extensions of $T_{\min}$ corresponding to separated boundary conditions which together with the Dirichlet extension are relatively prime with respect to $T_{\min}$.  The other cases may be treated in a similar fashion.  Henceforth, we shall assume that $\tau$ is regular on $(a,b)$.

Recall that in the regular case, if $f\in \dom(T_{\max})$, then $f$ and $pf'$ possess boundary values.  That is, the following limits exist
\begin{equation}
\begin{split}
f(a):=\lim_{x\downarrow a}f(x),\quad f^{[1]}(a) := \lim_{x\downarrow a} (pf')(x),\\
f(b):=\lim_{x\uparrow b}f(x),\quad f^{[1]}(b) := \lim_{x\uparrow b} (pf')(x).
\end{split}
	\lb{4.128}
\end{equation}
The self-adjoint extensions of $T_{\min}$ corresponding to separated boundary conditions are characterized in \cite[eq.~(3.1)]{CGNZ14} as a two-parameter family $\{H_{\theta_a,\theta_b}\}_{\theta_a,\theta_b\in [0,\pi)}$, where for each $\theta_a,\theta_b\in [0,\pi)$,
\begin{align}
&H_{\theta_a,\theta_b}f = T_{\max}f,\lb{4.129}\\
&f\in \dom(H_{\theta_a,\theta_b}) = \Bigg\{g\in \dom(T_{\max})\,\Bigg|\,
\begin{aligned}
\cos(\theta_a)g(a)+\sin(\theta_a)g^{[1]}(a)&=0,\no\\
\cos(\theta_b)g(b)-\sin(\theta_b)g^{[1]}(b)&=0
\end{aligned}
\Bigg\}.\no
\end{align}
Note that $H_{0,0}$ is the Dirichlet extension of $T_{\min}$.  We briefly explain how the Krein resolvent identity obtained in \cite[eq.~(3.13)]{CGNZ14} may be recovered as a special case of Theorem \ref{t4.4}.  For simplicity, we treat only the case $\theta_a,\theta_b\neq 0$.

To recover the Krein resolvent identity from \cite{CGNZ14}, one expresses $H_{\theta_a,\theta_b}$ in terms of the operators in \eqref{2.25} parametrized in terms of boundary condition bases.  Fix a pair of boundary condition bases $\{\phi_c,\psi_c\}$, $c\in \{a,b\}$, by choosing $\phi_a,\phi_b,\psi_a,\psi_b\in \dom(T_{\max})$ such that
\begin{equation}\lb{4.130}
\begin{split}
\phi_a(a) &= 0,\quad\, \phi_a^{[1]}(a)=1,\quad \phi_b(b) = 0,\quad \phi_b^{[1]}(b)=1,\\
\psi_a(a) &=1,\quad \psi_a^{[1]}(a)=0,\quad \psi_b(b)=1,\quad \psi_b^{[1]}(b)=0.
\end{split}
\end{equation}
The relations in \eqref{4.130} imply $[\psi_c,\phi_c](c)=1$, $c\in \{a,b\}$, and
\begin{equation}\lb{4.131}
\begin{split}
[g,\phi_a](a) &= g(a),\quad [g,\psi_a](a) = -g^{[1]}(a),\\
[g,\phi_b](b) &= g(b),\quad\; [g,\psi_b](b) = -g^{[1]}(b),
\end{split}
\quad g\in \dom(T_{\max}).
\end{equation}
With this choice of boundary condition bases, the self-adjoint extensions of $T_{\min}$ given by \eqref{2.25} are
\begin{align}
&T_{\alpha,\beta}f = T_{\max}f,\quad \alpha,\beta\in [0,\pi),\lb{4.132}\\
&f\in \dom(T_{\alpha,\beta}) = \Bigg\{g\in \dom(T_{\max})\,\Bigg|\,
\begin{aligned}
\cos(\alpha)g(a)-\sin(\alpha)g^{[1]}(a)&=0,\no\\
\cos(\beta)g(b)-\sin(\beta)g^{[1]}(b)&=0
\end{aligned}
\Bigg\}.\no
\end{align}
A comparison of \eqref{4.129} with \eqref{4.132} yields
\begin{equation}\lb{4.133}
H_{\theta_a,\theta_b} = T_{\pi-\theta_a,\theta_b},\quad H_{0,\theta_b} = T_{0,\theta_b},\quad  \theta_a\in (0,\pi),\, \theta_b\in [0,\pi).
\end{equation}
In particular, for the Dirichlet extension,
\begin{equation}\lb{4.134}
H_{0,0} = T_{0,0}.
\end{equation}
 The Krein resolvent identities in \cite[Theorem 3.1]{CGNZ14} relate the resolvents of $H_{\theta_a,\theta_b}$ and $H_{0,0}$ and are expressed in terms of the basis $\{u_j(z,\,\cdot\,)\}_{j=1,2}$, $z\in \rho(H_{0,0})$, of $\ker(T_{\max}-zI_{(a,b)})$ specified by the conditions
\begin{equation}\lb{4.135}
\begin{split}
&u_1(z,a)=0, \quad u_1(z,b)=1,  \\
&u_2(z,a)=1, \quad u_2(z,b)=0,
\end{split}
\quad z\in \rho(H_{0,0}).
\end{equation}
Comparing \eqref{4.2}, \eqref{4.5}, \eqref{4.6}, \eqref{4.134}, and \eqref{4.135}, one infers that
\begin{equation}\lb{4.136}
u_{z,j}=u_j(z,\,\cdot\,),\quad j\in \{1,2\},\, z\in \rho(H_{0,0})=\rho(T_{0,0}).
\end{equation}
If $\theta_a,\theta_b\in (0,\pi)$, then according to Theorem \ref{t4.4}, $T_{\pi-\theta_a,\theta_b}$ and $T_{0,0}$ are relatively prime.  For each $z\in \rho(T_{\pi-\theta_a,\theta_b})\cap\rho(T_{0,0})$, the matrix $K_{\pi-\theta_a,\theta_b}(z)$ given by \eqref{4.22} is invertible and the identity in \eqref{4.23} holds.  Using \eqref{4.131} and \eqref{4.136}, one computes
\begin{equation}\lb{4.137}
[u_{z,j},\psi_c](c)=-u_{z,j}^{[1]}(c)=-u_j^{[1]}(z,c),\quad j\in \{1,2\},\, c\in \{a,b\},\, z\in \rho(T_{0,0}).
\end{equation}
Therefore, by \eqref{4.22} and \eqref{4.137}, for $z\in \rho(T_{\pi-\theta_a,\theta_b})\cap\rho(T_{0,0})$,
\begin{align}
K_{\pi-\theta_a,\theta_b}(z) &=
\begin{pmatrix}
\cot(\theta_b)-u_1^{[1]}(z,b) & u_1^{[1]}(z,a)\\[2mm]
-u_2^{[1]}(z,b) & -\cot(\pi-\theta_a)+u_2^{[1]}(z,a)
\end{pmatrix}\no\\
&=
\begin{pmatrix}
\cot(\theta_b)-u_1^{[1]}(z,b) & u_1^{[1]}(z,a)\\[2mm]
-u_2^{[1]}(z,b) & \cot(\theta_a)+u_2^{[1]}(z,a)
\end{pmatrix}.\lb{4.138}
\end{align}
By \eqref{4.23}, \eqref{4.133}, \eqref{4.134}, and \eqref{4.138}, for $z\in \rho(H_{\theta_a,\theta_b})\cap\rho(H_{0,0})$,
\begin{align}
(H_{\theta_a,\theta_b}-zI_{(a,b)})^{-1} &= (T_{\pi-\theta_a,\theta_b}-zI_{(a,b)})^{-1} \no\\
&= (T_{0,0}-zI_{(a,b)})^{-1}\no\\
&\quad  - \sum_{j,k=1}^2\big[K_{\pi-\theta_a,\theta_b}(z)^{-1}\big]_{j,k}\langle u_{\ol{z},j},\,\cdot\,\rangle_{(a,b)} u_{z,k}\no\\
&= (H_{0,0}-zI_{(a,b)})^{-1}\no\\
&\quad  - \sum_{j,k=1}^2\big[K_{\pi-\theta_a,\theta_b}(z)^{-1}\big]_{j,k}\langle u_j(\ol{z},\,\cdot\,),\,\cdot\,\rangle_{(a,b)} u_k(z,\,\cdot\,),\lb{4.139}
\end{align}
which agrees with \cite[eq.~(3.13)]{CGNZ14} after interchanging the indices $j$ and $k$.

The other Krein resolvent identities in \cite[eqs.~(3.16), (3.19), (3.50), (3.53)]{CGNZ14} may be obtained in a similar manner as special cases of Theorems \ref{t4.5}, \ref{t4.6}, and \ref{t4.7}.\hfill $\diamond$
\end{remark}

\section{Applications to Bessel Operators} \lb{s5}

The Bessel differential operator has a storied history and has been studied by many authors; see \cite{GLN19} for an extensive list of references in this connection.  In this section, we consider the Bessel operator (in Liouville form) as an application of the results of Section \ref{s3} to a problem with singular endpoints.  Using Theorem \ref{t3.4}, we determine the explicit form of Krein's resolvent identity and use it to calculate the trace of the difference of the resolvents of the Friedrichs extension and any other self-adjoint extension.  The trace formula obtained via \eqref{3.14} is then used to explicitly compute the Krein spectral shift function for the pair of self-adjoint extensions.
\par
For $\nu\in [0,\infty)$ the Bessel differential expression $\tau_{\nu}$ is defined by choosing, in the notation of Hypothesis \ref{h2.1}:
\begin{equation}
\begin{split}
& a=0,\quad b=\infty,\\
&r(x)=p(x)=1,\quad q(x) = (\nu^2-(1/4))x^{-2},\quad x\in (0,\infty),
\end{split}
	\lb{5.1}
\end{equation}
so that the differential expression \eqref{2.3} takes the form
\begin{align}
\tau_{\nu}f=-f''+\f{\nu^{2}-\f{1}{4}}{x^{2}}f,\quad f\in\mathfrak{D}(0,\infty).
	\lb{5.2}
\end{align}
Following \eqref{2.8}, one defines the maximal operator $T_{\max}^{(\nu)}$ associated to $\tau_{\nu}$ in the Hilbert space $L^2((0,\infty);dx)$ by
\begin{equation}
\begin{split}
&T_{\max}^{(\nu)}f=\tau_{\nu}f,\\
&f\in \dom\big(T_{\max}^{(\nu)}\big)=\big\{g\in \mathfrak{D}(0,\infty)\,\big|\,g,\tau_{\nu}g\in L^2((0,\infty);dx)\big\},
\end{split}
	\lb{5.3}
\end{equation}
and, in accordance with \eqref{2.9}, the associated minimal operator $T_{\min}^{(\nu)}$ is defined by
\begin{equation}
T_{\min}^{(\nu)} := \big(T_{\max}^{(\nu)}\big)^*.
	\lb{5.4}
\end{equation}
Recall that $I_{(0,\infty)}$ and $\left<\,\cdot\,,\,\cdot\,\right>_{(0,\infty)}$ denote the identity operator and the inner product in $L^2((0,\infty);dx)$, respectively, and $\tr_{(0,\infty)}$ denotes the trace in $\cB_1\big(L^2((0,\infty);dx)\big)$.

As reported, for example, in \cite[Section 12]{Ev05}, for $\nu \in [0,1)$, $\tau_{\nu}$ is in the limit circle case at $x=0$ and in the limit point case at $x=\infty$.  On the other hand, for $\nu\in [1,\infty)$, $\tau_{\nu}$ is in the limit point case at both $x=0$ and $x=\infty$.  It follows from Theorem \ref{t2.18} that for $\nu\in [1,\infty)$, the minimal operator $T_{\min}^{(\nu)}$ is self-adjoint and, therefore, possesses no proper self-adjoint extensions.  So, it is to $\nu\in[0,1)$ that we restrict our attention.  In this case, the self-adjoint extensions of $T_{\min}^{(\nu)}$ are parametrized as a one-parameter family $\big\{T_{\theta}^{(\nu)}\big\}_{\theta\in[0,\pi)}$.

Our first goal is to explicitly compute the right-hand sides in \eqref{3.11}, \eqref{3.12}, and \eqref{3.14}.  This is carried out below in Propositions \ref{p5.1} and \ref{p5.6}.  The objects of interest ($\phi_0$, $\psi_0$, $k_{\theta}(\,\cdot\,)$, $w_z$, etc.) will all depend on $\nu$.  To clearly indicate this dependence, we append the subscript ``$\nu$'' to relevant quantities ($\phi_{0,\nu}$, $\psi_{0,\nu}$, $k_{\theta,\nu}(\,\cdot\,)$, $w_{z,\nu}$, etc.).  Moreover, as will become apparent, there is a natural bifurcation between $\nu\in(0,1)$ and $\nu=0$, so the two cases are treated in separate subsections.

The fact that $x=0$ is a limit circle endpoint and $x=\infty$ is a limit point endpoint implies that the difference of the resolvents of $T_{\theta}^{(\nu)}$ and $T_0^{(\nu)}$ is rank one, so that
\begin{equation}\lb{5.5a}
\begin{split}
\Big[\big(T_{\theta}^{(\nu)}-zI_{(0,\infty)}\big)^{-1} - \big(T_0^{(\nu)}-zI_{(0,\infty)}\big)^{-1}\Big]\in \cB_1\big(L^2((0,\infty);dx) \big),&\\
z\in \rho\big(T_{\theta}^{(\nu)}\big)\cap\rho\big(T_0^{(\nu)}\big).&
\end{split}
\end{equation}
By \cite[eqs.~(7.15) \& (8.16)]{EK07},
\begin{equation}
\sigma\big(T_0^{(\nu)}\big)=\sigma_{\text{ac}}\big(T_0^{(\nu)}\big)=\sigma_{\text{ess}}\big(T_0^{(\nu)}\big)=[0,\infty),\quad \sigma_{\text{p}}\big(T_0^{(\nu)}\big)=\emptyset,\quad \nu\in[0,1).
	\lb{5.10}
\end{equation}
In particular, $T_0^{(\nu)}$ has no eigenvalues.  Based on abstract principles (cf., e.g., \cite[Theorems 8.12 \& 9.29]{Sc12}), the condition in \eqref{5.5a} implies that $T_{\theta}^{(\nu)}$ and $T_0^{(\nu)}$ have the same essential (resp., absolutely continuous) spectra.  In particular,
\begin{equation}
\sigma_{\text{ac}}\big(T_{\theta}^{(\nu)}\big)=\sigma_{\text{ess}}\big(T_{\theta}^{(\nu)}\big)=[0,\infty),\quad \theta\in (0,\pi),\, \nu\in[0,1).
	\lb{5.34}
\end{equation}
Since the deficiency indices of $T_{\min}^{(\nu)}$ are $(1,1)$, $T_{\theta}^{(\nu)}$ has at most one negative eigenvalue of multiplicity one (cf., e.g., \cite[Chapter IV, Section 14.11, Theorem 16]{Na68}).  Later, we shall determine precisely when $T_{\theta}^{(\nu)}$ possesses a negative eigenvalue and compute it explicitly.

For each $\nu\in [0,1)$ and $\theta\in (0,\pi)$, the resolvent comparability condition \eqref{5.5a} and the fact that $T_{\theta}^{(\nu)}$ and $T_0^{(\nu)}$ are bounded from below ensure the existence of a unique real-valued spectral shift function (cf.~Appendix \ref{sA}),
\begin{equation}
\xi\big(\,\cdot\,; T_{\theta}^{(\nu)},T_0^{(\nu)}\big)\in L^1\left(\bbR;(1+\lambda^2)^{-1}\, d\lambda\right),
	\lb{5.35}
\end{equation}
such that
\begin{equation}
\xi\big(\lambda;T_{\theta}^{(\nu)},T_0^{(\nu)}\big)=0,\quad \lambda<\min\Big[\sigma\big(T_{\theta}^{(\nu)}\big)\cup\sigma\big(T_0^{(\nu)}\big)\Big],
	\lb{5.36}
\end{equation}
and for which the following trace formula holds:
\begin{align}
\begin{split}
&\tr_{(0,\infty)}\Big(\big(T_{\theta}^{(\nu)}-z I_{(0,\infty)}\big)^{-1} - \big(T_0^{(\nu)}-zI_{(0,\infty)}\big)^{-1}\Big)\\
&\quad = -\int_{\bbR}\frac{\xi\big(\lambda; T_{\theta}^{(\nu)},T_0^{(\nu)}\big)}{(\lambda - z)^2}\, d\lambda,\quad z\in \bbC\backslash\bbR.
\end{split}
	\lb{5.37}
\end{align}

Our second goal is to compute the spectral shift function $\xi\big(\,\cdot\,;T_{\theta}^{(\nu)},T_0^{(\nu)}\big)$.  By Lemma \ref{lA.3}, the behavior of the spectral shift function on $(-\infty,0)$ yields information about the presence of negative eigenvalues of $T_{\theta}^{(\nu)}$.  This analysis is carried out below in Propositions \ref{p5.3} and \ref{p5.8}.

Many of the formulas obtained in this section contain nonintegral powers of the complex parameter $z$.  For $z\in \bbC\backslash\{0\}$ and $\beta\in (0,1)$, we define the complex powers $z^{\pm \beta}$ by writing $z$ in polar form
\begin{equation}\lb{5.5}
\text{$z=|z|e^{i\arg(z)}$ with $\arg(z)\in [0,2\pi)$},
\end{equation}
and setting
\begin{equation}
z^{\pm \beta} = |z|^{\pm \beta}e^{\pm i\beta \arg(z)}.
	\lb{5.6}
\end{equation}
In particular, the convention of choosing $\arg(z)\in [0,2\pi)$ ensures $\Im(z^{1/2})\geq 0$.  To be consistent with \eqref{5.5}, we fix a branch of the logarithm by
\begin{equation}\lb{5.7}
\ln(z) = \ln|z| + i\arg(z),\quad \arg(z)\in [0,2\pi),\, z\in \bbC\backslash\{0\}.
\end{equation}

\subsection{The case $\nu\in (0,1)$}\lb{s5.1}
Let $\nu\in (0,1)$ be fixed throughout this subsection.  Following \cite[eq.~(8.1)]{EK07} (cf.~also \cite[Section 12]{Ev05}), one fixes a boundary condition basis $\{\phi_{0,\nu},\psi_{0,\nu}\}$ at $a=0$ by choosing functions $\phi_{0,\nu},\psi_{0,\nu}\in \dom\big(T_{\max}^{(\nu)}\big)$ which vanish in a neighborhood of $\infty$ and satisfy
\begin{align}
\phi_{0,\nu}(x)=x^{\f{1}{2}+\nu}\quad\text{and}\quad \psi_{0,\nu}(x)=\f{1}{2\nu}x^{\f{1}{2}-\nu},\quad x\in(0,1).
	\lb{5.8}
\end{align}
This is possible by the Naimark patching lemma \cite[Chapter V, Section 17.3, Lemma 2]{Na68}.  The self-adjoint extensions of $T_{\min}^{(\nu)}$ are parametrized according to Theorem \ref{t2.19} as a one-parameter family $\big\{T_{\theta}^{(\nu)}\big\}_{\theta\in[0,\pi)}$, where for $\theta\in [0,\pi)$,
\begin{align}
&T_{\theta}^{(\nu)}f = T_{\max}^{(\nu)}f,
	\lb{5.9}\\
&f\in \dom\big(T_{\theta}^{(\nu)}\big) = \big\{g\in \dom\big(T_{\max}^{(\nu)}\big)\,\big|\, \cos(\theta)[g,\phi_{0,\nu}](0) + \sin(\theta)[g,\psi_{0,\nu}](0)=0\big\}.\no
\end{align}
By \cite[eq.~(8.16)]{EK07} (cf.~also \cite[Proposition 5.3 $(i)$]{AB15}), the operator $T_0^{(\nu)}$ is the Friedrichs extension of $T_{\min}^{(\nu)}$.  Moreover, by \cite[Proposition 5.3 $(ii)$]{AB15}, the extension $T_{\pi/2}^{(\nu)}$ is the Krein--von Neumann extension of $T_{\min}^{(\nu)}$.  For details on the significance of the Friedrichs and Krein--von Neumann extensions, we refer to \cite{AS80}, \cite{AN70}, \cite{AGMST10}, and \cite{Kr47}.

For $z\in \bbC\backslash\{0\}$, a basis of solutions $\{s_{z,\nu},c_{z,\nu}\}$ to the equation
\begin{equation}
\tau_{\nu}y=zy\quad \text{on}\quad (0,\infty)
	\lb{5.11}
\end{equation}
is fixed by setting
\begin{equation}
\begin{split}
s_{z,\nu}(x) &= -2^{\nu}\Gamma(1+\nu)z^{-\nu/2}x^{1/2}J_{\nu}\big(z^{1/2}x\big),\\
c_{z,\nu}(x) &=\frac{\Gamma(1-\nu)}{2^{\nu+1}\nu}z^{\nu/2}x^{1/2}J_{-\nu}\big(z^{1/2}x\big),\quad x\in (0,\infty).
\end{split}
	\lb{5.12}
\end{equation}
Here $\Gamma(\,\cdot\,)$ denotes the gamma function and $J_{\pm\nu}(\,\cdot\,)$ denote the Bessel functions of the first kind with indices $\pm\nu$ (cf., e.g., \cite[Section 9.1]{AS72}):
\begin{equation}
J_{\pm\nu}(\zeta)=\bigg(\frac{\zeta}{2}\bigg)^{\pm\nu} \sum_{k=0}^{\infty} \frac{(-\zeta^2/4)^k}{k!\Gamma(k+1\pm \nu)},\quad \zeta\in \bbC\backslash \{0\}.
	\lb{5.13}
\end{equation}
Using the asymptotics as $x\downarrow 0$ implied by \eqref{5.13}, one verifies that the basis functions $s_{z,\nu}$ and $c_{z,\nu}$ satisfy the generalized boundary conditions (cf. \cite[eqs.~(4.5) \& (4.6)]{AB15} and \cite[Section 8.2]{EK07})
\begin{align}
&[s_{z,\nu},\phi_{0,\nu}](0) = 0,\quad [s_{z,\nu},\psi_{0,\nu}](0)= 1,
	\lb{5.14}\\
&[c_{z,\nu},\phi_{0,\nu}](0) = 1,\quad [c_{z,\nu},\psi_{0,\nu}](0)=0,\quad z\in\bbC\backslash\{0\},
	\lb{5.15}
\end{align}
in analogy with the classical boundary values of sine, cosine, and their derivatives.

If $z\in \rho\big(T_0^{(\nu)}\big)=\bbC\backslash[0,\infty)$, then $s_{z,\nu}$ is the unique solution to \eqref{5.11} which satisfies the conditions in \eqref{5.14}, and a nontrivial solution of \eqref{5.11} which lies in $L^2((0,\infty);dx)$ is given by
\begin{align}
&x^{1/2}H_{\nu}^{(1)}\big(z^{1/2}x\big)
	\lb{5.16}\\
&\quad=i\csc(\nu\pi)\left[e^{-i\nu\pi}x^{1/2}J_{\nu}\big(z^{1/2}x\big)-x^{1/2}J_{-\nu}\big(z^{1/2}x\big)\right]\no\\
&\quad = -\frac{ie^{-i\nu\pi}}{\sin(\nu\pi)2^{\nu}\Gamma(1+\nu)}z^{\nu/2}s_{z,\nu}(x) - \frac{i2^{\nu+1}\nu}{\sin(\nu\pi)\Gamma(1-\nu)}z^{-\nu/2}c_{z,\nu}(x),\quad x\in (0,\infty),\no
\end{align}
where $H_{\nu}^{(1)}(\,\cdot\,)=J_{\nu}(\,\cdot\,)+iY_{\nu}(\,\cdot\,)$ is the Hankel function of the first kind, a combination of the Bessel and Neumann functions (cf.\ \cite[Section 9.1]{AS72}).  That the function in \eqref{5.16} actually lies in $L^2((0,\infty);dx)$ is a consequence of the fact that $\tau_{\nu}$ is in the limit circle case at $x=0$ together with the asymptotic behavior of $H_{\nu}^{(1)}(\,\cdot\,)$ as its (generally complex) argument tends to infinity (cf., e.g., \cite[eq.~9.2.3]{AS72}):
\begin{equation}
H_{\nu}^{(1)}(\zeta)\underset{|\zeta|\to \infty}{\sim}\bigg(\frac{2}{\pi\zeta}\bigg)^{1/2}e^{i\left(\zeta -\frac{1}{2}\nu\pi-\frac{1}{4}\pi\right)}.
	\lb{5.17}
\end{equation}

In order to explicitly compute \eqref{3.11}, \eqref{3.12}, and \eqref{3.14}, the Weyl--Titchmarsh solution $w_{z,\nu}$ corresponding to ``$w_z$'' in Hypothesis \ref{h3.1} $(iv)$ must be determined.  To this end,  one computes
\begin{align}
&\Big[(\,\cdot\,)^{1/2}\,H_{\nu}^{(1)}\big(z^{1/2}\,\cdot\,\big),\phi_{0,\nu}\Big](0)\no\\
&\quad=-\f{ie^{-i\nu\pi}z^{\nu/2}}{2^{\nu}\sin(\nu\pi)\Gamma(1+\nu)}[s_{z,\nu},\phi_{0,\nu}](0)-\f{2^{\nu+1}\nu i}{\sin(\nu\pi)\Gamma(1-\nu)z^{\nu/2}}[c_{z,\nu},\phi_{0,\nu}](0)\no\\
&\quad=-\f{2^{\nu+1}\nu i}{\sin(\nu\pi)\Gamma(1-\nu)z^{\nu/2}},\quad z\in \rho\big(T_0^{(\nu)}\big).
	\lb{5.18}
\end{align}
For the sake of brevity, set
\begin{align}
\varpi_{\nu}(z):=\f{i\Gamma(1-\nu)\sin(\nu\pi)z^{\nu/2}}{2^{\nu+1}\nu},\quad z\in \rho\big(T_0^{(\nu)}\big).
	\lb{5.19}
\end{align}
Then, based on \eqref{5.18}, one infers
\begin{align}
\begin{split}
w_{z,\nu}(x)&=\varpi_{\nu}(z)x^{1/2}\,H_{\nu}^{(1)}\big(z^{1/2}x\big)\\
&=\f{\Gamma(1-\nu)e^{-i\nu\pi}z^{\nu}}{\Gamma(1+\nu)2^{2\nu+1}\nu}s_{z,\nu}(x)+c_{z,\nu}(x),\quad x\in (0,\infty),\, z\in \rho\big(T_0^{(\nu)}\big).
\end{split}
	\lb{5.20}
\end{align}
As a result, Hypothesis \ref{h3.1} holds with the choices
\begin{equation}
\begin{split}
&a=0,\quad b=\infty,\quad \phi_a=\phi_{0,\nu},\quad \psi_a=\psi_{0,\nu},\\
&T_{\theta}=T_{\theta}^{(\nu)},\quad u_z=s_{z,\nu},\quad w_z=w_{z,\nu},
\end{split}
	\lb{5.21}
\end{equation}
and Theorem \ref{t3.4} may be applied to relate the resolvents of $T_{\theta}^{(\nu)}$ and $T_0^{(\nu)}$ for any $\theta\in (0,\pi)$ and compute the trace of the corresponding resolvent difference.

To begin with, \eqref{5.14}, \eqref{5.15}, and \eqref{5.20} imply
\begin{align}
[w_{z,\nu},\psi_{0,\nu}](0)=\f{\Gamma(1-\nu)e^{-i\nu\pi}z^{\nu}}{\Gamma(1+\nu)2^{2\nu+1}\nu},\quad z\in \rho\big(T_0^{(\nu)}\big),
	\lb{5.22}
\end{align}
which yields an explicit expression for the right-hand side of \eqref{3.11}:
\begin{align}
k_{\theta,\nu}(z)&=\cot(\theta)+[w_{z,\nu},\psi_{0,\nu}](0)\no\\
&=\cot(\theta)+\f{\Gamma(1-\nu)e^{-i\nu\pi}z^{\nu}}{\Gamma(1+\nu)2^{2\nu+1}\nu},\quad z\in \rho\big(T_0^{(\nu)}\big)\cap\rho\big(T_{\theta}^{(\nu)}\big),\quad \theta\in (0,\pi).
	\lb{5.23}
\end{align}
With $w_{z,\nu}$ given by \eqref{5.20} and $k_{\theta,\nu}(\,\cdot\,)$ given by \eqref{5.23}, the right-hand side in \eqref{3.12} is completely determined.

To compute the right-hand side in \eqref{3.14}, the inner product $\left< w_{\overline{z},\nu},w_{z,\nu}\right>_{(0,\infty)}$ must be calculated for $z\in \rho\big(T_0^{(\nu)}\big)$.  Using the definition of the inner product and \eqref{3.3}, one computes for $z\in \rho\big(T_0^{(\nu)}\big)$,
\begin{align}
&\left<w_{\ol{z},\nu},w_{z,\nu}\right>_{(0,\infty)}	
	\lb{5.24}\\
&\quad=\varpi_{\nu}^{2}(z)\int_{0}^{\infty}x\big[H_{\nu}^{(1)}\big(z^{1/2}x\big)\big]^{2}\,dx\no\\
&\quad=\varpi_{\nu}^{2}(z)\left[\f{1}{2}x^{2}\left(\big[H_{\nu}^{(1)}\big(z^{1/2}x\big)\big]^{2}-H_{\nu-1}^{(1)}\big(z^{1/2}x\big)H_{\nu+1}^{(1)}\big(z^{1/2}x\big)\right)\right]_{0}^{\infty}.\no
\end{align}
The antiderivative after the final equality in \eqref{5.24} is due to a result of Lommel \cite[p.~135, (11)]{Wat44} for cylinder functions.  To evaluate at the endpoints, one relies on the asymptotic forms \cite[eqs.~9.1.9 \& 9.2.3]{AS72} and the order-reflection formula for Hankel functions \cite[eq.~9.1.6]{AS72}.  For the upper limit in \eqref{5.24}, one obtains
\begin{align}
&\underset{x\to\infty}\lim\left[\f{1}{2}x^{2}\left(\big[H_{\nu}^{(1)}\big(z^{1/2}x\big)\big]^{2}-H_{\nu-1}^{(1)}\big(z^{1/2}x\big)H_{\nu+1}^{(1)}\big(z^{1/2}x\big)\right)\right]\no\\
&\quad=\underset{x\to\infty}\lim\Bigg[\f{1}{2}x^{2}\Bigg(\f{2}{\pi z^{1/2}x}e^{2i\left(z^{1/2}x-\f{1}{2}\nu\pi-\f{1}{4}\pi\right)}\no\\
&\hspace*{1.8cm}-\f{2}{\pi z^{1/2}x}e^{i\left(z^{1/2}x-\f{1}{2}(\nu-1)\pi-\f{1}{4}\pi\right)}e^{i\left(z^{1/2}x-\f{1}{2}(\nu+1)\pi-\f{1}{4}\pi\right)}\Bigg)\Bigg]=0,
	\lb{5.25}
\end{align}
since $z\in \rho\big(T_0^{(\nu)}\big)=\bbC\backslash [0,\infty)$ implies $\Im(z^{1/2})>0$ (cf.~\eqref{5.6}).  In like fashion, at the lower limit:
\begin{align}
&\underset{x\downarrow 0}\lim\Bigg[\f{1}{2}x^{2}\Bigg(-\f{\Gamma^{2}(\nu)}{\pi^{2}\left(\f{1}{2}z^{1/2}x\right)^{2\nu}}
-\bigg(-\f{i\Gamma(1-\nu)e^{i(1-\nu)\pi}}{\pi\left(\f{1}{2}z^{1/2}x\right)^{1-\nu}}\cdot-\f{i\Gamma(1+\nu)}{\pi\left(\f{1}{2}z^{1/2}x\right)^{1+\nu}}\bigg)\Bigg)\Bigg]\no\\
&\quad=\underset{x\downarrow 0}\lim\Bigg[-\f{\Gamma^{2}(\nu)}{2\pi^{2}\left(\f{1}{2}z^{1/2}\right)^{2\nu}x^{2\nu-2}}+\f{2\Gamma(1-\nu)\Gamma(1+\nu)e^{i(1-\nu)\pi}}{\pi^{2}z}\Bigg]\no\\
&\quad=-\f{2\Gamma(1-\nu)\Gamma(1+\nu)e^{-i\nu\pi}}{\pi^{2} z}=-\f{2\nu e^{-i\nu\pi}}{\pi z\sin(\nu\pi)}.
	\lb{5.26}
\end{align}
Finally, substitution of \eqref{5.25} and \eqref{5.26} into \eqref{5.24} yields
\begin{align}
\left<w_{\ol{z},\nu},w_{z,\nu}\right>_{(0,\infty)}&=\varpi_{\nu}^{2}(z)\f{2\nu e^{-i\nu\pi}}{\pi z\sin(\nu\pi)}\no\\
&=-\f{\Gamma^{2}(1-\nu)\sin(\nu\pi)e^{-i\nu\pi}z^{\nu-1}}{2^{2\nu+1}\nu\pi}\no\\
&=-\f{\Gamma(1-\nu)e^{-i\nu\pi}z^{\nu-1}}{\Gamma(\nu)2^{2\nu+1}\nu},
	\lb{5.27}
\end{align}
where the last equality is due to \cite[eq.~6.1.17]{AS72}.  Now \eqref{5.23} and \eqref{5.27} permit one to explicitly compute the right-hand side in \eqref{3.14}.  The results are summarized in:

\begin{proposition}\lb{p5.1}
If $\nu\in (0,1)$, $\theta\in (0,\pi)$, and $z\in \rho\big(T_0^{(\nu)}\big)\cap\rho\big(T_{\theta}^{(\nu)}\big)$, then
\begin{align}
\begin{split}
&\big(T_{\theta}^{(\nu)}-zI_{(0,\infty)}\big)^{-1}-\big(T_0^{(\nu)}-zI_{(0,\infty)}\big)^{-1}\\
&\quad =\left[ \cot(\theta)+\f{\Gamma(1-\nu)e^{-i\nu\pi}}{\Gamma(1+\nu)2^{2\nu+1}\nu}z^{\nu}\right]^{-1}\left<w_{\overline{z},\nu},\,\cdot\,\right>_{(0,\infty)}w_{z,\nu},
\end{split}
	\lb{5.28}
\end{align}
where $w_{z,\nu}$ is defined by \eqref{5.20}.  In particular,
\begin{equation}
\Big[\big(T_{\theta}^{(\nu)}-zI_{(0,\infty)}\big)^{-1}-\big(T_0^{(\nu)}-zI_{(0,\infty)}\big)^{-1}\Big] \in \cB_1\big(L^2((0,\infty);dx)\big),
	\lb{5.29}
\end{equation}
and the following trace formula holds:
\begin{align}
&\tr_{(0,\infty)}\Big(\big(T_{\theta}^{(\nu)}-zI_{(0,\infty)}\big)^{-1}-\big(T_0^{(\nu)}-zI_{(0,\infty)}\big)^{-1}\Big)
	\lb{5.30}\\
&\quad = \left[ \cot(\theta)+\f{\Gamma(1-\nu)e^{-i\nu\pi}}{\Gamma(1+\nu)2^{2\nu+1}\nu}z^{\nu}\right]^{-1}\left<w_{\ol{z},\nu},w_{z,\nu}\right>_{(0,\infty)}\no\\
&\quad =-\left[ \cot(\theta)+\f{\Gamma(1-\nu)e^{-i\nu\pi}}{\Gamma(1+\nu)2^{2\nu+1}\nu}z^{\nu}\right]^{-1}\f{\Gamma(1-\nu)e^{-i\nu\pi}}{\Gamma(\nu)2^{2\nu+1}\nu}z^{\nu-1}\no\\
&\quad=-\f{\Gamma(1-\nu)\nu e^{-i\nu\pi}z^{\nu-1}}{\Gamma(1+\nu)2^{2\nu+1}\nu\cot(\theta)+\Gamma(1-\nu)e^{-i\nu\pi}z^{\nu}}.\no
\end{align}
\end{proposition}

\begin{remark}	\lb{r4.2}
In the special case $\nu=\tfrac{1}{2}$, the solutions \eqref{5.12} and \eqref{5.16} simplify to
\begin{equation}\lb{5.31}
\begin{split}
s_{z,1/2}(x)&=-\f{1}{z^{1/2}}\sin\big(z^{1/2}x\big),\\
c_{z,1/2}(x)&=\cos\big(z^{1/2}x\big),\quad x\in(0,\infty),\, z\in \bbC\backslash\{0\},
\end{split}
\end{equation}
and
\begin{align}
w_{z,1/2}(x)=e^{iz^{1/2}x},\quad x\in(0,\infty),\, z\in\rho\big(T_{0}^{(1/2)}\big).
	\lb{5.32}
\end{align}
In this case, the trace formula \eqref{5.30} reduces to
\begin{align}
&\tr_{(0,\infty)}\left(\big(T_{\theta}^{(1/2)}-zI_{(0,\infty)}\big)^{-1}-\big(T_{0}^{(1/2)}-zI_{(0,\infty)}\big)^{-1}\right)\no\\
&\quad =\f{i}{2z^{1/2}\big(\cot(\theta)-iz^{1/2}\big)},\quad z\in \rho\big(T_0^{1/2}\big)\cap\rho\big(T_{\theta}^{(1/2)}\big).
	\lb{5.33}
\end{align}
\hfill $\diamond$
\end{remark}

Next, as an application of the trace formula \eqref{5.30}, we explicitly compute the spectral shift function $\xi\big(\,\cdot\,;T_{\theta}^{(\nu)},T_0^{(\nu)}\big)$.  To simplify the statement of our results, we begin with a hypothesis that fixes some useful notation.

\begin{hypothesis}\lb{h5.3a}
$(i)$ Define the quantities
\begin{align}\lb{5.38}
e_{\theta,\nu}=-\left[\f{\Gamma(1+\nu)2^{2\nu+1}\nu|\cot(\theta)|}{\Gamma(1-\nu)}\right]^{1/\nu},\quad \theta\in (0,\pi),\, \nu\in (0,1),
\end{align}
and
\begin{align}
\lambda_{\theta,\nu}=\left[\f{\Gamma(1+\nu)2^{2\nu+1}\nu|\cot(\theta)|}{\Gamma(1-\nu)|\cos(\nu\pi)|}\right]^{1/\nu},\quad \theta\in (0,\pi),\, \nu\in (0,1)\backslash\big\{\tfrac{1}{2}\big\}.
	\lb{5.39}
\end{align}
$(ii)$ For each $(\theta,\nu)\in [(0,\pi)\times (0,1)]\backslash \big\{\big(\tfrac{\pi}{2},\tfrac{1}{2}\big)\big\}$, let the function $\Xi_{\theta,\nu}:\bbR\to\bbR$ be defined for each $\lambda\in \bbR$ by
\begin{align}
&\Xi_{\theta,\nu}(\lambda)=\lb{5.40a}\\
&\quad
\left\{
\begin{aligned}
&0,\quad \text{if $\Gamma(1+\nu)2^{2\nu+1}\nu\cot(\theta)+\Gamma(1-\nu)\cos(\nu\pi)|\lambda|^{\nu}= 0$},\\[2mm]
&\displaystyle{-\f{1}{\pi}\arctan\left(\f{\Gamma(1-\nu)\sin(\nu\pi)|\lambda|^{\nu}}{\Gamma(1+\nu)2^{2\nu+1}\nu\cot(\theta)+\Gamma(1-\nu)\cos(\nu\pi)|\lambda|^{\nu}}\right)},\quad \text{otherwise}.
\end{aligned}\right.\no
\end{align}
\end{hypothesis}

The following theorem provides the explicit form of the spectral shift function $\xi\big(\,\cdot\,;T_{\theta}^{(\nu)},T_0^{(\nu)}\big)$, $\theta\in (0,\pi)$, $\nu\in (0,1)$, in terms of the quantities \eqref{5.38} and \eqref{5.39}, and the function in \eqref{5.40a}.

\begin{proposition}\lb{p5.3}
Assume Hypothesis \ref{h5.3a}.  The following statements $(i)$--$(v)$ hold.\\[2mm]
$(i)$ If $\theta\in\big(0,\f{\pi}{2}\big]$, $\nu\in\big(0,\f{1}{2}\big]$, and $(\theta,\nu)\neq \big(\f{\pi}{2},\f{1}{2}\big)$, then
\begin{align}
&\xi\big(\lambda;T_{\theta}^{(\nu)},T_{0}^{(\nu)}\big)=\chi_{(0,\infty)}(\lambda)\Xi_{\theta,\nu}(\lambda)\, \text{ for a.e.~$\lambda\in \bbR$}.\lb{5.40}
\end{align}
In particular, $T_{\theta}^{(\nu)}$ has no negative eigenvalues.\\[2mm]
$(ii)$ If $\theta\in\big(0,\f{\pi}{2}\big]$ and $\nu\in\big(\f{1}{2},1\big)$, then
\begin{align}
&\xi\big(\lambda;T_{\theta}^{(\nu)},T_{0}^{(\nu)}\big)=-\chi_{(\lambda_{\theta,\nu},\infty)}(\lambda)+\chi_{(0,\infty)}(\lambda)\Xi_{\theta,\nu}(\lambda)\, \text{ for a.e.~$\lambda\in \bbR$}.\lb{5.41}
\end{align}
In particular, $T_{\theta}^{(\nu)}$ has no negative eigenvalues.\\[2mm]
$(iii)$ If $\theta\in\big(\f{\pi}{2},\pi\big)$ and $\nu\in\big(0,\f{1}{2}\big)$, then
\begin{align}
&\xi\big(\lambda;T_{\theta}^{(\nu)},T_{0}^{(\nu)}\big)=-\chi_{(e_{\theta,\nu},\lambda_{\theta,\nu})}(\lambda)+\chi_{(0,\infty)}(\lambda)\Xi_{\theta,\nu}(\lambda)\, \text{ for a.e.~$\lambda\in \bbR$}.\lb{5.42}
\end{align}
In particular, $T_{\theta}^{(\nu)}$ has a single negative eigenvalue $e_{\theta,\nu}$ of multiplicity one.\\[2mm]
$(iv)$ If $\theta\in\big(\f{\pi}{2},\pi\big)$ and $\nu\in\big[\f{1}{2},1\big)$, then
\begin{align}
&\xi\big(\lambda;T_{\theta}^{(\nu)},T_{0}^{(\nu)}\big)=-\chi_{(e_{\theta,\nu},\infty)}(\lambda)+\chi_{(0,\infty)}(\lambda)\Xi_{\theta,\nu}(\lambda)\, \text{ for a.e.~$\lambda\in \bbR$}.\lb{5.43}
\end{align}
In particular, $T_{\theta}^{(\nu)}$ has a single negative eigenvalue $e_{\theta,\nu}$ of multiplicity one.\\[2mm]
$(v)$ If $\theta=\f{\pi}{2}$ and $\nu=\f{1}{2}$, then
\begin{equation}\lb{5.44}
\xi\big(\lambda; T_{\pi/2}^{(1/2)},T_0^{(1/2)}\big) = -\f{1}{2}\chi_{(0,\infty)}(\lambda)\, \text{ for a.e.~$\lambda\in \bbR$}.
\end{equation}
In particular, $T_{\pi/2}^{(1/2)}$ has no negative eigenvalues.
\end{proposition}
\begin{proof}
Let $\theta\in(0,\pi)$ and $\nu\in(0,1)$.  Temporarily taking \eqref{5.40}--\eqref{5.44} for granted, the claims about negative eigenvalues of $T_{\theta}^{(\nu)}$ in each of the cases ($i$)--($v$) are immediate consequences of \eqref{5.10}, \eqref{5.34}, and Lemma \ref{lA.3} in conjunction with \eqref{5.40}--\eqref{5.44}.  It remains to justify \eqref{5.40}--\eqref{5.44}.

We begin with some general considerations before specializing to the individual cases $(i)$--$(v)$.  The trace formula \eqref{5.37} implies
\begin{align}
&\tr_{(0,\infty)}\left(\big(T_{\theta}^{(\nu)}-zI_{(0,\infty)}\big)^{-1}-\big(T_{0}^{(\nu)}-zI_{(0,\infty)}\big)^{-1}\right)\no\\
&\quad=-\int_{\bbR}\f{\xi\big(\lambda;T_{\theta}^{(\nu)},T_{0}^{(\nu)}\big)}{(\lambda-z)^{2}}\,d\lambda\no\\
&\quad=-\int_{\bbR}\xi\big(\lambda;T_{\theta}^{(\nu)},T_{0}^{(\nu)}\big)\f{d}{dz}\left[\f{1}{\lambda-z}-\f{\lambda}{1+\lambda^{2}}\right]\,d\lambda\no\\
&\quad=-\f{d}{dz}\int_{\bbR}\xi\big(\lambda;T_{\theta}^{(\nu)},T_{0}^{(\nu)}\big)\left[\f{1}{\lambda-z}-\f{\lambda}{1+\lambda^{2}}\right]\,d\lambda,\quad z\in\bbC\backslash\bbR.
	\lb{5.45}
\end{align}
The interchange of the integral and derivative in \eqref{5.45} is justified based on \eqref{5.35}.  By \eqref{5.30},
\begin{equation}\lb{5.46}
\begin{split}
\tr_{(0,\infty)}\left(\big(T_{\theta}^{(\nu)}-zI_{(0,\infty)}\big)^{-1}-\big(T_{0}^{(\nu)}-zI_{(0,\infty)}\big)^{-1}\right)=-\f{d}{dz}\ln\big(m_{\theta,\nu}(z)\big),&\\
z\in\bbC\backslash\bbR,&
\end{split}
\end{equation}
where
\begin{equation}
m_{\theta,\nu}(z):=\Gamma(1+\nu)2^{2\nu+1}\nu\cot(\theta)+\Gamma(1-\nu)e^{-i\nu\pi}z^{\nu},\quad z\in \bbC\backslash\bbR.
	\lb{5.47}
\end{equation}
Note that the condition $z\in \bbC\backslash \bbR$ in \eqref{5.46} implies that $m_{\theta,\nu}(z)\in \bbC\backslash [0,\infty)$, so the branch cut of the logarithm along $[0,\infty)$ (cf.~\eqref{5.7}) is avoided.  A comparison of \eqref{5.45} and \eqref{5.46} implies
\begin{equation}\lb{5.48}
\int_{\bbR}\xi\big(\lambda;T_{\theta}^{(\nu)},T_{0}^{(\nu)}\big)\left[\f{1}{\lambda-z}-\f{\lambda}{1+\lambda^{2}}\right]\,d\lambda=\ln\big(m_{\theta,\nu}(z)\big)+C_{\theta,\nu},\quad z\in\bbC\backslash\bbR,
\end{equation}
for some constant $C_{\theta,\nu}\in\bbC$.  The spectral shift function may be recovered pointwise a.e.~from \eqref{5.48} by employing the Stieltjes inversion-based technique used in the proofs of \cite[Theorem 5.5]{CGNZ14} and \cite[Lemma 7.4]{GLMZ05}.  Specifically, by the Stieltjes inversion formula \cite[Theorem 2.2 $(v)$]{GT00}, applied separately to the positive and negative parts of $\xi\big(\,\cdot\,;T_{\theta}^{(\nu)},T_{0}^{(\nu)}\big)$, one obtains
\begin{align}
\xi\big(\lambda;T_{\theta}^{(\nu)},T_{0}^{(\nu)}\big)&=\underset{\varepsilon\downarrow 0}\lim\,\f{1}{\pi}\Im\big[\ln\big(m_{\theta,\nu}(\lambda+i\varepsilon)\big)\big]+\widehat{C}_{\theta,\nu}\no\\
&=\underset{\varepsilon\downarrow 0}\lim\,\f{1}{\pi}\arg\big(m_{\theta,\nu}(\lambda+i\varepsilon)\big)+\widehat{C}_{\theta,\nu}\, \text{ for a.e.~} \lambda\in\bbR,
	\lb{5.49}
\end{align}
for some constant $\widehat{C}_{\theta,\nu}\in\bbR$ (in fact, $\widehat{C}_{\theta,\nu}=\pi^{-1}\Im(C_{\theta,\nu})$).  Note that
\begin{align}
&\Re\big(m_{\theta,\nu}(\lambda+i\varepsilon)\big)\lb{5.50}\\
&\quad = \Gamma(1+\nu)2^{2\nu+1}\nu\cot(\theta)+\Gamma(1-\nu)|\lambda+i\varepsilon|^{\nu}\cos\big(\nu[\pi-\arg(\lambda+i\varepsilon)]\big),\no\\
&\hspace*{8.1cm} \lambda\in \bbR,\, \varepsilon\in (0,\infty),\no
\end{align}
and
\begin{equation}\lb{5.51}
\begin{split}
\Im\big(m_{\theta,\nu}(\lambda+i\varepsilon)\big) = -\Gamma(1-\nu)|\lambda+i\varepsilon|^{\nu}\sin\big(\nu[\pi-\arg(\lambda+i\varepsilon)]\big)<0,&\\
\lambda\in \bbR,\, \varepsilon\in (0,\infty).&
\end{split}
\end{equation}

To compute the limit in \eqref{5.49}, one treats separately the cases $\lambda\in(-\infty,0)$ and $\lambda\in(0,\infty)$.  The case $\lambda=0$ may be dismissed as negligible since the spectral shift function is only determined almost everywhere.
\par
Note that
\begin{align}
\underset{\varepsilon\downarrow 0}\lim\,\arg(\lambda+i\varepsilon)=\pi,\quad\lambda\in(-\infty,0).
	\lb{5.52}
\end{align}
Define
\begin{align}
\Lambda_{\theta,\nu}^{-}(\lambda)&=\,\lim_{\varepsilon \downarrow 0}\Re\big(m_{\theta,\nu}(\lambda+i\varepsilon)\big)\no\\
&=\,\Gamma(1+\nu)2^{2\nu+1}\nu\cot(\theta)+\Gamma(1-\nu)|\lambda|^{\nu},\quad\lambda\in(-\infty,0),
	\lb{5.53}
\end{align}
and let
\begin{align}
\begin{split}
&\cP_{\theta,\nu}^{-}=\{\lambda\in(-\infty,0)\,|\,\Lambda_{\theta,\nu}^{-}(\lambda)>0\},\\
&\cN_{\theta,\nu}^{-}=\{\lambda\in(-\infty,0)\,|\,\Lambda_{\theta,\nu}^{-}(\lambda)<0\},\\
&\cZ_{\theta,\nu}^{-}=\{\lambda\in(-\infty,0)\,|\,\Lambda_{\theta,\nu}^{-}(\lambda)=0\}.
\end{split}
	\lb{5.54}
\end{align}
The sets in \eqref{5.54} allow one to decompose $(-\infty,0)$ into a disjoint union:
\begin{align}
(-\infty,0)=\cP_{\theta,\nu}^{-}\cup\cN_{\theta,\nu}^{-}\cup\cZ_{\theta,\nu}^{-}.
\end{align}
Note that $\cZ_{\theta,\nu}^{-}$ contains at most one element, so it has Lebesgue measure zero:
\begin{align}
\big|\cZ_{\theta,\nu}^{-}\big|=0.
	\lb{5.55}
\end{align}
If $\lambda\in\cP_{\theta,\nu}^{-}$, then
\begin{equation}
\Re\big(m_{\theta,\nu}(\lambda+i\varepsilon)\big)>0,\quad 0<\varepsilon\ll 1.
	\lb{5.56}
\end{equation}
Therefore, for $0<\varepsilon\ll 1$, $m_{\theta,\nu}(\lambda+i\varepsilon)$ has a positive real part and a negative imaginary part.  Thus, by \eqref{5.51}, \eqref{5.52}, and \eqref{5.56},
\begin{align}
\underset{\varepsilon\downarrow 0}\lim\,\f{1}{\pi}\,\arg\big(m_{\theta,\nu}(\lambda+i\varepsilon)\big)&=\f{1}{\pi}\,\underset{\varepsilon\downarrow 0}\lim\Bigg[2\pi -\arctan\left(\f{\big|\Im\big(m_{\theta,\nu}(\lambda+i\varepsilon)\big)\big|}{\Re\big(m_{\theta,\nu}(\lambda+i\varepsilon)\big)}\right)\Bigg]\no\\
&=\f{1}{\pi}[2\pi-0]=2,\quad\lambda\in\cP_{\theta,\nu}^{-}.
	\lb{5.57}
\end{align}
If $\lambda\in\cN_{\theta,\nu}^{-}$, then
\begin{equation}
\Re\big(m_{\theta,\nu}(\lambda+i\varepsilon)\big)<0,\quad 0<\varepsilon\ll 1.
	\lb{5.58}
\end{equation}
Therefore, for $0<\varepsilon\ll 1$, $m_{\theta,\nu}(\lambda+i\varepsilon)$ has a negative real part and a negative imaginary part.  Thus, by \eqref{5.51}, \eqref{5.52}, and \eqref{5.58},
\begin{align}
\underset{\varepsilon\downarrow 0}\lim\,\f{1}{\pi}\,\arg\big(m_{\theta,\nu}(\lambda+i\varepsilon)\big)&=\f{1}{\pi}\,\underset{\varepsilon\downarrow 0}\lim\,\Bigg[\pi + \arctan\left(\f{\big|\Im\big(m_{\theta,\nu}(\lambda+i\varepsilon)\big)\big|}{\big|\Re\big(m_{\theta,\nu}(\lambda+i\varepsilon)\big)\big|}\right)\Bigg]\no\\
&=\f{1}{\pi}[\pi+0]=1,\quad\lambda\in\cN_{\theta,\nu}^{-}.
	\lb{5.59}
\end{align}
By \eqref{5.49}, \eqref{5.55}, \eqref{5.57}, and \eqref{5.59},
\begin{align}
\xi\big(\lambda;T_{\theta}^{(\nu)},T_{0}^{(\nu)}\big)=2\chi_{\cP_{\theta,\nu}^{-}}(\lambda)+\chi_{\cN_{\theta,\nu}^{-}}(\lambda)+\widehat{C}_{\theta,\nu}\, \text{ for a.e.~$\lambda\in (-\infty,0)$.}
	\lb{5.60}
\end{align}
The explicit form for $\Lambda_{\theta,\nu}^{-}$ in \eqref{5.53} implies $\lim_{\lambda\to-\infty} \Lambda_{\theta,\nu}^{-}(\lambda)=\infty$, so one infers $(-\infty,-n_{\theta,\nu})\subset\cP_{\theta,\nu}^{-}$, for some $n_{\theta,\nu}\in\bbN$.  Thus, taking $\lambda$ sufficiently small in \eqref{5.60}, and applying \eqref{5.36}, one obtains
\begin{align}
\widehat{C}_{\theta,\nu}=-2.
	\lb{5.63}
\end{align}
Thus, \eqref{5.60} reduces to
\begin{align}
\xi\big(\lambda;T_{\theta}^{(\nu)},T_{0}^{(\nu)}\big)&=2\chi_{\cP_{\theta,\nu}^{-}}(\lambda)+\chi_{\cN_{\theta,\nu}^{-}}(\lambda)-2\no\\
&=2\chi_{\cP_{\theta,\nu}^{-}}(\lambda)+\chi_{\cN_{\theta,\nu}^{-}}(\lambda)-2\Big[\chi_{\cP_{\theta,\nu}^{-}}(\lambda)+\chi_{\cN_{\theta,\nu}^{-}}(\lambda)\Big]\no\\
&=-\chi_{\cN_{\theta,\nu}^{-}}(\lambda)\, \text{ for a.e.~$\lambda\in(-\infty,0)$}.
	\lb{5.64}
\end{align}
To obtain the second equality in \eqref{5.64}, one uses \eqref{5.55}, which implies
\begin{align}
\chi_{\cP_{\theta,\nu}^{-}}(\lambda)+\chi_{\cN_{\theta,\nu}^{-}}(\lambda)=1\, \text{ for a.e.~$\lambda\in(-\infty,0)$}.
	\lb{5.65}
\end{align}

Next, consider $\lambda\in(0,\infty)$.  In this case,
\begin{align}
\underset{\varepsilon\downarrow 0}\lim\,\arg(\lambda+i\varepsilon)=0,\quad \lambda\in(0,\infty).
	\lb{5.66}
\end{align}
Define
\begin{align}
\Lambda_{\theta,\nu}^{+}(\lambda)&=\underset{\varepsilon\downarrow 0}\lim\,\Re\big(m_{\theta,\nu}(\lambda+i\varepsilon) \big)\no\\
&=\Gamma(1+\nu)2^{2\nu+1}\nu\cot(\theta)+\Gamma(1-\nu)\lambda^{\nu}\cos(\nu\pi),\quad\lambda\in(0,\infty),
	\lb{5.67}
\end{align}
and let
\begin{align}
\begin{split}
\cP_{\theta,\nu}^{+}&=\{\lambda\in(0,\infty)\,|\,\Lambda_{\theta,\nu}^{+}(\lambda)>0\},\\
\cN_{\theta,\nu}^{+}&=\{\lambda\in(0,\infty)\,|\,\Lambda_{\theta,\nu}^{+}(\lambda)<0\},\\
\cZ_{\theta,\nu}^{+}&=\{\lambda\in(0,\infty)\,|\,\Lambda_{\theta,\nu}^{+}(\lambda)=0\}.
\end{split}
	\lb{5.68}
\end{align}
The sets in \eqref{5.68} allow one to express $(0,\infty)$ as a disjoint union:
\begin{equation}
(0,\infty) = \cP_{\theta,\nu}^+ \cup \cN_{\theta,\nu}^+ \cup \cZ_{\theta,\nu}^+.
	\lb{5.69}
\end{equation}
If $(\theta,\nu)\neq \big(\f{\pi}{2},\f{1}{2}\big)$, then $\cZ_{\theta,\nu}^{+}$ contains at most one element, so it has Lebesgue measure zero,
\begin{align}
\big|\cZ_{\theta,\nu}^{+}\big|=0,\quad (\theta,\nu)\in [(0,\pi)\times(0,1)]\backslash\big\{\big(\tfrac{\pi}{2},\tfrac{1}{2}\big)\big\},
	\lb{5.70}
\end{align}
while for $(\theta,\nu)=\big(\tfrac{\pi}{2},\tfrac{1}{2}\big)$ one infers that
\begin{equation}
\cZ_{\pi/2,1/2}^+ = (0,\infty)\quad \text{and}\quad \cP_{\pi/2,1/2}^+=\cN_{\pi/2,1/2}^+=\emptyset.
	\lb{5.71}
\end{equation}

If $\lambda\in\cP_{\theta,\nu}^{+}$, then
\begin{equation}
\Re\big(m_{\theta,\nu}(\lambda+i\varepsilon)\big)>0,\quad 0<\varepsilon\ll 1.
	\lb{5.72}
\end{equation}
Therefore, for $0<\varepsilon\ll 1$, $m_{\theta,\nu}(\lambda+i\varepsilon)$ has a positive real part and a negative imaginary part.  Thus, by \eqref{5.50}, \eqref{5.51}, \eqref{5.66}, and \eqref{5.72},
\begin{align}
\underset{\varepsilon\downarrow 0}\lim\,\f{1}{\pi}\arg\big(m_{\theta,\nu}(\lambda+i\varepsilon) \big)&=\f{1}{\pi}\,\underset{\varepsilon\downarrow 0}\lim\Bigg[2\pi -\arctan\left(\f{\big|\Im\big(m_{\theta,\nu}(\lambda+i\varepsilon)\big)\big|}{\Re\big(m_{\theta,\nu}(\lambda+i\varepsilon)\big)}\right)\Bigg]\lb{5.73}\\
&=2+\Xi_{\theta,\nu}(\lambda),\quad\lambda\in\cP_{\theta,\nu}^{+}.\no
\end{align}
If $\lambda\in\cN_{\theta,\nu}^{+}$, then
\begin{equation}
\Re\big(m_{\theta,\nu}(\lambda+i\varepsilon)\big)<0,\quad 0<\varepsilon\ll 1.
	\lb{5.74}
\end{equation}
Therefore, for $0<\varepsilon\ll 1$, $m_{\theta,\nu}(\lambda+i\varepsilon)$ has a negative real part and a negative imaginary part.  Thus, by \eqref{5.50}, \eqref{5.51}, \eqref{5.66}, and \eqref{5.74},
\begin{align}
\underset{\varepsilon\downarrow 0}\lim\,\f{1}{\pi}\arg\big( m_{\theta,\nu}(\lambda+i\varepsilon)\big)&=\f{1}{\pi}\underset{\varepsilon\downarrow 0}\lim\Bigg[\pi+\arctan\left(\f{\big|\Im\big(m_{\theta,\nu}(\lambda+i\varepsilon)\big)\big|}{\big|\Re\big(m_{\theta,\nu}(\lambda+i\varepsilon)\big)\big|}\right)\Bigg]\lb{5.75}\\
&=1+\Xi_{\theta,\nu}(\lambda),\quad\lambda\in\cN_{\theta,\nu}^{+}.\no
\end{align}
By \eqref{5.49}, \eqref{5.63}, \eqref{5.70}, \eqref{5.73}, and \eqref{5.75},
\begin{align}
&\xi\big(\lambda;T_{\theta}^{(\nu)},T_{0}^{(\nu)}\big)
	\lb{5.76}\\
&\quad=\big[2+\Xi_{\theta,\nu}(\lambda)\big]\chi_{\cP_{\theta,\nu}^{+}}(\lambda)+\big[1+\Xi_{\theta,\nu}(\lambda)\big]\chi_{\cN_{\theta,\nu}^{+}}(\lambda)-2\left[\chi_{\cP_{\theta,\nu}^{+}}(\lambda)+\chi_{\cN_{\theta,\nu}^{+}}(\lambda)\right]\no\\
&\quad=-\chi_{\cN_{\theta,\nu}^{+}}(\lambda)+\Xi_{\theta,\nu}(\lambda)\, \text{ for a.e.~$\lambda\in (0,\infty)$},\, (\theta,\nu)\in [(0,\pi)\times(0,1)]\backslash\{(\tfrac{\pi}{2},\tfrac{1}{2})\}.\no
\end{align}
The extreme case $(\theta,\nu)=\big(\tfrac{\pi}{2},\tfrac{1}{2}\big)$ will be addressed below in the proof of item $(v)$.

With these general considerations out of the way, we analyze the individual cases $(i)$--$(v)$.\\[1mm]

\noindent
$(i)$: If $\theta\in\big(0,\f{\pi}{2}\big]$, $\nu\in\big(0,\f{1}{2}\big]$, and $(\theta,\nu)\neq \big(\tfrac{\pi}{2},\tfrac{1}{2}\big)$, then 
\begin{equation}
(\cot(\theta),\cos(\nu\pi))\in ([0,\infty)\times [0,\infty))\backslash\{(0,0)\}.
	\lb{5.77}
\end{equation}
In particular, the explicit forms for $\Lambda_{\theta,\nu}^{\pm}$ in \eqref{5.53} and \eqref{5.67} imply
\begin{align}
\Lambda_{\theta,\nu}^{-}(\lambda)>0,\quad&\lambda\in(-\infty,0),
	\lb{5.78}\\
\Lambda_{\theta,\nu}^{+}(\lambda)>0,\quad&\lambda\in(0,\infty),
	\lb{5.79}
\end{align}
so that
\begin{align}
\cN_{\theta,\nu}^{\pm}=\emptyset.
	\lb{5.80}
\end{align}
Therefore, \eqref{5.64}, \eqref{5.76}, and \eqref{5.80} imply \eqref{5.40}.\\[1mm]

$(ii)$: If $\theta\in\big(0,\f{\pi}{2}\big]$ and $\nu\in\big(\f{1}{2},1\big)$, then $\cot(\theta)\geq 0$ and $\cos(\nu\pi)<0$.  In particular, the explicit forms for $\Lambda_{\theta,\nu}^{\pm}$ in \eqref{5.53} and \eqref{5.67} imply
\begin{align}
&\Lambda_{\theta,\nu}^{-}(\lambda)>0,\quad\lambda\in(-\infty,0),
	\lb{5.81}\\
&\text{$\Lambda_{\theta,\nu}^{+}(\lambda)<0$ if and only if $\lambda\in(\lambda_{\theta,\nu},\infty)$},
	\lb{5.82}
\end{align}
so that
\begin{align}
\cN_{\theta,\nu}^{-}=\emptyset\quad\text{and}\quad \cN_{\theta,\nu}^{+}=(\lambda_{\theta,\nu},\infty).
	\lb{5.83}
\end{align}
Therefore, \eqref{5.64}, \eqref{5.76}, and \eqref{5.83} imply \eqref{5.41}.\\[1mm]

$(iii)$: If $\theta\in\big(\f{\pi}{2},\pi\big)$ and $\nu\in\big(0,\f{1}{2}\big)$, then $\cot(\theta)<0$ and $\cos(\nu\pi)>0$.  In particular, the explicit forms for $\Lambda_{\theta,\nu}^{\pm}$ in \eqref{5.53} and \eqref{5.67} imply
\begin{align}
&\text{$\Lambda_{\theta,\nu}^{-}(\lambda)<0$ if and only if $\lambda\in(e_{\theta,\nu},0)$},
	\lb{5.84}\\
&\text{$\Lambda_{\theta,\nu}^{+}(\lambda)<0$ if and only if $\lambda\in(0,\lambda_{\theta,\nu})$},
	\lb{5.85}
\end{align}
so that
\begin{align}
\cN_{\theta,\nu}^{-}=(e_{\theta,\nu},0)\quad\text{and}\quad\cN_{\theta,\nu}^{+}=(0,\lambda_{\theta,\nu}).
	\lb{5.86}
\end{align}
Therefore, \eqref{5.64}, \eqref{5.76}, and \eqref{5.86} imply \eqref{5.42}.\\[1mm]

$(iv)$: If $\theta\in\big(\f{\pi}{2},\pi\big)$ and $\nu\in\big[\f{1}{2},1\big)$, then $\cot(\theta)<0$ and $\cos(\nu\pi)\leq 0$.  In particular, the explicit forms for $\Lambda_{\theta,\nu}^{\pm}$ in \eqref{5.53} and \eqref{5.67} imply
\begin{align}
&\text{$\Lambda_{\theta,\nu}^{-}(\lambda)<0$ if and only if $\lambda\in(e_{\theta,\nu},0)$},
	\lb{5.87}\\
&\text{$\Lambda_{\theta,\nu}^{+}(\lambda)<0$ if and only if $\lambda\in(0,\infty)$},
	\lb{5.88}
\end{align}
so that
\begin{align}
\cN_{\theta,\nu}^{-}=(e_{\theta,\nu},0),\quad\cN_{\theta,\nu}^{+}=(0,\infty).
	\lb{5.89}
\end{align}
Therefore, \eqref{5.64}, \eqref{5.76}, and \eqref{5.89} imply \eqref{5.43}.\\[1mm]

$(v)$:  If $(\theta,\nu)=\big(\tfrac{\pi}{2},\tfrac{1}{2}\big)$, then $\cot(\theta)=0$ and $\cos(\nu\pi)=0$.  The explicit form for $\Lambda_{\pi/2,1/2}^{-}$ in \eqref{5.53} implies
\begin{align}
\Lambda_{\pi/2,1/2}^-(\lambda)>0,\quad \lambda\in (-\infty,0).
	\lb{5.90}
\end{align}
Therefore,
\begin{equation}\lb{5.91}
\cN_{\pi/2,1/2}^- = \emptyset,
\end{equation}
and \eqref{5.64} implies
\begin{equation}\lb{5.92}
\xi\big(\lambda;T_{\pi/2}^{(1/2)},T_0^{(1/2)}\big)=0\, \text{ for a.e.~$\lambda\in (-\infty,0)$}.
\end{equation}

In addition, \eqref{5.49} and \eqref{5.63} imply
\begin{align}
\xi\big(\lambda;T_{\pi/2}^{(1/2)},T_0^{(1/2)}\big)=\underset{\varepsilon\downarrow 0}\lim\,\f{1}{\pi}\arg\big(m_{\pi/2,1/2}(\lambda+i\varepsilon) \big)-2\, \text{ for a.e.~} \lambda\in(0,\infty).\lb{5.93}
\end{align}
By \eqref{5.50} and \eqref{5.51} with $\theta=\tfrac{\pi}{2}$ and $\nu=\tfrac{1}{2}$, one infers that $m_{\pi/2,1/2}(\lambda+i\varepsilon)$ has a positive real part and a negative imaginary part for every $\varepsilon>0$ since $\arg(\lambda+i\varepsilon)\in (0,\tfrac{\pi}{2})$ when $\lambda\in (0,\infty)$.  Thus,
\begin{align}
\xi\big(\lambda;T_{\pi/2}^{(1/2)},T_0^{(1/2)}\big)&=\underset{\varepsilon\downarrow 0}\lim\,\f{1}{\pi}\arg\big(m_{\pi/2,1/2}(\lambda+i\varepsilon) \big)-2\no\\
&=\lim_{\varepsilon\downarrow 0} \frac{1}{\pi}\left[ 2\pi - \arctan\left(\frac{\sin\big(\tfrac{1}{2}[\pi-\arg(\lambda+i\varepsilon)]\big)}{\cos\big(\tfrac{1}{2}[\pi-\arg(\lambda+i\varepsilon)]\big)}\right)\right]-2\no\\
&=\lim_{\varepsilon\downarrow 0} \frac{1}{\pi}\left[ 2\pi - \frac{1}{2}[\pi-\arg(\lambda+i\varepsilon)]\right]-2\no\\
&=  \frac{1}{\pi}\left[ 2\pi - \frac{1}{2}[\pi-0]\right]-2\no\\
&= -\frac{1}{2}\, \text{ for a.e.~} \lambda\in(0,\infty).\lb{5.94}
\end{align}
Therefore, \eqref{5.92} and \eqref{5.94} imply \eqref{5.44}.
\end{proof}

Proposition \ref{p5.3} implies that for each fixed $\nu\in (0,1)$, $\sigma\big(T_{\theta}^{(\nu)}\big)\subseteq [0,\infty)$ if and only if $\theta\in \big[0,\tfrac{\pi}{2}\big]$.  Therefore, Proposition \ref{p5.3} recovers the following characterization of the nonnegative self-adjoint extensions of $T_{\min}^{(\nu)}$ obtained in \cite[Corollary 5.1]{AB15}:

\begin{corollary}	\lb{c5.4}
If $\nu \in (0,1)$, then $T_{\theta}^{(\nu)}$ is a nonnegative self-adjoint extension of $T_{\min}^{(\nu)}$ if and only if $\theta\in \big[0,\tfrac{\pi}{2}\big]$.
\end{corollary}

In the special case $(\theta,\nu)=\big(\tfrac{\pi}{2},\tfrac{1}{2}\big)$, the operators $T_0^{(1/2)}$ and $T_{\pi/2}^{(1/2)}$ are simply the Dirichlet and Neumann Laplacians, respectively, on $(0,\infty)$,
\begin{equation}
-\Delta_D:=T_0^{(1/2)}\quad \text{and}\quad -\Delta_N:=T_{\pi/2}^{(1/2)}.
\end{equation}
The trace identity \eqref{A.5} and the simple structure of the spectral shift function \eqref{5.44} in this case permit one to easily calculate the trace of $f(-\Delta_N) - f(-\Delta_D)$ for any $f\in \mathfrak{F}(\bbR)$ (cf.~Definition \ref{dA.1} and Remark \ref{rA.2}) in terms of the values $f(0)$ and $f(\infty)$ (i.e., the limiting value of $f$ at $\infty$):

\begin{corollary}\lb{c5.5}
If $f\in \mathfrak{F}(\bbR)$ and  $f(\infty):=\lim_{\lambda\to \infty}f(\lambda)$, then
\begin{equation}\lb{5.94a}
\big[f(-\Delta_N) - f(-\Delta_D)\big] \in \cB_1\big(L^2((0,\infty);dx)\big)
\end{equation}
and
\begin{equation}
\tr_{(0,\infty)}\big(f(-\Delta_N) - f(-\Delta_D) \big) = \frac{1}{2}[f(0)-f(\infty)].
	\lb{5.95}
\end{equation}
\end{corollary}
\begin{proof}
Let $f\in \mathfrak{F}(\bbR)$.  The containment \eqref{5.94a} follows from \eqref{A.5a}.  By \eqref{A.5} and \eqref{5.44},
\begin{equation}\lb{5.96a}
\tr_{(0,\infty)}\big(f(-\Delta_N) - f(-\Delta_D) \big) = -\frac{1}{2}\int_0^{\infty}f'(\lambda)\, d\lambda = \frac{1}{2}[f(0)-f(\infty)].
\end{equation}
\end{proof}

\begin{remark}\lb{r5.7}
When $\nu = \tfrac{1}{2}$ in \eqref{5.2}, the resulting differential expression $\tau_{1/2}$ is regular at the endpoint $x=0$.  In this case, the spectral shift function for $T_{\theta}^{(1/2)}$ and $T_0^{(1/2)}$ may be recovered as a special case of \cite[Lemma 2.3]{GS96}, which actually applies to more general Schr\"odinger operators of the form $-d^2/dx^2 + V(x)$ on $(0,\infty)$ with $V\in L^1((0,\ell);dx)$ for all $\ell\in (0,\infty)$.\hfill $\diamond$
\end{remark}

\subsection{The case $\nu=0$}	\lb{s5.2}

The case $\nu=0$ is more nuanced, as it may not be analyzed by merely taking $\nu\downarrow 0$ in the formulas from Section \ref{s5.1}.  Notice that, in fact, some of the formulas from Section \ref{s5.1} become highly singular in the limit $\nu\downarrow 0$.  Instead, one must adopt a different boundary condition basis and solutions $s_{z,0}$, $c_{z,0}$.  Following \cite[Section 7]{EK07}, one fixes a boundary condition basis $\{\phi_{0,0},\psi_{0,0}\}$ at $a=0$ by choosing functions $\phi_{0,0},\psi_{0,0}\in \dom\big(T_{\max}^{(0)}\big)$ which vanish in a neighborhood of $\infty$ and satisfy
\begin{align}
\phi_{0,0}(x)=x^{1/2}\quad\text{and}\quad \psi_{0,0}(x)=-x^{1/2}\ln(x),\quad x\in (0,1).
	\lb{5.96}
\end{align}
This is possible by the Naimark patching lemma \cite[Chapter V, Section 17.3, Lemma 2]{Na68}.  The self-adjoint extensions of $T_{\min}^{(0)}$ are parametrized according to Theorem \ref{t2.19} as a one-parameter family $\big\{T_{\theta}^{(0)}\big\}_{\theta\in[0,\pi)}$, where for $\theta\in [0,\pi)$,
\begin{align}
&T_{\theta}^{(0)}f = T_{\max}^{(0)}f,
	\lb{5.97}\\
&f\in \dom\big(T_{\theta}^{(0)}\big) = \big\{g\in \dom\big(T_{\max}^{(0)}\big)\,\big|\, \cos(\theta)[g,\phi_{0,0}](0) + \sin(\theta)[g,\psi_{0,0}](0)=0\big\}.\no
\end{align}
By \cite[eq.~(7.5)]{EK07}, $T_0^{(0)}$ is the Friedrichs extension of $T_{\min}^{(0)}$.  In addition, by \cite[Proposition 5.3 $(ii)$]{AB15}, $T_0^{(0)}$ is also the Krein--von Neumann extension.  Thus, the Friedrichs and Krein--von Neumann extensions coincide in this case, and it follows that $T_0^{(0)}$ is the only nonnegative self-adjoint extension of $T_{\min}^{(0)}$ (cf.~\cite[Corollary 5.2]{AB15}).

For $z\in \bbC\backslash\{0\}$, a basis of solutions $\{s_{z,0},c_{z,0}\}$ to the equation
\begin{equation}
\tau_{0}y=zy\quad \text{on}\quad (0,\infty)
	\lb{5.99}
\end{equation}
is fixed by setting
\begin{align}
s_{z,0}(x)&=-x^{1/2}J_{0}\big(z^{1/2}x\big), \lb{5.100}\\
c_{z,0}(x)&=-\f{\pi}{2}x^{1/2}Y_{0}\big(z^{1/2}x\big)+\left[\ln\left(\f{1}{2}z^{1/2}\right)+\gamma\right]x^{1/2}J_{0}\big(z^{1/2}x\big),\quad x\in (0,\infty),\no
\end{align}
where $\gamma= 0.577215664...$~is the Euler--Mascheroni constant.  Analogous to \eqref{5.14} and \eqref{5.15}, the functions $s_{z,0}$ and $c_{z,0}$ satisfy the generalized boundary conditions (cf. \cite[eqs.~(4.9) \& (4.10)]{AB15} and \cite[Section 7.2]{EK07})
\begin{equation}\lb{5.14ab}
\begin{split}
&[s_{z,0},\phi_{0,0}](0) = 0,\quad [s_{z,0},\psi_{0,0}](0)= 1,\\
&[c_{z,0},\phi_{0,0}](0) = 1,\quad [c_{z,0},\psi_{0,0}](0)=0,\quad z\in \bbC\backslash\{0\}.
\end{split}
\end{equation}
The Weyl--Titchmarsh solution is based on the Hankel function of the first kind, as for all $z\in\rho\big(T_0^{(0)}\big)$,
\begin{align}
(\,\cdot\,)^{1/2}H_{0}^{(1)}\big(z^{1/2}\,\cdot\,\big)\in L^{2}((0,\infty);dx).
	\lb{5.101}
\end{align}
The function in \eqref{5.101} admits an expansion in terms of the basis $\{s_{z,0},c_{z,0}\}$:
\begin{equation}\lb{5.102}
\begin{split}
x^{1/2}H_0^{(1)}\big(z^{1/2}x\big) = -\frac{2i}{\pi}c_{z,0}(x) - \left( 1+ \frac{2i}{\pi}\left[\ln\left(\frac{1}{2}z^{1/2}\right)+\gamma\right]\right)s_{z,0}(x),&\\
x\in (0,\infty),\,z\in \rho\big(T_0^{(0)}\big),&
\end{split}
\end{equation}
from which one deduces
\begin{align}
\Big[(\,\cdot\,)^{1/2}H_{0}^{(1)}\big(z^{1/2}\,\cdot\,\big),\phi_{0,0}\Big](0)=-\f{2i}{\pi},\quad z\in \rho\big(T_0^{(0)}\big).
	\lb{5.103}
\end{align}
As a consequence, one obtains
\begin{align}
w_{z,0}(x)&=i\f{\pi}{2}x^{1/2}H_{0}^{(1)}\big(z^{1/2}x\big)\lb{5.104}\\
&=c_{z,0}(x) + \left[\ln\left(\frac{1}{2}z^{1/2}\right)+\gamma-\frac{i\pi}{2}\right]s_{z,0}(x),\quad x\in (0,\infty),\,z\in \rho\big(T_0^{(0)}\big),\no
\end{align}
and a calculation, omitted here, reveals
\begin{align}
[w_{z,0},\psi_{0,0}](0)= \ln\left(\frac{1}{2}z^{1/2}\right)+\gamma-\frac{i\pi}{2},\quad z\in \rho\big(T_0^{(0)}\big).
	\lb{5.105}
\end{align}
In particular, by \eqref{3.11},
\begin{equation}\lb{5.106}
\begin{split}
k_{\theta,0}(z)=\cot(\theta)+\ln\left(\f{1}{2}z^{1/2}\right)+\gamma-\f{i\pi}{2},&\\
z\in \rho\big(T_0^{(0)}\big)\cap \rho\big(T_{\theta}^{(0)}\big),\, \theta\in (0,\pi).&
\end{split}
\end{equation}
With $w_{z,0}$ given by \eqref{5.104} and $k_{\theta,0}(\,\cdot\,)$ given by \eqref{5.106}, the right-hand side in \eqref{3.12} is determined.  To compute the right-hand side in \eqref{3.14}, it is necessary to calculate the inner product $\left< w_{\overline{z},0},w_{z,0}\right>_{(0,\infty)}$ for $z\in \rho\big(T_0^{(0)}\big)$.  Applying the definition of the inner product, \eqref{3.3}, and the order reflection formula \cite[eq.~9.1.6]{AS72}, one computes for $z\in \rho\big(T_0^{(0)}\big)$,
\begin{align}
\left<w_{\ol{z},0},w_{z,0}\right>_{(0,\infty)}&=-\f{\pi^{2}}{4}\int_{0}^{\infty}x\big[H_{0}^{(1)}\big(z^{1/2}x\big)\big]^{2}\,dx\no\\
&=-\f{\pi^{2}}{4}\Bigg[\f{1}{2}x^{2}\Big(\big[H_{0}^{(1)}\big(z^{1/2}x\big)\big]^{2}+\big[H_{1}^{(1)}\big(z^{1/2}x\big)\big]^{2}\Big)\Bigg]_{0}^{\infty}.
	\lb{5.107}
\end{align}

A calculation similar to the one in \eqref{5.25}, reveals that the upper limit at $\infty$ is zero.  However, the low-argument asymptotics for $H_{0}^{(1)}$ differ from those of nonzero order, so one applies \cite[eq.~9.1.8]{AS72} to calculate
\begin{align}
&\underset{x\downarrow 0}\lim\left[\f{1}{2}x^{2}\left(\big[H_{0}^{(1)}\big(z^{1/2}x\big)\big]^{2}+\big[H_{1}^{(1)}\big(z^{1/2}x\big)\big]^{2}\right)\right]\no\\
&\quad=\underset{x\downarrow 0}\lim\left[\f{1}{2}x^{2}\left(-\f{4}{\pi^{2}}\ln^{2}\big(z^{1/2}x\big)-\f{1}{\pi^{2}\big(\f{1}{2}z^{1/2}x\big)^{2}}\right)\right]\no\\
&\quad =\underset{x\downarrow 0}\lim\left[-\f{2}{\pi^{2}}\left(x^{2}\ln^{2}\big(z^{1/2}x\big)+\f{1}{z}\right)\right]\no\\
&\quad=-\f{2}{\pi^{2}z},\quad z\in \rho\big(T_0^{(0)}\big).
	\lb{5.108}
\end{align}
Substitution of \eqref{5.108} into \eqref{5.107} yields
\begin{align}
\left<w_{\ol{z},0},w_{z,0}\right>_{(0,\infty)}=-\f{\pi^{2}}{4}\left[\f{2}{\pi^{2}z}\right]=-\f{1}{2z},\quad z\in \rho\big(T_0^{(0)}\big).
	\lb{5.109}
\end{align}
Finally, \eqref{5.106} and \eqref{5.109} permit one to explicitly compute the right-hand side in \eqref{3.14}.  The results are summarized in:

\begin{proposition}	\lb{p5.6}
If $\theta\in (0,\pi)$ and $z\in \rho\big(T_0^{(0)}\big)\cap\rho\big(T_{\theta}^{(0)}\big)$, then
\begin{align}
\begin{split}
&\big(T_{\theta}^{(0)}-zI_{(0,\infty)}\big)^{-1}-\big(T_0^{(0)}-zI_{(0,\infty)}\big)^{-1}\\
&\quad =\left[ \cot(\theta)+\ln\left(\f{1}{2}z^{1/2}\right)+\gamma-\f{i\pi}{2}\right]^{-1}\left<w_{\overline{z},0},\,\cdot\,\right>_{(0,\infty)}w_{z,0},
\end{split}
	\lb{5.110}
\end{align}
where $w_{z,0}$ is defined by \eqref{5.104}.  In particular,
\begin{align}
\Big[\big(T_{\theta}^{(0)}-zI_{(0,\infty)}\big)^{-1}-\big(T_0^{(0)}-zI_{(0,\infty)}\big)^{-1}	\Big]	\in \cB_1\big(L^2((0,\infty);dx)\big),
	\lb{5.111}
\end{align}
and the following trace formula holds:
\begin{align}
\begin{split}
&\tr_{(0,\infty)}\Big(\big(T_{\theta}^{(0)}-zI_{(0,\infty)}\big)^{-1}-\big(T_0^{(0)}-zI_{(0,\infty)}\big)^{-1}\Big)\\
&\quad=-\f{1}{2z\big(\cot(\theta)+\ln\left(\f{1}{2}z^{1/2}\right)+\gamma-\f{i\pi}{2}\big)}.
\end{split}
	\lb{5.112}
\end{align}
\end{proposition}

Using the trace formula \eqref{5.112}, we explicitly compute the spectral shift function $\xi\big(\,\cdot\,;T_{\theta}^{(0)},T_{0}^{(0)}\big)$, $\theta\in (0,\pi)$.

\begin{proposition}\lb{p5.8}
If $\theta\in (0,\pi)$ and $e_{\theta,0} = -4e^{-2[\cot(\theta)+\gamma]}$, then for a.e.~$\lambda\in \bbR$,
\begin{align}
&\xi\big(\lambda;T_{\theta}^{(0)},T_{0}^{(0)}\big)\lb{5.114a}\\
&= -\chi_{(e_{\theta,0},-e_{\theta,0})}(\lambda) - \dfrac{1}{\pi}\chi_{(0,\infty)}(\lambda)\arctan\bigg(\dfrac{\pi}{2\cot(\theta)-2\,\ln(2)+2\gamma + \ln|\lambda|} \bigg).\no
\end{align}
In particular, $T_{\theta}^{(0)}$ has a single negative eigenvalue $e_{\theta,0}$ of multiplicity one.
\end{proposition}
\begin{proof}
Let $\theta\in (0,\pi)$.  Combining \eqref{5.37}, \eqref{5.112}, and rewriting the logarithm on the right-hand side in \eqref{5.112}, one obtains
\begin{equation}\lb{5.114}
\begin{split}
\int_{\bbR}\frac{\xi\big(\lambda;T_{\theta}^{(0)},T_0^{(0)}\big)}{(\lambda-z)^2}\, d\lambda = \frac{1}{z(2\cot(\theta)-2\,\ln(2)+\ln(z)+2\gamma - i\pi)},&\\
z\in \bbC\backslash\bbR.&
\end{split}
\end{equation}
The identity in \eqref{5.114} may be recast as
\begin{align}
&\frac{d}{dz}\int_{\bbR}\xi\big(\lambda;T_{\theta}^{(0)},T_0^{(0)}\big)\bigg[\frac{1}{\lambda-z}-\frac{\lambda}{1+\lambda^2}\bigg]\, d\lambda= \frac{d}{dz} \ln\big(m_{\theta,0}(z)\big),\quad z\in \bbC\backslash \bbR,\lb{5.115}
\end{align}
where
\begin{equation}\lb{5.116}
m_{\theta,0}(z):= 2\cot(\theta)-2\,\ln(2)+\ln(z)+2\gamma - i\pi,\quad z\in \bbC\backslash\bbR.
\end{equation}
Therefore,
\begin{align}
&\int_{\bbR}\xi\big(\lambda;T_{\theta}^{(0)},T_0^{(0)}\big)\bigg[\frac{1}{\lambda-z}-\frac{\lambda}{1+\lambda^2}\bigg]\, d\lambda =\ln\big(m_{\theta,0}(z)\big) + C_{\theta,0},\quad z\in \bbC\backslash \bbR,\lb{5.117}
\end{align}
for some constant $C_{\theta,0}\in \bbC$.  By the Stieltjes inversion formula \cite[Theorem 2.2 $(v)$]{GT00}, applied separately to the positive and negative parts of $\xi\big(\,\cdot\,;T_{\theta}^{(0)},T_{0}^{(0)}\big)$, one obtains
\begin{equation}\lb{5.118}
\xi\big(\lambda;T_{\theta}^{(0)},T_{0}^{(0)}\big) = \frac{1}{\pi}\lim_{\varepsilon\downarrow0}\arg\big(m_{\theta,0}(\lambda+i\varepsilon)\big) +\widehat C_{\theta,0}\, \text{ for a.e.~$\lambda\in \bbR$},
\end{equation}
for some constant $\widehat C_{\theta,0}\in \bbR$ (in fact, $\widehat C_{\theta,0} = \pi^{-1}\Im(C_{\theta,0})$).

The representation in \eqref{5.116} implies
\begin{equation}\lb{5.119}
\begin{split}
\Re\big(m_{\theta,0}(\lambda+i\varepsilon)\big) &= 2\cot(\theta)-2\,\ln(2)+2\gamma + \ln|\lambda+i\varepsilon|,\\
\Im\big(m_{\theta,0}(\lambda+i\varepsilon)\big) &= -\big[ \pi - \arg(\lambda+i\varepsilon)\big], \quad \lambda\in \bbR,\, \varepsilon\in (0,\infty).
\end{split}
\end{equation}
By \eqref{5.119}, $m_{\theta,0}(\lambda+i\varepsilon)$ has a negative imaginary part for every $\lambda\in \bbR$, $\varepsilon\in(0,\infty)$.  Define
\begin{equation}\lb{5.120}
\Lambda_{\theta,0}(\lambda) = \lim_{\varepsilon\downarrow 0}\Re\big(m_{\theta,0}(\lambda+i\varepsilon)\big) = 2\cot(\theta)-2\,\ln(2)+2\gamma + \ln|\lambda|,\quad \lambda\in \bbR\backslash\{0\}.
\end{equation}
One then infers
\begin{equation}\lb{5.121}
\begin{cases}
\Lambda_{\theta,0}(\lambda)<0,\quad \lambda\in (e_{\theta,0},-e_{\theta,0})\backslash\{0\},\\
\Lambda_{\theta,0}(\lambda)>0,\quad \lambda\in \bbR\backslash (e_{\theta,0},-e_{\theta,0}).
\end{cases}
\end{equation}
If $\lambda\in (e_{\theta,0},-e_{\theta,0})\backslash\{0\}$, then \eqref{5.120} and \eqref{5.121} imply $\Re\big(m_{\theta,0}(\lambda+i\varepsilon)\big)<0$ for $0<\varepsilon\ll 1$.  Therefore,
\begin{align}
\arg\big(m_{\theta,0}(\lambda+i\varepsilon)\big) = \pi - \arctan \bigg(\frac{\pi- \arg(\lambda+i\varepsilon)}{2\cot(\theta)-2\,\ln(2)+2\gamma + \ln|\lambda+i\varepsilon|}\bigg),&\lb{5.122}\\
0<\varepsilon\ll1,&\no
\end{align}
and as a consequence,
\begin{align}
&\dfrac{1}{\pi}\lim_{\varepsilon\downarrow 0}\arg\big(m_{\theta,0}(\lambda+i\varepsilon)\big)\lb{5.123}\\
&\quad =
\begin{cases}
1,& \lambda\in (e_{\theta,0},0),\\
1 - \dfrac{1}{\pi}\arctan\bigg(\dfrac{\pi}{2\cot(\theta)-2\,\ln(2)+2\gamma + \ln|\lambda|} \bigg),& \lambda\in (0,-e_{\theta,0}).
\end{cases}\no
\end{align}
If $\lambda\in \bbR\backslash(e_{\theta,0},-e_{\theta,0})$, then \eqref{5.120} and \eqref{5.121} imply $\Re\big(m_{\theta,0}(\lambda+i\varepsilon)\big)>0$ for $0<\varepsilon\ll 1$.  Therefore,
\begin{align}
\arg\big(m_{\theta,0}(\lambda+i\varepsilon)\big) = 2\pi - \arctan \bigg(\frac{\pi- \arg(\lambda+i\varepsilon)}{2\cot(\theta)-2\,\ln(2)+2\gamma + \ln|\lambda+i\varepsilon|}\bigg),&\lb{5.124}\\
0<\varepsilon\ll1,&\no
\end{align}
and as a consequence,
\begin{align}
&\dfrac{1}{\pi}\lim_{\varepsilon\downarrow 0}\arg\big(m_{\theta,0}(\lambda+i\varepsilon)\big)\lb{5.125}\\
&\quad =
\begin{cases}
2,& \lambda\in (-\infty,e_{\theta,0}),\\
2 - \dfrac{1}{\pi}\arctan\bigg(\dfrac{\pi}{2\cot(\theta)-2\,\ln(2)+2\gamma + \ln|\lambda|} \bigg),& \lambda\in (-e_{\theta,0},\infty).
\end{cases}\no
\end{align}
By \eqref{5.118} and \eqref{5.125},
\begin{equation}\lb{5.126}
\xi\big(\lambda;T_{\theta}^{(0)},T_{0}^{(0)}\big) = 2 + \widehat{C}_{\theta,0}\,\text{ for a.e.~$\lambda\in (-\infty,e_{\theta,0})$}.
\end{equation}
Since the spectral shift function vanishes a.e.~in a neighborhood of $-\infty$, the relation in \eqref{5.126} implies
\begin{equation}\lb{5.127}
\widehat{C}_{\theta,0} = -2.
\end{equation}
Upon combining \eqref{5.118}, \eqref{5.123}, \eqref{5.125}, and \eqref{5.127}, the identity in \eqref{5.114a} holds for a.e.~$\lambda\in \bbR$.

Finally, the claim regarding the negative eigenvalue of $T_{\theta}^{(0)}$ is a consequences of \eqref{5.10}, \eqref{5.34}, and Lemma \ref{lA.3} in conjunction with \eqref{5.114a}.
\end{proof}

Proposition \ref{p5.8} implies that $\sigma\big(T_{\theta}^{(0)}\big)\subseteq [0,\infty)$ if and only if $\theta=0$.  Therefore, Proposition \ref{p5.8} recovers the following characterization of the nonnegative self-adjoint extensions of $T_{\min}^{(0)}$ obtained in \cite[Corollary 5.2]{AB15}:

\begin{corollary}	\lb{c5.11}
$T_0^{(0)}$ is the unique nonnegative self-adjoint extension of $T_{\min}^{(0)}$.
\end{corollary}

\appendix
\section{The Spectral Shift Function} \lb{sA}
\renewcommand{\theequation}{A.\arabic{equation}}
\renewcommand{\thetheorem}{A.\arabic{theorem}}
\setcounter{theorem}{0} \setcounter{equation}{0}

In this appendix, we recall a few basic facts on the spectral shift function for a pair of resolvent comparable self-adjoint operators which are bounded from below in a Hilbert space $\cH$.  These facts are used extensively in Section \ref{s5.1} in connection with Proposition \ref{p5.3}.

Suppose that $S$ and $S_0$ are self-adjoint operators which are bounded from below in $\cH$ and {\it resolvent comparable} in the sense that for some (hence, for all) $z\in \bbC\backslash \bbR$,
\begin{equation}\lb{A.1}
\big[(S-zI_{\cH})^{-1} - (S_0-zI_{\cH})^{-1}\big] \in \cB_1(\cH).
\end{equation}

\begin{definition}\lb{dA.1}
The class $\mathfrak{F}(\bbR)$ is defined to be the set of all functions $f:\bbR\to \bbC$ which possess two locally bounded derivatives and satisfy
\begin{equation}\lb{A.2}
\big|\big(\lambda^2f'(\lambda)\big)'\big|\leq C\lambda^{-1-\varepsilon},\quad \lambda\in (0,\infty),
\end{equation}
for some $C=C(f)\in (0,\infty)$ and $\varepsilon=\varepsilon(f)\in (0,\infty)$.
\end{definition}

\begin{remark}\lb{rA.2}
Note that if $f\in \mathfrak{F}(\bbR)$, then $f(\infty):=\lim_{\lambda\to \infty}f(\lambda)$ exists.\hfill $\diamond$
\end{remark}

The condition \eqref{A.1} guarantees the existence of a unique real-valued function (cf., e.g., \cite[Theorem 9.28]{Sc12} and \cite[Sections 8.7 \& 8.8]{Ya92})
\begin{equation}\lb{A.3}
\xi\big(\,\cdot\,; S,S_0)\in L^1\left(\bbR;(1+\lambda^2)^{-1}\, d\lambda\right),
\end{equation}
called the {\it spectral shift function} for $S$ and $S_0$, such that
\begin{equation}\lb{A.4}
\xi\big(\lambda;S,S_0\big)=0,\quad \lambda<\min\big[\sigma(S)\cup\sigma\big(S_0\big)\big],
\end{equation}
and for which the following trace formula holds: if $f\in \mathfrak{F}(\bbR)$, then
\begin{equation}\lb{A.5a}
[f(S)-f(S_0)] \in \cB_1(\cH),
\end{equation}
and
\begin{equation}\lb{A.5}
\tr_{\cH}\big(f(S) - f(S_0)\big)= \int_{\bbR} f'(\lambda)\, \xi\big(\lambda; S,S_0)\, d\lambda.
\end{equation}
One infers that $(\,\cdot\, - z)^{-1}\in \mathfrak{F}(\bbR)$, $z\in \bbC\backslash\bbR$, so that
\begin{equation}\lb{A.6}
\tr_{\cH}\Big(\big(S-zI_{\cH}\big)^{-1} - \big(S_0-zI_{\cH}\big)^{-1}\Big)= -\int_{\bbR} \frac{\xi(\lambda; S,S_0)}{(\lambda - z)^2}\, d\lambda,\quad z\in \bbC\backslash\bbR.
\end{equation}

Moreover, the spectral shift function may also be used to detect the presence of isolated eigenvalues of $S$ or $S_0$ (cf., e.g., the discussion following \cite[Theorem 8.7.2]{Ya92}).

\begin{lemma}\lb{lA.3}
On component intervals of $\rho(S)\cap\rho(S_0)$ in $\bbR$, the spectral shift function assumes constant values.  If $\lambda_1$ is an isolated eigenvalue of multiplicity $k_0$ of the operator $S_0$ and $k$ of the operator $S$, then
\begin{equation}\lb{A.7}
\xi(\lambda_1+0;S,S_0)-\xi(\lambda_1-0;S,S_0) = k_0 - k,
\end{equation}
where
\begin{equation}
\xi(\lambda_1\pm0;S,S_0) := \lim_{\lambda\to \lambda_1^{\pm}}\xi(\lambda;S,S_0).
\end{equation}
\end{lemma}

\medskip

\noindent {\bf Acknowledgments.}
The authors would like to thank Fritz Gesztesy for helpful correspondence in connection with the trace formula \eqref{5.95} and for pointing out the reference \cite{GS96}.  The research of the authors was supported by the National Science Foundation Grant DMS-1852288.

 
\end{document}